\documentclass[11pt,reqno]{amsart}
\usepackage[utf8]{inputenc}

\usepackage[margin=1in]{geometry}
\usepackage{amsmath}
\usepackage{amsfonts}
\usepackage{amssymb}
\usepackage{amsthm}
\usepackage{enumitem}
\usepackage{mathrsfs}
\usepackage{mathtools}
\usepackage{tikz-cd}
\usepackage{enumitem}

\usepackage[colorlinks=true,hyperindex, linkcolor=magenta, pagebackref=false, citecolor=cyan,pdfpagelabels]{hyperref}

\newtheorem{theorem}{Theorem}[section]
\newtheorem{prop}[theorem]{Proposition}
\newtheorem{lemma}[theorem]{Lemma}
\newtheorem{cor}[theorem]{Corollary}

\theoremstyle{definition}
\newtheorem{definition}[theorem]{Definition}
\newtheorem{example}[theorem]{Example}
\newtheorem{remark}[theorem]{Remark}

\renewcommand{\rm}{\mathrm}

\newcommand{\Spec}{\mathrm{Spec} \,}
\newcommand{\al}{\alpha}

\newcommand{\eps}{\epsilon}

\newcommand{\vp}{\varphi}

\newcommand{\bG}{{\mathbf G}}

\newcommand{\bQ}{{\mathbf Q}}
\newcommand{\bR}{{\mathbf R}}

\newcommand{\bZ}{{\mathbf Z}}

\newcommand{\cG}{\mathcal{G}}

\newcommand{\sD}{{\mathscr D}}

\newcommand{\sH}{{\mathscr H}}

\newcommand{\sN}{{\mathscr N}}
\newcommand{\sO}{{\mathscr O}}

\newcommand{\sT}{{\mathscr T}}
\newcommand{\sU}{{\mathscr U}}

\newcommand{\fg}{\mathfrak{g}}

\newcommand{\ft}{\mathfrak{t}}
\newcommand{\fu}{\mathfrak{u}}

\newcommand{\fz}{\mathfrak{z}}

\DeclareMathOperator{\ad}{ad}
\DeclareMathOperator{\Ad}{Ad}

\DeclareMathOperator{\Aut}{Aut}

\DeclareMathOperator{\chara}{char}

\DeclareMathOperator{\gl}{\mathfrak{g}\mathfrak{l}}
\DeclareMathOperator{\GL}{GL}

\DeclareMathOperator{\id}{id}

\DeclareMathOperator{\Lie}{Lie}

\DeclareMathOperator{\PGL}{PGL}
\DeclareMathOperator{\red}{red}
\DeclareMathOperator{\rk}{rk}
\DeclareMathOperator{\SL}{SL}
\DeclareMathOperator{\SO}{SO}
\DeclareMathOperator{\Sp}{Sp}
\DeclareMathOperator{\Spin}{Spin}
\DeclareMathOperator{\Stab}{Stab}

\DeclareMathOperator{\Sym}{Sym}

\DeclareMathOperator{\var}{var}

\newcommand{\wtil}[1]{\wtil}

\newcommand{\ov}[1]{\overline{#1}}

\title{Springer isomorphisms over a general base scheme}
\author{Sean Cotner}

\begin{document}
\bibliographystyle{halpha-abbrv}

\begin{abstract}
We establish the existence of Springer isomorphisms for reductive group schemes over general base schemes. For this, we first study centralizers of fiberwise regular sections of reductive group schemes, and we establish their flatness in many cases. The hypotheses in our results are essentially optimal. Our results clarify some aspects of Springer isomorphisms even over a field, and the arguments simplify considerably in this case.
\end{abstract}

\maketitle

\setcounter{tocdepth}{1}
\tableofcontents

\section{Introduction}

\subsection{Overview}

Let $G$ be a connected reductive group over a field $k$. Attached to $G$ are the \textit{unipotent variety} $\sU_G^{\var}$ and the \textit{nilpotent variety} $\sN_G^{\var}$, geometrically integral $k$-schemes parameterizing the unipotent elements of $G$ and the nilpotent elements of $\Lie G$, respectively.\footnote{This cumbersome notation is meant to distinguish these varieties from the unipotent and nilpotent schemes defined later, which can fail to be reduced. It would perhaps be more precise to write, e.g., $\sU_{G,\red}$, but this appears to be even more cumbersome.} If $\chara k$ is either $0$ or ``large", then the matrix logarithm (relative to some embedding of $G$ in $\GL_n$) is well-defined on unipotent matrices and defines a $G$-equivariant isomorphism $\sU_G^{\var} \to \sN_G^{\var}$. Such an isomorphism is useful for many purposes, and it is desirable to establish its existence under the mildest possible hypotheses on $G$ and $k$. \smallskip

In \cite{Springer-isomorphism}, Springer showed that as long as $\chara k$ is ``good" for $G$ (see Section~\ref{subsection:good-primes}) and the derived group $\sD(G)$ is split and simply connected, then there is a (non-canonical) $G$-equivariant birational universal homeomorphism of $k$-varieties $\rho: \sU_G^{\var} \to \sN_G^{\var}$. It follows from \cite{Demazure} that $\sN_G^{\var}$ is normal under these hypotheses, and hence $\rho$ is actually an isomorphism; it is now called a \textit{Springer isomorphism}. A different proof of the same result was later given in \cite{Bardsley-Richardson}, under some mild extra hypotheses in type $\rm{A}$. Since Springer's original paper, a number of refinements have been established; see \cite{Mcninch-optimal}, \cite{McNinch-Testerman-springer}, \cite{Sobaje}, and \cite{Jay-GGGR} for example. \smallskip


Springer claimed that a Springer isomorphism exists even if $G$ is not split. However, his proof appears to have a gap when the quasi-split inner form of $G$ is not split. Specifically, to establish the existence of a Springer isomorphism in the quasi-split case he uses crucially the claim that if $\rho: \sU_G^{\var} \to \sN_G^{\var}$ is a Springer isomorphism, then also $\Ad(g) \circ \rho$ is a Springer isomorphism for any $g \in G(k)$. Despite being false for all non-central $g$ this claim is used several times in \cite{Springer-isomorphism}.\smallskip

For groups of type $\mathrm{A}$, \cite{Springer-isomorphism} has a separate argument which uses the claim that the isomorphism $\sU_{\SL_n}^{\var} \to \sN_{\SL_n}^{\var}$ given by $u \mapsto u - 1$ is $\Aut_{\SL_n/k}$-equivariant. In fact, this claim is also false in all cases except $n = 2$, and there is \textit{never} such an $\Aut_{\SL_n/k}$-equivariant isomorphism when $\chara k = 2$ and $n \geq 3$. Some later sources (like \cite{Humphreys}) have exercised caution around the non-split case, but these errors do not seem to have been pointed out in the literature.\smallskip

Finally, for some integral questions it is of interest to know whether Springer isomorphisms exist for reductive group schemes over a general base. Thus we are left with the following questions.
\begin{enumerate}[label=(\alph*)]
    \item\label{item:intro-1} What if $G$ is not split?
    \item\label{item:intro-3} What if $G$ lives over a more general base scheme (e.g., $\Spec \bZ[1/N]$)?
    \item\label{item:intro-2} What if $\chara k$ is good but $|\pi_1(\sD(G))|$ is zero in $k$?
\end{enumerate}

The following theorem resolves \ref{item:intro-1} and \ref{item:intro-3}; question \ref{item:intro-2} is handled in \cite[Theorem 1.5(2)]{Cotner-non-etale}.

\begin{theorem}[Theorem~\ref{theorem:relative-springer-isomorphism}]\label{theorem:intro-springer-iso}
Let $S$ be an affine scheme and let $G$ be a reductive $S$-group scheme. Let $\sU_G$ and $\sN_G$ be the unipotent and nilpotent schemes of $G$, respectively, defined in Sections~\ref{subsection:unip-sch} and \ref{subsection:nilp-sch}. Then there is a $G$-equivariant $S$-isomorphism $\rho: \sU_G \to \sN_G$ provided the following two conditions hold.
\begin{enumerate}
    \item $|\pi_1(\sD(G))|$ is invertible on $S$,
    \item for each $s \in S$, $\chara k(s)$ is good for $G_s$.
\end{enumerate}
\end{theorem}

Affineness of $S$ is almost never necessary in Theorem~\ref{theorem:intro-springer-iso}; in some sense, the only problem comes from outer forms of $\mathrm{A}_n$ ($n \geq 2$) in characteristic $2$ and triality forms of $\mathrm{D}_4$ in characteristic $3$. See Theorem~\ref{theorem:relative-springer-isomorphism} for a precise statement, and see Remarks~\ref{remark:affine-necessary-type-a} and \ref{remark:affine-necessary-type-d4} for more discussion.\smallskip

The proof of the existence claim in Theorem~\ref{theorem:intro-springer-iso} follows the overall strategy of Springer's original proof fairly closely in the split case, and we describe this strategy in Section~\ref{subsection:outline}. In the split case, one of the biggest difficulties comes from (defining and) proving properties of the unipotent and nilpotent schemes for split reductive group schemes over $\Spec \bZ$, a task we take up in Section~\ref{section:unip-nilp}. To understand the regular loci of the unipotent and nilpotent schemes over an arbitrary base scheme, we prove the following flatness result. For the definition of ``strongly regular", see Section~\ref{subsection:regular-elements}.

\begin{theorem}[Theorem~\ref{theorem:flat-centralizer}]\label{theorem:intro-flat-regular-centralizer}
Let $S$ be a scheme and let $G \to S$ be a reductive group scheme. Suppose that $|\pi_1(\sD(G))|$ is invertible on $S$.
\begin{enumerate}
    \item\label{item:intro-flat-regular-centralizer-group} If $g \in G(S)$ is a fiberwise strongly regular section of $G$, then $Z_G(g)$ is flat and $Z_G(g)/Z(G)$ is smooth.
    \item\label{item:intro-flat-regular-centralizer-lie-algebra} Suppose that for every $s \in S$, the characteristic of $k(s)$ is good for $G_s$. If $X \in \mathfrak{g}(S)$ is a fiberwise regular section of $\mathfrak{g}$, then $Z_G(X)$ is flat and $Z_G(X)/Z(G)$ is smooth.
\end{enumerate}
\end{theorem}

The flatness assertion in Theorem~\ref{theorem:intro-flat-regular-centralizer}(\ref{item:intro-flat-regular-centralizer-lie-algebra}) is a special case of \cite[Thm.\ 4.2.8]{Bouthier-Cesnavicius}, while the smoothness assertion is a slight generalization of \cite[Prop.\ 4.2.11]{Bouthier-Cesnavicius}. Our proof strategy is completely different; roughly speaking, the difficulty comes almost entirely from understanding centralizers over a field, a task we take up in Sections~\ref{subsection:centralizers-of-group-elements} and \ref{subsection:centralizers-of-lie-algebra-elements}. The hypotheses in Theorem~\ref{theorem:intro-flat-regular-centralizer} are essentially optimal, as we show in \cite[Corollary 5.4, Corollary 5.8]{Cotner-non-etale}. \smallskip

While flatness is a basic technical desideratum, it does not preclude the centralizers in Theorem~\ref{theorem:intro-flat-regular-centralizer} from being ``weird", as the following example illustrates.

\begin{example}
Theorem~\ref{theorem:intro-flat-regular-centralizer} applies in particular to semisimple deformations of regular unipotent elements. For instance, let $G = \SL_2$ over $\bZ_p$ and let $g \in G(\bZ_p)$ be the section given by
\[
g = \begin{pmatrix} 1 + p & 1 \\ 0 & (1 + p)^{-1} \end{pmatrix}.
\]
The special fiber $g_s$ of $g$ is then a regular unipotent element, while the generic fiber $g_\eta$ is a regular semisimple element. We have
\[
Z_{G_s}(g_s) \cong \mu_2 \times \bG_a
\]
and
\[
Z_{G_\eta}(g_\eta) \cong \bG_m,
\]
so Theorem~\ref{theorem:intro-flat-regular-centralizer} asserts in this case that $Z_G(g)$ is a flat deformation of $\mu_2 \times \bG_a$ to $\bG_m$ and $Z_G(g)/Z(G)$ is a smooth deformation of $\bG_a$ to $\bG_m$. \smallskip

There is a similar element $g \in \SL_3(\bZ_2)$ such that $Z_{\SL_3}(g)$ has special fiber $\mu_3 \times \rm{W}_2$ and generic fiber $\bG_m^2$, where $\rm{W}_2$ is the scheme of length $2$ Witt vectors.
\end{example}

The arguments in this paper simplify considerably if the reader is interested primarily in the case $S = \Spec k$ for a field $k$, and we advise such a reader to ignore Section~\ref{subsection:regular-centralizers}, Theorem~\ref{theorem:unipotent-scheme}, and Theorem~\ref{theorem:nilpotent-scheme}, and to simply assume $S = \Spec k$ in the statements and proofs of the results in Section~\ref{section:springer-iso} and Appendix~\ref{appendix}. Still, the relative perspective is useful even over a field; we conclude this overview with the following consequence of Theorem~\ref{theorem:intro-flat-regular-centralizer}(\ref{item:intro-flat-regular-centralizer-group}).

\begin{cor}[Corollary~\ref{corollary:commutative-centralizer}]\label{corollary:intro-cor}
Let $G$ be a connected reductive group over a field $k$ of characteristic $p \geq 0$. If $p \nmid |\pi_1(\sD(G))|$ and $g \in G(k)$ is strongly regular, then $Z_G(g)$ is commutative.
\end{cor}

Corollary~\ref{corollary:intro-cor} has been known for a long time on the level of geometric points; see \cite{Springer-note} and \cite{Lou}. However, as far as we know the only prior proof of this fact involves rather extensive case-by-case checking using the classification (involving long computer calculations in bad characteristic) to control the component group of $Z_G(g)$. Our proof is entirely conceptual. Roughly speaking, we use Theorem~\ref{theorem:intro-flat-regular-centralizer} to reduce to the case that $g$ is \textit{semisimple}, in which case $Z_G(g)$ is a torus and thus obviously commutative. We note that the result is not true if $|\pi_1(\sD(G))|$ is divisible by $p$ and $Z_G(g)$ is interpreted as a scheme; this is expounded upon in \cite[Theorem 5.1]{Cotner-non-etale}. \smallskip

We conclude this section by remarking that, as pointed out to us by Yakov Varshavsky, \cite[Lemma 1.8.12]{Kazhdan-Varshavsky} proves a result similar to Theorem~\ref{theorem:relative-springer-isomorphism} over the ring of integers of a local field under some (mild) extra hypotheses on $G$. We will discuss this briefly in Section~\ref{ss:quasi-log}.

\subsection{Outline of the proof}\label{subsection:outline}

We give a brief outline of the proof of the existence part of Theorem~\ref{theorem:intro-springer-iso} for the convenience of the reader.\smallskip

First assume that $G$ is split and $S = \Spec \bZ[1/N]$ for some $N$ which is divisible by $|\pi_1(\sD(G))|$ and the bad primes for $G$. In this case, Theorems~\ref{theorem:unipotent-scheme} and \ref{theorem:nilpotent-scheme} show that $\sU_G$ and $\sN_G$ are normal schemes. Moreover, they admit $S$-smooth fiberwise dense open subschemes $\sU_{\rm{reg}}$ and $\sN_{\rm{reg}}$ of complementary codimension $2$ on fibers. To show the existence of $\rho$, it suffices therefore to show that there is a $G$-equivariant $S$-isomorphism $\rho: \sU_{\rm{reg}} \to \sN_{\rm{reg}}$.\smallskip

Calculations of Springer show that there exist sections $u \in \sU_{\rm{reg}}(S)$ and $X \in \sN_{\rm{reg}}(S)$ such that $Z_G(u) = Z_G(X)$, and using the flatness assertion in Theorem~\ref{theorem:intro-flat-regular-centralizer} we show in Lemma~\ref{lemma:regular-orbit-spaces} that there are natural $G$-equivariant $S$-isomorphisms $\sU_{\rm{reg}} \cong G/Z_G(u)$ and $\sN_{\rm{reg}} \cong G/Z_G(X)$. From this follows the existence of a \textit{unique} $\rho$ as above satisfying $\rho(u) = X$. In fact, exercising some care and using results from \cite{ALRR}, we can find $u$ and $X$ in almost all cases such that $\rho$ is $\Aut_{G/S}$-equivariant. \smallskip

This work having been done, the general case follows from a simple twisting argument, explained at the end of Section~\ref{subsection:general}, except for type $\mathrm{A}_n$ ($n \geq 2$) when $S$ has points of characteristic $2$ and type $\mathrm{D}_4$ when $S$ has points of characteristic $3$; these somewhat complicated cases are dealt with separately in Appendix~\ref{appendix}.

\subsection{Notation and conventions}

In this section we collect a few conventions which we will use without comment below.\smallskip

If $M$ is a finite locally free module over a ring $A$, then we will abuse notation to identify $M$ with the associated affine space $\Spec(\Sym_A M^*)$.\smallskip

If $A$ is a(n abstract) group with identity element $e$ then we say ``$A$ has $n$-torsion" for some $n \in \bZ_{> 0}$ to mean that there is an element $a \in A$ such that $a \neq e$ and $a^n = e$. A similar convention is applied to the phrase ``$A$ has no $n$-torsion."\smallskip

We refer to discrete valuation rings as DVRs.\smallskip

If $S$ is a scheme and $s \in S$, then $\overline{s}$ always denotes a geometric point lying over $s$. If $S$ is assumed to be local, then we always denote by $s$ the closed point of $S$; if $S$ is assumed to be irreducible, then we always denote by $\eta$ the generic point of $S$.\smallskip

If $S$ is an affine scheme and $X = \Spec A$ is an affine $S$-scheme equipped with an action of an $S$-group scheme $G$, then we denote by $X/\!/G$ the GIT quotient $\Spec A^G$ of $X$ by $G$.\smallskip

If $G$ is a semisimple group scheme over a base scheme $S$, then the \textit{universal cover} $\widetilde{G}$ is the unique semisimple simply connected $S$-group scheme equipped with a central isogeny $\pi_G: \widetilde{G} \to G$, whose existence follows from \cite[Exp.\ XXI, Cor.\ 6.5.10; Exp.\ XXV, Thm.\ 1.1]{SGA3III}. We define $\pi_1(G) = \ker(\pi_G)$, and we regard $|\pi_1(\sD(G))|$ as a locally constant function $|S| \to \bZ$. Note that this disagrees with another common convention, as in \cite[1.1]{Borovoi}. \smallskip

We use the word ``embedding" to refer to immersions of schemes (in the terminology of \cite{EGA}).\smallskip

Since it is the original source for such things, we will usually cite SGA3 for results on reductive group schemes, but we occasionally prefer to cite \cite{Conrad} instead, especially for results whose proofs are simplified by methods which appeared after the publication of SGA3.

\subsection{Acknowledgements}

I thank Jeremy Booher, Spencer Dembner, and Vaughan McDonald for helpful conversations. I thank Simon Riche for his early interest in and encouragement of this project. I thank Jay Taylor for a question which led to Proposition~\ref{prop:kawanaka}. I thank Ravi Vakil for helpful and encouraging conversations, and for his suggestion to make this a standalone paper. I thank Yakov Varshavsky for pointing out his paper \cite{Kazhdan-Varshavsky}. I thank an anonymous referee for helpful comments. Finally, I thank my advisor, Brian Conrad, for very extensive edits and helpful suggestions on previous versions of this paper.

\section{Preliminaries}\label{section:preliminaries}

In this section we record some definitions and results which we will use. Almost everything below is standard, but some things often appear in the literature without schemes. In Section~\ref{subsection:regular-elements}, we will discuss and compare two standard definitions of regularity from the literature.

\subsection{Schematic stabilizers}\label{subsection:schematic-stabilizers}

Suppose $S$ is a scheme, $G$ is an $S$-group scheme, $X$ is an $S$-scheme, and there is an $S$-action of $G$ on $X$, i.e., an $S$-morphism $G \times X \to X$ which is a group action on $S'$-points for any $S$-scheme $S'$. If $Y$ is a closed subscheme of $X$ then we define the functorial stabilizer $\underline{\Stab}_G(Y)$ functorially by
\[
\underline{\Stab}_G(Y)(S') = \{g \in G(S'): g \text{ acts trivially on } Y_{S'}\}.
\]
If $X = G$ and the action $G \times X \to X$ is given by conjugation, then $\underline{\Stab}_G(Y)$ is the functorial centralizer defined in \cite[Def.\ 2.2.1]{Conrad}, and the following lemma is contained in \cite[Lem.\ 2.2.4]{Conrad}.

\begin{lemma}\label{lemma:representability-of-stabilizer}
Suppose $X$ is finitely presented and $S$-affine and $Y$ is finite flat and finitely presented over $S$. If $G$ is finitely presented, then $\underline{\Stab}_G(Y)$ is represented by a finitely presented closed subgroup of $G$.
\end{lemma}

\begin{proof}
This follows from \cite[Exp.\ VI\textsubscript{B}, Ex.\ 6.2.4(a)]{SGA3I}.
\end{proof}

\begin{example}
There are two special cases of Lemma~\ref{lemma:representability-of-stabilizer} which will concern us, both applied to the case that $G$ is smooth and $S$-affine.
\begin{enumerate}
    \item Let $X = G$, let $G \times X \to X$ be the conjugation action, and let $Y = \{g\}$ for some section $g \in X(S)$. Then $\underline{\Stab}_G(Y)$ is called the centralizer of $g$ in $G$ and denoted by $Z_G(g)$.
    \item Let $\fg$ be the Lie algebra of $G$. If $X = \fg$, $G \times X \to X$ is the adjoint action, and $Y = \{x\}$ for some section $x \in X(S)$, then $\underline{\Stab}_G(Y)$ is called (perhaps abusively) the centralizer of $x$ in $G$ and denoted by $Z_G(x)$. We remark that if $S = \Spec k$ for an algebraically closed field $k$, then this is not the same as the centralizer appearing in \cite[2.1]{Jantzen-nilpotent}; that centralizer is the underlying reduced scheme.
\end{enumerate}
\end{example}

If $G$ is smooth and $S$-affine for some affine base scheme $S = \Spec A$, and if $\fg = \Lie G$, $X \in \fg$, then $\fz_\fg(X)$ denotes the Lie algebra centralizer of $X$. This we will regard simply as a Lie algebra over $A$, so we don't need Lemma~\ref{lemma:representability-of-stabilizer} to obtain existence of this object.

\begin{lemma}\label{lemma:lie-algebra-of-centralizer}
Let $G$ be a smooth $S$-affine group scheme over an affine base scheme $S = \Spec A$. If $\fg = \Lie G$, $X \in \fg$, and $g \in G(A)$ then $\Lie Z_G(X) = \fz_\fg(X)$ and $\Lie Z_G(g) = \fg^{\Ad(g)}$.
\end{lemma}

\begin{proof}
Define a morphism $\phi: (G, 1) \to (\fg, 0)$ via $g \mapsto \Ad(g)X - X$. Consider the following diagram with rows which are equalizer sequences.
\[
\begin{tikzcd}
Z_G(X)(A[\varepsilon]/(\varepsilon^2))) \arrow[r] \arrow[d]
    &G(A[\varepsilon]/(\varepsilon^2)) \arrow[r, shift left, "\phi"] \arrow[r, shift right, "0", labels=below] \arrow[d]
    &\fg \otimes_A A[\varepsilon]/(\varepsilon^2) \arrow[d] \\
Z_G(X)(A) \arrow[r]
    &G(A) \arrow[r, shift left, "\phi"] \arrow[r, shift right, "0", labels=below]
    &\fg
\end{tikzcd}
\]
Passing to ``vertical" fibers over the respective identity elements, a diagram chase shows that this induces an equalizer sequence
\[
\begin{tikzcd}
\Lie Z_G(X) \arrow[r]
    &\fg \arrow[r, shift left, "d\phi"] \arrow[r, shift right, "0", labels=below]
    &\fg
\end{tikzcd}
\]
The differential of $\Ad: G \to \GL(\fg)$ is $\ad: \fg \to \gl(\fg)$, so $d\phi$ is given by $(d\phi)(Y) = [Y, X]$, and we conclude $\Lie Z_G(X) = \fz_\fg(X)$. The argument for $\Lie Z_G(g)$ is completely similar, using that the differential at $1$ of the morphism $G \to G$, defined by $x \mapsto gxg^{-1}$ is $Y \mapsto \Ad(g)Y$.
\end{proof}

\subsection{Reductive group schemes}

We recall some basic notions from the theory of reductive group schemes; see \cite{Conrad} or \cite{SGA3III} for more comprehensive treatments. We recall first that a smooth affine group scheme $G$ over a field $k$ is called \textit{reductive} if its geometric unipotent radical is trivial; i.e., if $G_{\overline{k}}$ contains no nontrivial smooth connected normal unipotent subgroup. Since the unipotent radical of $G_{\overline{k}}$ is a characteristic subgroup scheme of $G_{\overline{k}}^0$, this condition is equivalent to the condition that $G^0$ is connected reductive. A reductive $k$-group $G$ is called \textit{semisimple} if moreover its center $Z(G)$ is finite.

\begin{definition}
Let $S$ be a scheme. An $S$-group scheme $G$ is \textit{reductive} if it is smooth, $S$-affine, and $G_{\overline{s}}$ is connected reductive for every geometric point $\overline{s}$ of $S$. Similarly, $G$ is \textit{semisimple} if it is reductive and $G_{\overline{s}}$ is semisimple for every geometric point $\overline{s}$ of $S$.
\end{definition}

We note the slight discrepancy in definitions: over a field, we have not required that $G$ be connected, but over a general base scheme we have required that $G$ have connected fibers. For the remainder of this paper, any result stated exclusively over a field will use the term in the first more general sense, and any result stated over a more general base will use the term in the second more restricted sense. \smallskip

Every reductive group scheme $G \to S$ admits a split fiberwise maximal torus $T$ over some etale cover of $S$ \cite[Lem.\ 5.1.3]{Conrad}, and (locally on $S$) the pair $(G, T)$ has a root system $\Phi$ coming from a decomposition of the Lie algebra $\fg$ of $G$ into weight spaces for $T$ under the adjoint action. As in the classical setting, the functorial center $Z(G)$ is a subgroup scheme of $G$ of multiplicative type over $S$, and it is the kernel of the adjoint action $\Ad_G: G \to \GL(\fg)$ \cite[Prop.\ 3.3.8]{Conrad}. Many other classical results and constructions concerning reductive groups extend to this more general setting: the theory of Borel and parabolic subgroups \cite[\S 5.2]{Conrad}; the theory of central isogenies, universal covers, and adjoint quotients \cite[\S 6]{Conrad}; the decomposition of simply connected and adjoint type groups into simple components \cite[Thm.\ 5.1.19]{Conrad}; and the theory of the derived group \cite[Thm.\ 5.3.1]{Conrad}. We will use all of these notions frequently without comment. \smallskip

See \cite[Lem.\ 2.14, Thm.\ 2.15]{SteinbergTorsion} for the indication of a proof of the following theorem.

\begin{theorem}\label{theorem:steinberg-connectedness}
Let $G$ be a reductive group over a field $k$, and let $g \in G(k)$ be semisimple. Then
\begin{enumerate}
    \item $Z_G(g)$ is a reductive group,
    \item If $G$ is connected and $\sD(G)$ is simply connected, then $Z_G(g)$ is connected.
\end{enumerate}
\end{theorem}

\subsection{Jordan decomposition}

The following lemma is useful, in light of Theorem~\ref{theorem:steinberg-connectedness}, for reducing questions about centralizers of elements of connected reductive groups to questions about centralizers of unipotent elements.

\begin{lemma}\label{lemma:jordan-decomposition-centralizer}
Let $k$ be a perfect field and let $G$ be a linear algebraic group over $k$. If $g \in G(k)$ has Jordan decomposition $g = tu$, then as group schemes
\[
Z_G(g) = Z_{Z_G(t)}(u).
\]
If $\mathfrak{g}$ is the Lie algebra of $G$ and $X \in \mathfrak{g}$ has Jordan decomposition $X = X_{\mathrm{ss}} + X_{\mathrm{n}}$, then as group schemes
\[
Z_G(X) = Z_{Z_G(X_{\mathrm{ss}})}(X_{\mathrm{n}}).
\]
\end{lemma}

\begin{proof}
This follows from \cite[3.3]{Slodowy}.
\end{proof}

\subsection{Good primes}\label{subsection:good-primes}

We will give a brief summary of the notions of good and very good primes, following \cite{Springer-Steinberg}.

\begin{definition}
Given a reduced root system $\Phi$, we say that $p$ is a \textit{bad prime} for $\Phi$ if there is a closed subsystem $\Sigma \subset \Phi$ such that $\bZ\Phi/\bZ\Sigma$ has $p$-torsion. Similarly, $p$ is a \textit{torsion prime} if there is some such $\Sigma$ so that $\bZ\Phi^\vee/\bZ\Sigma^\vee$ has $p$-torsion. If $p$ is not a bad prime, then it is a \textit{good prime}.
\end{definition}

The following definition is non-standard, but we find it clarifying.

\begin{definition}
If $V = \bZ\Phi \otimes_Z \bQ$, set $P = (\bZ\Phi^\vee)^* \subset V$. We say that $p$ is a \textit{singular prime} if $p$ divides $|P/\bZ\Phi|$. If $p$ is not a singular prime, then it is a \textit{smooth prime}.
\end{definition}

Finally, we say that $p$ is a \textit{very good prime} if $p$ is both a good prime and a smooth prime. If $G$ is a split reductive group over a field $k$ of characteristic $p > 0$, then we will say that that $p$ is a bad (resp.\ torsion, resp.\ singular, resp.\ good, resp.\ smooth, resp.\ very good) prime for $G$ provided that it is the same for the root system $\Phi$. In general, we will say that $\chara k$ is \textit{good} for $G$ provided that either $\chara k = 0$ or $\chara k$ is a good prime for $G$. We remark that if $G$ is connected, semisimple and simply connected, then the quotient $P/\bZ\Phi$ above is Cartier dual to the center $Z(G)$. On the other hand, if $G$ is connected, semisimple, and of adjoint type, then $P/\bZ\Phi$ is Cartier dual to the fundamental group of $G$. In general, if $G$ is connected and semisimple, then $p$ is a smooth prime if and only if the map $\pi_G: \widetilde{G} \to G$ from the universal cover is smooth and the center $Z(G)$ of $G$ is smooth, whence the name. \smallskip

In \cite[I, 4.3, 4.4]{Springer-Steinberg}, the bad primes and torsion primes are determined for all irreducible reduced $\Phi$. In particular, every torsion prime is a bad prime, and all primes greater than $5$ are good for every $\Phi$. The singular primes may also be determined from the tables in \cite[Chap.\ VI, \S 4]{Bourbaki}. Table~\ref{table:1} shows the bad primes, torsion primes, and singular primes for the various irreducible root systems. Notice in particular that every torsion prime is a bad prime, and a prime $p$ is very good when it is good and there is no component of type $\rm{A}_n$ for $p \mid n+1$.

\begin{table}
\centering
\begin{tabular}{ |c|c|c|c| } 
\hline
Type & Bad primes & Torsion primes & Singular primes \\
\hline
$\rm{A}_n$ ($n \geq 1$) & none & none & $p \mid n + 1$ \\ 
$\rm{B}_n$ ($n \geq 2$) & $2$ & $2$ ($n \geq 3$) & $2$ \\ 
$\rm{C}_n$ ($n \geq 3$) & $2$ & none & $2$ \\ 
$\rm{D}_n$ ($n \geq 4$) & $2$ & $2$ & $2$ \\
$\rm{E}_6$ & $2, 3$ & $2, 3$ & $3$ \\
$\rm{E}_7$ & $2, 3$ & $2, 3$ & $2$ \\
$\rm{E}_8$ & $2, 3, 5$ & $2, 3, 5$ & none \\
$\rm{F}_4$ & $2, 3$ & $2, 3$ & none \\
$\rm{G}_2$ & $2, 3$ & $2$ & none \\
\hline
\end{tabular}
\caption{Bad, torsion, and singular primes for the various types}
\label{table:1}
\end{table}

\subsection{Weights and representations}

If $G$ is a split reductive group over a field $k$, and we are given a (split) maximal torus $T$ inside a Borel subgroup $B$, then there is an associated root system $\Phi = \Phi(G, T)$ with a system of simple roots $\Delta = \Delta(B, T) = \{\alpha_1, \dots, \alpha_r\}$. We let $\alpha_1^{\vee}, \dots, \alpha_r^{\vee}$ be the associated system of simple coroots. If (and only if) $G$ is semisimple and simply connected, the simple coroots form a basis for the cocharacter lattice $X_*(T)$, and we let $\omega_1, \dots, \omega_r$ be the dual basis for the character lattice $X(T)$, called the \textit{fundamental weights}. In any case, we let $X(T)_+$ denote the set of \textit{dominant characters}; i.e., the set of characters $\lambda \in X(T)$ such that $\langle \alpha_i^{\vee}, \lambda \rangle \geq 0$ for all $i$. By \cite[Chap.\ VI, \S 1, Thm.\ 2]{Bourbaki}, if $W = N_G(T)/T$ is the Weyl group of $(G, T)$, then every element of $X(T)$ is $W$-conjugate to a unique dominant character. If $G$ is semisimple and simply connected, then $X(T)_+$ is a free semigroup, generated by the fundamental weights $\omega_1, \dots, \omega_r$. Associated to each fundamental weight $\omega_i$ there is a representation $V_i$ of $G$, called the \textit{$i$th fundamental representation}, with highest weight $\omega_i$, see \cite[II, \S 2]{Jantzen}. We will maintain this notation below whenever given a triple $(G, B, T)$ as above.

\subsection{Regular elements}\label{subsection:regular-elements}

We briefly recall Steinberg's theory of regular elements of a connected reductive group $G$. Let $r$ denote the rank of $G$; i.e., the dimension of a maximal torus of $G$. If $g \in G(k)$, then $Z_G(g)$ is of dimension $\geq r$. Indeed, suppose $k$ is algebraically closed and $B$ is a Borel subgroup of $G$ containing $g$ with unipotent radical $U$. Then the orbit map $B/Z_B(g) \to B \to B/U$ under the conjugation action is constant since $B/U$ is commutative. So the image of the orbit map is contained in $gU$ and we obtain the dimension bound. We say that $g$ is \textit{regular} if equality holds, i.e., $Z_G(g)$ is of dimension $r$. If $g$ is semisimple, this implies that $Z_G(g)^0$ is a torus, but $g$ need not be semisimple in this definition: in general, if $g = tu$ is the Jordan decomposition of $g$, then $Z_G(t)^0$ is a connected reductive group of rank $r$, so that $u$ is a regular unipotent element of $Z_G(t)^0$. We summarize some results of Steinberg on regular unipotent elements in the following theorem.

\begin{theorem}\label{theorem:properties-of-regular-unipotents}
Let $G$ be a connected reductive group over an algebraically closed field $k$.
\begin{enumerate}
    \item \cite[Thm.\ 3.1]{SteinbergReg} Regular unipotent elements exist in $G(k)$, and any two are $G(k)$-conjugate.
    \item \cite[Lem.\ 3.2, Thm.\ 3.3]{SteinbergReg} If $u \in G(k)$ is regular unipotent, then there is a unique Borel subgroup $B$ of $G$ containing $u$.
    \item\label{item:lie-algebra-dimension-bound} \cite[\S 4]{SteinbergReg} If $\sD(G)$ is simply connected and $u \in G(k)$ is regular unipotent, then the fixed point algebra $\fg^{\Ad(u)}$ has dimension $\leq \dim \fz + r_{\rm{ss}}$, where $\fz$ is the center of the Lie algebra $\fg = \Lie G$ and $r_{\rm{ss}}$ is the semisimple rank of $G$.
\end{enumerate}
\end{theorem}

\begin{proof}
All points are proved in \cite{SteinbergReg} under the further assumption that $G$ is semisimple. Since all unipotent elements of $G(k)$ lie in $\sD(G)(k)$, the first two points extend immediately to the more general reductive case. The third point extends to reductive $G$ in a less straightforward manner, and we will indicate how to perform this extension. Let $T$ be a maximal torus of $G$, and let $T' = T \cap \sD(G)$ be the maximal torus of $\sD(G)$ contained in $T$. Let $\alpha_1, \dots, \alpha_r$ be a system of simple roots for $T$ corresponding to a Borel $B$, with corresponding fundamental weights $\omega_1', \dots, \omega_r'$ for $T'$, and let $w = s_1 \cdots s_r$ denote the Coxeter element corresponding to this system of simple roots. By \cite[Theorem 3.3, Lemma 4.5]{SteinbergReg}, we may pass to a conjugate of $u$ to assume that $u \in B(k)wB(k)$. The only part of the argument in \cite[\S 4]{SteinbergReg} using semisimplicity (or simple connectedness!) of $G$ is step 4) in the proof of \cite[Lem.\ 4.3]{SteinbergReg}, where it is proved (in the semisimple case) that the kernel of $1 - \Ad(w)$ on $\Lie T$ is $\fz$ (certainly it contains $\fz$). We prove this below under our more general hypotheses, following the main strategy of \cite{SteinbergReg}. Recall that $\fz$ is the subalgebra of $\Lie T$ consisting of elements vanishing along each root differential $d\alpha_i$. \smallskip

Choose lifts $\omega_1, \dots, \omega_r \in X(T)$ of $\omega_1', \dots, \omega_r'$ arbitrarily. We note that $s_j \cdot \omega_i = \omega_i - \langle \omega_i, \alpha_j^\vee \rangle \alpha_j$, so by \cite[Chap.\ VI, \S 1.10]{Bourbaki} we have
\[
s_j \cdot \omega_i = \omega_i - \delta_{ij} \alpha_i
\]
by the corresponding relation for $\omega_i^\vee$ and the fact that $\alpha_j^\vee$ takes values in $T'$. Choose $\lambda_1, \dots, \lambda_s \in X(T/T')$ forming a basis of $X(T/T')$, so that $\omega_1, \dots, \omega_r, \lambda_1, \dots, \lambda_s$ forms a basis of $X(T)$. In particular, $d\omega_1, \dots, d\omega_r, d\lambda_1, \dots, d\lambda_s$ forms a basis of $\ft^*$. Suppose that $H \in \ft$ is an element such that $(1 - \Ad(w))H = 0$. Applying $\Ad(s_1)$, we find
\[
(1 - \Ad(s_1))H = (1 - \Ad(s_2 \cdots s_r))H.
\]
Applying $d\lambda_i$ to either side yields $0$ because $W$ acts trivially on $X(T/T')$. Similarly, applying $d\omega_2, \dots, d\omega_r$ to the left side yields $0$, and applying $d\omega_1$ to the right side yields $0$, both by the above displayed relations for $\omega_i$. Induction shows $(1 - \Ad(s_i))H = 0$ for all $i$. Applying $d\omega_i$, we find $(d\alpha_i)(H) = 0$ for all $i$. Thus $H \in \fz$, as desired.
\end{proof}

The following nonstandard definition will also be used later.

\begin{definition}\label{definition:strongly-regular}
We say that an element $g \in G(k)$ is \textit{strongly regular} if it is regular and in the Jordan decomposition $g_{\overline{k}} = tu$, the centralizer $Z_{G_{\overline{k}}}(t)$ is connected.
\end{definition}

Note that if $\sD(G)$ is simply connected, then regularity and strong regularity are equivalent conditions by Theorem~\ref{theorem:steinberg-connectedness}.\smallskip

There is also a notion of regular element for elements of the Lie algebra $\mathfrak{g}$ of $G$. By the same argument as before, if $X \in \mathfrak{g}$, then $\dim Z_G(X) \geq r$. If equality holds, we will say that $X$ is \textit{regular}. Unlike in the group case, regular semisimple elements need not exist: see Lemma~\ref{lemma:chevalley-regular}. On the other hand, \cite[Lem.\ 3.1.1]{Riche-universal-centralizer} shows that regular nilpotent elements always exist. If $X \in \fg$ is regular nilpotent, then $X$ is contained in the Lie algebra of a unique Borel subgroup of $G$ by \cite[Lem.\ 5.3]{Springer}. \smallskip

There is another useful definition of regularity for elements of a Lie algebra which is common in the literature (as in for instance \cite[Exps.\ XIII, XIV]{SGA3II}). This definition plays no role in the present paper, and we discuss it only for the sake of completeness. Given a Lie algebra $\mathfrak{g}$ over a ring $k$ and an element $X \in \mathfrak{g}$ there is a characteristic polynomial $P(X, t) = t^n + c_{n-1}(X)t^{n-1} + \dots + c_r(X)t^r$ for the adjoint action of $X$ on $\mathfrak{g}$. Here the $c_i$ are functorial in $k$, so they are given by elements of the symmetric algebra of $\mathfrak{g}^*$ (well-defined even if $k$ is finite), and we assume that $c_r$ is not the $0$ element. If $k$ is a field, then this $r$ is called the \textit{nilpotent rank} of $\mathfrak{g}$. We note that if $\fg$ is the Lie algebra of a connected semisimple group $G$ over a field $k$, then $r$ is at least the rank of $G$, as can be seen from the fact that for any $X \in \fg$ we have $\Lie Z_G(X) = \fz_{\fg}(X)$ (by Lemma~\ref{lemma:lie-algebra-of-centralizer}). For the rest of this section, $k$ is assumed to be a field.

\begin{definition}
We will say that an element $X \in \mathfrak{g}$ is \textit{Chevalley regular} if $c_r(X) \neq 0$.
\end{definition}

This is nonstandard terminology which we introduce only to distinguish it from our previous definition of regularity. See \cite[Exps.\ XIII, XIV]{SGA3II} for a comprehensive treatment of this notion of regularity, as well as some interesting applications to Zariski-local torus lifting in reductive group schemes. The following lemma relates the two notions. Recall from \cite[\S 4.1.1]{Bouthier-Cesnavicius} that $G$ is \textit{root-smooth} if for every maximal $\ov k$-torus $T \subset G_{\ov k}$, every root $\alpha\colon T \to \bG_m$ is smooth. As explained in \textit{loc.\ cit.}, a split simple $k$-group $G$ is root-smooth unless $\chara k = 2$ and $G \cong \Sp_{2n}$ for some $n \geq 1$.

\begin{lemma}\label{lemma:chevalley-regular}
Let $G$ be a semisimple group over a field $k$ with Lie algebra $\mathfrak{g}$. If $k$ is infinite, then Chevalley regular elements always exist in $\mathfrak{g}$. Moreover:
\begin{enumerate}
    \item\label{item:regular-semisimple-implies-chevalley} If $X \in \mathfrak{g}$ is regular and semisimple, then it is Chevalley regular.
    \item\label{item:chevalley-implies-regular-semisimple} If $G$ is root-smooth, then an element $X \in \mathfrak{g}$ is Chevalley regular if and only if it is regular and semisimple.
    \item\label{item:no-semisimple-regular-elements} If $G$ is not root-smooth, then there are no semisimple regular elements in $\mathfrak{g}$.
\end{enumerate}
\end{lemma}

\begin{proof}[Proof of Lemma~\ref{lemma:chevalley-regular}]
The claim about existence of Chevalley regular elements when $k$ is infinite is true for all Lie algebras, as is easy to see since the locus of $x \in \fg$ such that $c_r(x) \neq 0$ is open and dense in $\fg$. In fact, \cite[Exp.\ XIV, App.]{SGA3II} shows that Chevalley regular elements always exist when $G$ is semisimple of adjoint type, even if $k$ is finite. For the remainder of the proof, we may and do assume that $k$ is algebraically closed. \smallskip

For (\ref{item:regular-semisimple-implies-chevalley}), suppose that $X \in \mathfrak{g}$ is regular and semisimple. Then by Proposition~\ref{prop:semisimple-centralizer-lie-algebra}, $Z_G(X)^0$ is a connected reductive group of dimension equal to the rank of $G$. The proof of Proposition~\ref{prop:semisimple-centralizer-lie-algebra} also shows that $X$ lies in the Lie algebra of a maximal torus $T$ of $G$. Thus $T \subset Z_G(X)^0$, and dimension considerations combine with smoothness to show that $T = Z_G(X)^0$. Since $\Lie Z_G(X) = \fz_{\fg}(X)$, the nilpotent rank of $\fg$ is equal to the rank of $G$ and thus $X$ is Chevalley regular. \smallskip

Point (\ref{item:chevalley-implies-regular-semisimple}) is proven in \cite[\S 4.1.5, discussion preceding Lemma 4.1.6]{Bouthier-Cesnavicius}. Point (\ref{item:no-semisimple-regular-elements}) follows immediately from \cite[Exp.\ XIII, Prop.\ 4.6]{SGA3II}.
\end{proof}

\section{Centralizers of regular elements}\label{section:regular-centralizers}
In this section, we will prove Theorem~\ref{theorem:intro-flat-regular-centralizer}. To do so, we first need to understand centralizers of regular elements over a field, so in Section~\ref{subsection:centralizers-of-group-elements} our main aim is to establish Corollary~\ref{corollary:general-centralizer}. After this we use the case-checking in \cite{Springer} to establish smoothness results for centralizers of regular elements of the Lie algebra in Section~\ref{subsection:centralizers-of-lie-algebra-elements}. Finally, Section~\ref{subsection:regular-centralizers} combines these results to deduce Theorem~\ref{theorem:intro-flat-regular-centralizer}. \smallskip

Throughout Sections~\ref{subsection:centralizers-of-group-elements} and \ref{subsection:centralizers-of-lie-algebra-elements}, $G$ is a connected reductive group over a field $k$ of characteristic $p \geq 0$, $T$ is a maximal $k$-torus of $G$, and $\Phi = \Phi(G_{\overline{k}}, T_{\overline{k}})$ and $\Phi^\vee$ are the root system and the coroot system corresponding to this data, respectively. We recall that $\pi_1(\sD(G_{\overline{k}}))$ (defined as the kernel of the universal cover of $\sD(G_{\overline{k}})$) is Cartier dual to the constant $\overline{k}$-group $((X_*(T_{\overline{k}})/\bZ \Phi^\vee)_{\rm{tors}})^*$ (where the asterisk denotes the dual finite abelian group), and $Z(G_{\overline{k}})$ is Cartier dual to the constant $\overline{k}$-group $X(T_{\overline{k}})/\bZ \Phi$. Note that our definition of the fundamental group differs from \cite[1.1]{Borovoi}.

\subsection{Centralizers of group elements}\label{subsection:centralizers-of-group-elements}

\begin{theorem}[{\cite[Thm.\ 2.21]{SteinbergTorsion}}] \label{theorem:steinberg-torsion}
Suppose $t \in G(k)$ is semisimple, and suppose $n$ is a positive integer such that $t^n \in Z(G)(k)$. Suppose $t \in T$, and let $\Phi_t$ be the root system corresponding to the pair $(Z_{G_{\overline{k}}}(t)^0, T)$.
\begin{enumerate}
    \item\label{item:steinberg-torsion-1} If $\bZ \Phi^\vee/\bZ (\Phi_t)^\vee$ has $\ell$-torsion for a prime number $\ell$, then $\ell$ divides $n$.
    \item\label{item:steinberg-torsion-2} If $\bZ \Phi/\bZ \Phi_t$ has $\ell$-torsion for a prime number $\ell$, then $\ell$ divides $n$.
    \item\label{item:steinberg-torsion-3} If $G$ is absolutely simple, then $(\bZ \Phi^\vee/\bZ \Phi_t^\vee)_{\rm{tors}}$ and $(\bZ \Phi/\bZ\Phi_t)_{\rm{tors}}$ are cyclic groups whose orders divide $n$.
    \item\label{item:steinberg-torsion-4} If $G$ is absolutely simple, $(\bZ \Phi/\bZ\Phi_t)_{\rm{tors}}$ has order $m$, and $\alpha \in \bZ \Phi$ maps to a generator, then the value $\alpha(t) \in k^\times$ has (exact) order $m$.
\end{enumerate}
\end{theorem}

Only the first point in this theorem is explicitly stated in \cite{SteinbergTorsion}, but the proof establishes the other points as well. This is an important result for us, so we will describe the proof in some detail, following closely the proof of \cite[Thm.\ 2.21]{SteinbergTorsion}.

\begin{proof}
We may evidently assume that $k$ is algebraically closed, and we will do so. By \cite[Lem.\ 2.9(a)]{SteinbergTorsion}, every closed subsystem of $\Phi$ is the root system of some reductive subgroup of $G$ of the same rank. Thus by passing from $\Phi$ to $\bQ\Phi_t \cap \Phi$, we may assume that $\Phi$ is a subset of $\bQ \Phi_t$. We will further reduce the proofs of the first two points to the case that $G$ is simple and simply connected. If $S$ is the maximal central torus in $G$, then the multiplication map $S \times \sD(G) \to G$ is a central isogeny, and there is an induced isomorphism of root systems in the opposite direction. Since the root system of $S$ is trivial, we may pass from $G$ to $\sD(G)$ and from $t$ to some component in $\sD(G)(k)$ and thus assume that $G$ is semisimple. Replacing $G$ by its universal cover $\widetilde{G}$ and $t$ by an arbitrary lift in $\widetilde{G}(k)$ (which does not affect $\Phi$ or $\Phi_t$), we may then assume that $G$ is simply connected. Under these hypotheses, $G$ is a product of simple groups, so by the same reasoning as before we may assume that $G$ is simple. We have thus reduced (\ref{item:steinberg-torsion-1}) and (\ref{item:steinberg-torsion-2}) to (\ref{item:steinberg-torsion-3}). \smallskip

Let $\sT = (\bR \otimes X_*(T))/X_*(T)$, so that $\sT$ is a compact torus over $\bR$ with character lattice $X(T)$. We may and will regard every character of $T$ as a function $\sT \to \bR/\bZ$. By \cite[5.1]{SteinbergEndomorphisms}, there exists $\tau \in \sT$ such that precisely the same characters vanish at $t$ and $\tau$. Since $t^n \in Z(G)$, all roots vanish at $t^n$. Since $\alpha(t^n) = (n\alpha)(t)$ for all $n$ and all characters $\alpha \in X(T)$ (and similarly for $\tau$), it follows that all roots vanish at $\tau^n$. We wish now to understand the least $n$ such that all roots vanish at $\tau^n$, a task which we take up below. We note for later reference that
\[
\Phi_t = \{\alpha \in \Phi: \alpha(t) = 0\} = \{\alpha \in \Phi: \alpha(\tau) = 0\}.
\]

Let $\al_1, \dots, \al_r$ be a basis for $\Phi$, and let $\al_0 = \sum_{i=1}^{r} -n_i \al_i$ be the negation of the highest root of $\Phi$. Let $V = \bR \otimes X_*(T) = \bR \otimes \bZ \Phi^\vee$, and let $\Sigma$ denote the simplex in $V$ defined by the inequalities $\al_i \geq 0$ for $1 \leq i \leq r$ and $\al_0 \geq -1$. Then $\Sigma$ is a fundamental domain for the action of the affine Weyl group $W_a = \bZ\Phi^\vee \rtimes W = X_*(T) \rtimes W$, where $W$ is the (finite) Weyl group, see \cite[Chap.\ V, \S 3, Thm.\ 2]{Bourbaki}. So the image $\overline{\Sigma}$ of $\Sigma$ in $\sT$ is a fundamental domain for the action of $W$. \smallskip

Let $v \in \Sigma$ be any point, and let $\Phi_v$ denote the subsystem of $\Phi$ consisting of those roots which take integral values at $v$. In other words, if we consider every character $\alpha \in X(T)$ as a function $\sT \to \bR/\bZ$, then $\Phi_v$ consists of those roots which vanish on the image of $v$ in $\sT$. We claim that if $v$ is not a vertex of $\Sigma$, then $\bZ \Phi_v$ is not finite index in $\bZ \Phi$, i.e., $\bQ \Phi_v \neq \bQ \Phi$. Note that in any case we have $-1 \leq \alpha_0(v) \leq 0$ and $0 \leq \alpha_i(v) \leq 1/n_i$ for all $1 \leq i \leq r$. If $v$ is not a vertex, then there exist two indices $0 \leq i < j \leq r$ such that $\alpha_i(v)$ and $\alpha_j(v)$ take neither extreme value in the above inequalities. First suppose $i = 0$, so that $-1 < \alpha_0(v) < 0$ and $0 < \alpha_j(v) < 1/n_j$. It follows that $\Phi_v$ contains no positive root $\beta = \sum_{m=1}^r p_m \alpha_m$ with $p_j \neq 0$: by \cite[Chap.\ VI, \S 1, Prop.\ 25(i)]{Bourbaki}, we have $p_m \leq n_m$ for all $m$, and because $\alpha_m(v) \geq 0$ for all $m \geq 1$ it follows that
\[
0 < \beta(v) = \sum_{m=1}^r p_m \alpha_m(v) < 1,
\]
proving the claim. Thus $\Phi_v \subset \sum_{m \neq j} \bZ\alpha_m$ and $\bQ\Phi_v$ has dimension $\leq r - 1$, as desired. Now suppose $i \neq 0$, so that $0 < \alpha_i(v) < 1/n_i$ and $0 < \alpha_j(v) < 1/n_j$. A similar argument to the above shows that $\Phi_v$ contains no positive root $\sum_{m=1}^r p_m \alpha_m$ with $p_i \neq 0$ and $p_j \neq 0$ other than $-\alpha_0$. Thus $\Phi_v \subset \bZ\alpha_0 + \sum_{m \neq i, j} \bZ\alpha_m$, and we win again. \smallskip

For each $1 \leq i \leq r$, let $v_i$ be the vertex of $\Sigma$ in $V$ defined by $\al_j(v_i) = 0$ for $j \neq i$ and $\al_i(v_i) = 1/n_i$ (so that $\al_0(v_i) = -1$). We now reduce to the case that $\Phi_t = \Phi_i$ for some $i$. Since $\overline{\Sigma}$ is a fundamental domain for the Weyl group action on $\sT$, the point $\tau$ is equivalent to some point of $\overline{\Sigma}$. Note that if $w \in W$ then precisely the same characters vanish at $w \cdot \tau$ and $w \cdot t$ since $\alpha(w \cdot t) = (w^{-1} \cdot \alpha)(t)$ and similarly for $\tau$. Since $\Phi_t$ and $\Phi$ span the same vector space, $\tau$ is in fact equivalent to a vertex by the previous paragraph. In fact we may assume this vertex is not $0$ since otherwise $\Phi_t = \Phi$ and the result is trivial. Thus $\tau$ is equivalent to some $v_i$ as above, so we may assume $\Phi_t = \Phi_i$. Clearly $n_i v_i$ is the smallest multiple of $v_i$ at which all roots vanish, so $n_i$ divides $n$. By \cite[1.15]{SteinbergTorsion}, if $\Phi_i = \Phi_{v_i}$, then $\bZ \Phi/\bZ \Phi_i$ is cyclic of order $n_i$, and in particular its order divides $n$. \smallskip

Let $\alpha_0^\vee = -\sum_{i=1}^r n_i^\vee \alpha_i^\vee$ be the expression of $\alpha_0^\vee$ as a linear combination of elements in the dual coroot basis. By \cite[1.15]{SteinbergTorsion}, $\bZ \Phi^\vee/\bZ \Phi_t^\vee$ has order $n_i^\vee$. One can show (see \cite[1.1]{SteinbergTorsion}) $n_i^\vee\frac{(\alpha_0, \alpha_0)}{(\alpha_i, \alpha_i)} = n_i$, where $(\cdot, \cdot)$ is a chosen $W$-invariant inner product. By \cite[Chap.\ VI, \S 1, Prop.\ 25(iii)]{Bourbaki}, $\alpha_0$ is a long root, so by \cite[Chap.\ VI, \S 1, Prop.\ 12(i)]{Bourbaki} $n_i^\vee$ divides $n_i$. This establishes (\ref{item:steinberg-torsion-3}). \smallskip

For (\ref{item:steinberg-torsion-4}), suppose again that $\Phi_t = \Phi_i$, as we may by replacing $t$ by a $W$-conjugate. In particular, this means that $\alpha_j(t) = 1$ for all $j \neq i$ and $n_i$ is the minimal positive integer such that $\alpha_i(t^{n_i}) = 1$. So $n_i$ is the minimal positive integer such that $t^{n_i} \in Z(G)$. Since also $n_i$ is the order of $\bZ\Phi/\bZ \Phi_i$, point (\ref{item:steinberg-torsion-4}) follows.
\end{proof}

\begin{cor}\label{corollary:steinberg-miracle}
Let $t \in G(k)$ be semisimple and let $H = Z_G(t)^0$.
\begin{enumerate}
    \item\label{item:p-free-fundamental-group} If $p \nmid |\pi_1(\sD(G))|$, then $p \nmid |\pi_1(\sD(H))|$.
    \item\label{item:smooth-center} The quotient $Z(H)/Z(G)$ is smooth.
    \item\label{item:center-cyclic-component-group} If $\sD(G_{\overline{k}})$ is simple then $Z(H)/Z(H)^0Z(G)$ is cyclic, generated by $t$.
\end{enumerate}
\end{cor}

\begin{proof}
We may and do assume that $k$ is algebraically closed and that $G$ is semisimple. For (\ref{item:p-free-fundamental-group}), one can reduce easily to the case that $G$ is simply connected. Note that $\overline{\langle t \rangle}$ is a group of multiplicative type, so we may decompose it as $\overline{\langle t \rangle} = S \times M$, where $S$ is a torus and $M$ is a finite $k$-group of multiplicative type. By \cite[Lem.\ 2.17]{SteinbergTorsion}, $Z_G(S)$ is a connected reductive $k$-group such that $\sD(Z_G(S))$ is simply connected. Since $Z_G(t) = Z_{Z_G(S)}(M)$, we may pass from $G$ to $\sD(Z_G(S))$ to assume that $t$ has finite order $n$. By Theorem~\ref{theorem:steinberg-torsion}(\ref{item:steinberg-torsion-1}), if $T$ is a maximal torus of $H$ then every prime dividing the order of $(\bZ \Phi^\vee/\bZ (\Phi_t)^\vee)_{\rm{tors}}$ divides $n$. By hypothesis, $p$ does not divide the order of $(X_*(T)/\bZ\Phi^\vee)_{\rm{tors}}$. Since $t$ is semisimple, $p$ does not divide $n$. Consider the left exact sequence
\[
0 \to (\bZ\Phi^\vee/\bZ(\Phi_t)^\vee)_{\rm{tors}} \to (X_*(T)/\bZ(\Phi_t)^\vee)_{\rm{tors}} \to (X_*(T)/\bZ\Phi^\vee)_{\rm{tors}}.
\]
We have shown that each of the outer two terms has order not divisible by $p$, so $p$ does not divide the order of $(X_*(T)/\bZ (\Phi_t)^\vee)_{\rm{tors}}$. Thus $p$ does not divide $|\pi_1(\sD(H))|$. \smallskip

For (\ref{item:smooth-center}), arguing as in the previous paragraph we may deal separately with the cases that $\overline{\langle t \rangle} = S$ is a torus and that $t$ is of finite order. In the first case, note that $\Phi_t$ is the set of roots of $\Phi$ which vanish when paired against $X_*(S)$, i.e., $\Phi_t = \Phi \cap (X(T/S) \otimes \bQ)$. Thus $\bZ \Phi/\bZ \Phi_t$ is torsion-free, being a subgroup of $X(S) \otimes \bQ$. The centers of $G$ and $Z_G(S)$ are dual to $X(T)/\bZ \Phi$ and $X(T)/\bZ \Phi_t$, respectively, so that the quotient $Z(Z_G(S))/Z(G)$ is dual to $\bZ \Phi/\bZ \Phi_t$. Thus this quotient is a torus, and in particular it is smooth.\smallskip

Now we may assume that $t$ has finite order, in which case we need to show that $\bZ \Phi / \bZ \Phi_t$ has no $p$-torsion when $p > 0$. Since any semisimple element of $G$ of finite order has order not divisible by $p$, the result follows immediately from Theorem~\ref{theorem:steinberg-torsion}(\ref{item:steinberg-torsion-2}).\smallskip

For (\ref{item:center-cyclic-component-group}), we remark that $Z(H)/Z(H)^0Z(G)$ is Cartier dual to $(\bZ\Phi/\bZ\Phi_t)_{p'\rm{-tors}}$, where the subscript indicates the subgroup of torsion elements of order prime to $p$ (equal by convention to $\bZ\Phi/\bZ\Phi_t$ if $p = 0$). Indeed, $Z(H)/Z(G)$ is Cartier dual to $\bZ\Phi/\bZ\Phi_t$, and the component group of $Z(H)/Z(G)$ is therefore Cartier dual to $(\bZ\Phi/\bZ\Phi_t)_{p'-\rm{tors}}$. Since
\[
(Z(H)/Z(G))/(Z(H)/Z(G))^0 \cong Z(H)/Z(H)^0Z(G),
\]
the remark follows. By Theorem~\ref{theorem:steinberg-torsion}(\ref{item:steinberg-torsion-3}) and (\ref{item:steinberg-torsion-4}), the group $(\bZ\Phi/\bZ\Phi_t)_{p'\rm{-tors}}$ is cyclic of some order $m$, and if $\alpha$ is a generator then $\alpha(t) \in k^\times$ is of exact order $m$. Thus in the bilinear pairing
\[
(Z(H)/Z(H)^0Z(G))(k) \times (\bZ \Phi/\bZ\Phi_t)_{p'\rm{-tors}} \to k^\times
\]
we see that $t$ is dual to a generator of $(\bZ \Phi/\bZ\Phi_t)_{p'\rm{-tors}}$, and thus the result follows.
\end{proof}

\begin{lemma}\label{lemma:etale-central-isogeny-centralizer}
Let $H$ be a connected reductive $k$-subgroup of $G$ containing $\sD(G)$, and let $h \in H(k)$.
\begin{enumerate}
    \item\label{item:subgroup-centralizer-no-change} The natural morphism $Z_H(h)/Z(H) \to Z_G(h)/Z(G)$ is an isomorphism.
    \item\label{item:nu-diagram-homomorphism} Suppose that $\pi: G' \to H$ is a central isogeny of connected reductive groups and $g' \in G'(k)$ is a lift of $h$. There is a diagram
    \[
    \begin{tikzcd}
    Z_{G'}(g') \arrow[r, "i"] \arrow[d]
        &\pi^{-1}(Z_H(h)) \arrow[r] \arrow[d, "\nu"]
        &G' \arrow[d, "\mu"] \\
    \Spec k \arrow[r, "1"]
        &\ker \pi \arrow[r]
        &G'
    \end{tikzcd}
    \]
    whose squares are Cartesian, where $i: Z_{G'}(g') \to \pi^{-1}(Z_H(h))$ is the canonical inclusion and $\mu: G' \to G'$ is the map $\mu(x) = xg'x^{-1}g'^{-1}$. The map $\nu$ is a homomorphism.
    \item\label{item:i-is-an-open-embedding} If $\pi$ is etale, then $i$ is an open embedding.
    \item\label{item:nu-factors-through-identity-component} Suppose further that $\pi$ is etale and $h$ has the following property: for $h_{\overline{k}} = tu$ the Jordan decomposition of $h_{\overline{k}}$ in $H(\overline{k})$, the component group $Z_{H_{\overline{k}}}(t)/Z_{H_{\overline{k}}}(t)^0$ has order prime to the degree of $\pi$. Then $\nu$ factors through $(\ker \pi)^0 = \Spec k$ and the natural morphism $Z_{G'}(g')/Z(G') \to Z_H(g)/Z(H)$ is an isomorphism.
\end{enumerate}
\end{lemma}

\begin{proof}
We may and do assume that $k$ is algebraically closed. For (\ref{item:subgroup-centralizer-no-change}), let $S$ be the maximal central torus in $G$, and let $S_0 \subset S$ be a subtorus such that the multiplication morphism $S_0 \times H \to G$ is a central isogeny. Let $\vp: Z_H(h)/Z(H) \to Z_G(h)/Z(G)$ be the natural morphism; we claim that $\vp$ is an isomorphism. Since $H$ is a closed subgroup of $G$, $\vp$ is clearly a monomorphism. To show that $\vp$ is an isomorphism, it therefore suffices to show that it is an epimorphism of fppf sheaves. If $R$ is a $k$-algebra and $\overline{x} \in (Z_G(h)/Z(G))(R)$, then after replacing $R$ by an fppf extension we may assume that there is some $x \in Z_G(h)(R)$ lifting $\overline{x}$. Passing to a further fppf extension and translating by a point of $S_0(R)$, we may assume $x \in Z_G(h) \cap H = Z_H(h)$. So indeed $\vp$ is an isomorphism. \smallskip

The claim that the diagram in (\ref{item:nu-diagram-homomorphism}) exists and has Cartesian squares is simple and left to the reader. To prove $\nu$ is a homomorphism, let $R$ be a $k$-algebra and let $x, y \in \pi^{-1}(Z_H(h))(R)$. This means that $x g' x^{-1} g'^{-1}$ and $y g' y^{-1} g'^{-1}$ lie in $\ker \pi$, and in particular these elements are central. Thus we have
\begin{align*}
\nu(xy) = xyg'y^{-1}x^{-1}g'^{-1} &= x(yg'y^{-1}g'^{-1})g'x^{-1}g'^{-1} \\
    &= xg'x^{-1}g'^{-1}yg'y^{-1}g'^{-1} \\
    &= \nu(x)\nu(y).
\end{align*}
So indeed (\ref{item:nu-diagram-homomorphism}) holds. If $\pi$ is etale, then the identity section $1: \Spec k \to \ker \pi$ is an open embedding, so $i$ is an open embedding by base change, proving (\ref{item:i-is-an-open-embedding}). \smallskip

For (\ref{item:nu-factors-through-identity-component}), it is harmless to replace $h$ by a translate by any element of $Z(G)(k)$, and so by (\ref{item:subgroup-centralizer-no-change}) we may pass from $G$ to $\sD(G)$ to assume that $G$ is semisimple and in particular $H = G$. Let $g' = t'u'$ be the Jordan decomposition of $g'$, giving a corresponding decomposition $h = tu$ of $h$ by functoriality of the Jordan decomposition. Note that $\ker \pi$ lies in $Z_{G'}(t')^0$: if $T'$ is a maximal torus of $G'$ containing $t'$, then $\ker \pi \subset T' \subset Z_{G'}(t')^0$. Looking at the Cartesian diagram
\[
\begin{tikzcd}
Z_{G'}(t') \arrow[r] \arrow[d]
    &\pi^{-1}(Z_H(t)) \arrow[d, "\nu"] \\
\Spec k \arrow[r]
    &\ker \pi
\end{tikzcd}
\]
in the lemma, we see that $Z_{G'}(t')$ is an open subscheme of $\pi^{-1}(Z_H(t))$, and thus $Z_{G'}(t')/\ker \pi$ is an open subscheme of $\pi^{-1}(Z_H(t))/\ker \pi = Z_H(t)$. By the previous paragraph, $\nu$ is a homomorphism, and since it is trivial when restricted to the open subscheme $Z_{G'}(t')/\ker \pi$, the morphism $\nu$ factors through the component group of $Z_H(t)$. Since $Z_H(t)/Z_H(t)^0$ has order prime to the order of $\ker \pi$ by hypothesis, in fact $\nu$ is trivial. Thus the natural morphism $Z_{G'}(t')/\ker \pi \to Z_H(t)$ is an isomorphism, and thus $Z_{G'}(t') = \pi^{-1}(Z_H(t))$. By replacing $H$ and $G'$ by $Z_H(t)$ and $Z_{G'}(t')$, respectively, we may therefore assume that $h$ and $g'$ are unipotent. \smallskip

If $R$ is a $k$-algebra and $x \in G'(R)$ is such that $xg'x^{-1}g^{-1} \in (\ker \pi)(R)$ then $xg'x^{-1} \in (\ker \pi)(R) \cdot g'$. Since $g'$ is unipotent, the same is true of any conjugate, and thus the homomorphism $\nu$, being given by a commutator formula, factors through $(\ker \pi)^0 = \Spec k$. Thus again we see $Z_{G'}(g') = \pi^{-1}(Z_H(h))$, and passing to the quotient by $Z(G')$ yields the result.
\end{proof}

It is claimed in some places (e.g., \cite[\S 4.1, following remark]{Humphreys}) that if $B \subset G$ is a Borel $k$-subgroup and $u$ is a regular unipotent element in the unipotent radical of $B$, then $Z_G(u) = Z_B(u)$. This is true on the level of $\ov k$-points, but it fails schematically, as the following example shows.

\begin{example}\label{example:pgl2-centralizer}
    Let $G = \PGL_2$ over a field $k$ of characteristic $2$, let $u' = \begin{pmatrix}
        1 &1 \\ 0 &1
    \end{pmatrix} \in \SL_2(k)$, and let $u$ be the image of $u'$ in $G(k)$. A straightforward computation yields
    \[
    Z_B(u) = \left\{\begin{pmatrix}
        a &b \\ c &d
    \end{pmatrix}\colon c^2 = 0, a^2 = ad-bc-ac\right\}.
    \]
    From this description, we see that $Z_G(u)_{\rm{red}} = \bG_a$, the unipotent radical of the upper-triangular Borel subgroup $B$ of $G$. However, $Z_G(u)$ is not reduced: concretely, the element $\begin{pmatrix}
        1 + \eps &0 \\ \eps &1
    \end{pmatrix}$ lies in $Z_G(u)(k[\eps]/(\eps^2))$ but is not upper-triangular.
\end{example}

A slightly weaker version of the following lemma appears in \cite[Props.\ 2.3 and 2.6]{BRR}, where it is assumed that $Z(G)$ is smooth. We remove this assumption using a simple trick which we will find useful several times in the sequel.

\begin{lemma}\label{lemma:smooth-unipotent-centralizer}
Let $u \in G(k)$ be regular unipotent, and let $B$ be the unique Borel subgroup of $G$ containing $u$. Then $Z_B(u)$ is the direct product of $Z(G)$ and $Z_U(u)$, where $U$ is the unipotent radical of $B$, and $Z_U(u)$ is smooth. If moreover $p \nmid |\pi_1(\sD(G))|$, then $Z_G(u) = Z_B(u)$.
\end{lemma}

\begin{proof}
We may and do assume that $k$ is algebraically closed. The claim that $Z_B(u) = Z(G) \times Z_U(u)$ is standard; see for instance the discussion in the final paragraph of \cite[\S 4.1]{Humphreys}. The claim that $Z_U(u)$ is smooth is also standard; since we need a stronger claim below anyway, we will reprove this in the final paragraph. The claim that $Z_G(u) = Z_B(u)$ is also ``standard" (see \cite[\S 4.1]{Humphreys} again), but Example~\ref{example:pgl2-centralizer} shows that it is false without the assumption $p \nmid |\pi_1(\sD(G))|$. In the following, we will therefore assume that $p \nmid |\pi_1(\sD(G))|$.

Now we reduce to proving that $Z_G(u)/Z(G)$ is smooth; suppose that $Z_G(u)/Z(G)$ is smooth. Since $B$ is the unique Borel subgroup of $G$ containing $u$ and $N_G(B) = B$, we see that $Z_G(u)(k) = Z_B(u)(k)$, and thus $Z_B(u)$ and $Z_G(u)$ have the same underlying reduced closed subscheme. In particular, the inclusion $Z_B(u)/Z(G) \to Z_G(u)/Z(G)$ must be an equality by smoothness of the latter. So we see that $Z_G(u) = Z_B(u)$. By the previous paragraph, $Z_B(u) = Z(G) \times Z_U(u)$, so from smoothness of $Z_B(u)/Z(G)$ it follows that $Z_U(u)$ is smooth.\smallskip

Next we reduce to the case that $\sD(G)$ is simply connected. Let $S$ be the maximal central torus of $G$, so the multiplication map $S \times \sD(G) \to G$ is a central isogeny. Then the map $S \times Z_{\sD(G)}(u) \to Z_G(u)$ is a central isogeny, so we may pass from $G$ to $\sD(G)$. Let $\pi: \widetilde{G} \to \sD(G)$ be the universal cover (which is etale by hypothesis on $p$). Choose an arbitrary lift $\widetilde{g} \in \widetilde{G}(k)$ of $g$, so that $\widetilde{g}$ is regular semisimple. By Lemma~\ref{lemma:etale-central-isogeny-centralizer}(\ref{item:i-is-an-open-embedding}), since $\pi$ is etale the morphism $Z_{\widetilde{G}}(\widetilde{g}) \to \pi_G^{-1}(Z_G(g))$ is an open embedding. Thus also $Z_{\widetilde{G}}(\widetilde{g})/Z(\widetilde{G}) \to Z_G(g)/Z(G) = \pi_G^{-1}(Z_G(g))/Z(\widetilde{G})$ is an open embedding as this can be checked after fppf base change. Thus the smoothness of $Z_G(g)/Z(G)$ is equivalent to that of $Z_{\widetilde{G}}(\widetilde{g})/Z(\widetilde{G})$. So we may and do assume from now on that $\sD(G)$ is simply connected.\smallskip

Finally we reduce to the case that $Z(G)$ is smooth. Embed the center $Z(G)$ in a torus $T$, and let $G' = G \times^{Z(G)} T$ be the pushout of $G$ and $T$ along this inclusion, so that $G'$ is a connected reductive group containing $G$ with smooth center, and $\sD(G') \cong \sD(G)$. Moreover, $Z_{G'}(u) = Z_G(u) \times^{Z(G)} T$, so that smoothness of $Z_G(u)/Z(G)$ is equivalent to that of $Z_{G'}(u)$.\smallskip

It remains to prove that $Z_G(u)$ is smooth under the further assumptions that $\sD(G)$ is simply connected and $Z(G)$ is smooth. In this case, Theorem~\ref{theorem:properties-of-regular-unipotents}(\ref{item:lie-algebra-dimension-bound}) shows that $\dim \fg^{\Ad(u)} \leq \dim \fz + r$, where $\fz$ is the center of the Lie algebra $\fg$ and $r$ is the semisimple rank of $G$. By \cite[Prop.\ 3.3.8]{Conrad}, $\fz = \Lie Z(G)$, and Lemma~\ref{lemma:lie-algebra-of-centralizer} shows that $\Lie Z_G(u) = \fg^{\Ad(u)}$. Since $Z_U(u)$ is $r$-dimensional by regularity of $u$, we have
\begin{align*}
\dim \fz + r &\leq \dim \Lie Z(G) + \dim \Lie Z_U(u) \\
    &= \dim \Lie Z_B(u) \\
    &\leq \dim \fg^{\Ad(u)} \\
    &\leq \dim \fz + r.
\end{align*}
Thus every inequality is an equality, so $\Lie Z_U(u)$ is $r$-dimensional and $\Lie Z_G(u) = \Lie Z_B(u)$ is $(\dim \fz + r)$-dimensional. So $Z_U(u)$ is smooth, and because $Z_G(u)$ contains the $(\dim \fz + r)$-dimensional group $Z(G) \times Z_U(u)$, it is in fact equal to this subgroup. Thus we see that $Z_G(u)$ is smooth and hence we are done.
\end{proof}

\begin{cor}\label{corollary:general-centralizer}
Suppose $p\nmid |\pi_1(\sD(G))|$, and let $g \in G(k)$ be a regular element. Then $Z_G(g)/Z(G)$ is smooth.
\end{cor}

\begin{proof}
Let $g = tu$ be the Jordan decomposition of $g$, and let $H = Z_G(t)^0$, so that $H$ is a connected reductive group. By Corollary~\ref{corollary:steinberg-miracle}(\ref{item:p-free-fundamental-group}), $p$ does not divide the order of $\pi_1(\sD(H))$, so that Lemma~\ref{lemma:smooth-unipotent-centralizer} shows that $Z_H(u)/Z(H)$ is smooth. By Corollary~\ref{corollary:steinberg-miracle}(\ref{item:smooth-center}), $Z(H)/Z(G)$ is smooth, so that $Z_H(u)/Z(G)$ is also smooth. By Lemma~\ref{lemma:jordan-decomposition-centralizer}, we have $Z_{Z_G(t)}(u) = Z_G(g)$, so because $Z_{Z_G(t)}(u)^0 = Z_H(u)^0$, it follows that
\[
Z_H(u)^0 = Z_G(g)^0.
\]
This yields smoothness of $Z_G(g)^0/(Z_G(g)^0 \cap Z(G))$, thus also of $Z_G(g)/Z(G)$ since the former is the identity component of the latter.
\end{proof}

\begin{remark}\label{remark:pgl-2}
The hypothesis on $|\pi_1(\sD(G))|$ in Lemma~\ref{lemma:smooth-unipotent-centralizer} is necessary. For an explicit example, let $G = \PGL_2$ over a field $k$ of characteristic $2$, and let $u = \begin{pmatrix} 1 & 1 \\ 0 & 1 \end{pmatrix}$. In that case, a simple calculation shows that functorially
\[
Z_{\PGL_2}(u) = \left\{\begin{pmatrix} a & b \\ c & d \end{pmatrix}: c^2 = 0, \,a^2 = ad + bc + ac\right\}.
\]
Thus $Z_{\PGL_2}(u)/Z(\PGL_2) = Z_{\PGL_2}(u)$ is not smooth.
\end{remark}

\subsection{Centralizers of Lie algebra elements}\label{subsection:centralizers-of-lie-algebra-elements}

In this section we collect a few results which will be necessary in proving Theorem~\ref{theorem:intro-flat-regular-centralizer} in the Lie algebra case. Since we restrict attention to the case of good characteristic, the results that we need are considerably simpler and more uniform than those in the previous section.

\begin{prop}\label{prop:semisimple-centralizer-lie-algebra}
If $X \in \mathfrak{g}$ is semisimple, then $Z_G(X)$ is a reductive group. If $p = \chara k$ is not a torsion prime for $G$, then $Z_G(X)$ is connected and the quotient $Z(Z_G(X))/Z(G)$ is connected. If $p$ is good for $G$, then $Z(Z_G(X))/Z(G)$ is a torus.
\end{prop}

\begin{proof}
First we note that $Z_G(X)$ is reductive by \cite[13.19 Prop.]{Borel}. The connectedness claim follows from \cite[Thm.\ 3.14]{SteinbergTorsion}. Now $Z(Z_G(X))/Z(G)$ is Cartier dual to $Y = \bZ\Phi/\{\alpha \in \bZ\Phi: d\alpha(X) = 0\}$. Note that $Y$ is clearly $n$-torsion free for every integer $n$ not divisible by $p$, so $Z(Z_G(X))/Z(G)$ is indeed connected. If $p$ is good for $G$, then $Y$ is $p$-torsion free by definition of good primes, and thus $Z(Z_G(X))/Z(G)$ is smooth. Since every torsion prime is a bad prime, it follows that $Z(Z_G(X))/Z(G)$ is indeed a torus.
\end{proof}

The following proposition has a statement very similar to \cite[Thm.\ 5.9]{Springer}, but there are two notable differences. First, the smoothness statement of \cite[Thm.\ 5.9 a)]{Springer} assumes that $p$ is either $0$ or a \textit{very good} prime for $G$, and we assume simply that $p$ is good for $G$. Second, in good characteristic, the claim in \cite[Thm.\ 5.9 b)]{Springer} that $Z_G(X) = Z(G) \times Z_U(X)$ is only generally true on the level of $\overline{k}$-points; to obtain the stronger schematic statement, it is necessary that $p$ does not divide $|\pi_1(\sD(G))|$.

\begin{prop}\label{prop:connectedness-of-lie-algebra-centralizers}
Let $X \in \mathfrak{g}$ be regular nilpotent, and let $B$ be the unique Borel subgroup of $G$ containing $X$ in its Lie algebra. If $p$ is good for $G$ and $p$ does not divide $|\pi_1(\sD(G))|$, then $Z_G(X) = Z(G) \times Z_U(X)$, and $Z_U(X)$ is smooth.
\end{prop}

\begin{proof}
We may and do assume that $k$ is separably closed, so $G$ is split. Using the assumption on $|\pi_1(\sD(G))|$, we may pass from $G$ to the universal cover of $\sD(G)$ to assume that $G$ is simply connected. Passing to simple factors, we may further assume that $G$ is simple. If $G = \SL_n$, then $X+1 \in \SL_n(k)$ is unipotent and $Z_G(X) = Z_G(X+1)$, so the result follows from Lemma~\ref{lemma:smooth-unipotent-centralizer}. If $G$ is not of type $\rm{A}$, then $p$ is very good for $G$ and thus the result follows from \cite[Thm.\ 5.9]{Springer}.
\end{proof}

\begin{cor}\label{corollary:smooth-connected-lie-algebra-centralizer}
Let $G$ be a reductive group over a field $k$ of characteristic $p \geq 0$, and let $X \in \mathfrak{g}$ be a regular element. If $p$ is good for $G$ and $p \nmid |\pi_1(\sD(G))|$, then $Z_G(X)/Z(G)$ is smooth and connected.
\end{cor}

\begin{proof}
We may and do assume that $k$ is algebraically closed. In this case, we have a Jordan decomposition $X = X_{\mathrm{ss}} + X_{\mathrm{n}}$ for commuting elements $X_{\mathrm{ss}}, X_{\mathrm{n}} \in \mathfrak{g}$, where $X_{\mathrm{ss}}$ is semisimple and $X_{\mathrm{n}}$ is nilpotent. By Proposition~\ref{prop:semisimple-centralizer-lie-algebra}, $Z_G(X_{\mathrm{ss}})$ is a connected reductive group. Since $X_{\mathrm{ss}}$ lies in the Lie algebra of a maximal torus of $G$ \cite[11.8 Prop.]{Borel}, we see that $Z_G(X_{\mathrm{ss}})$ has the same rank as $G$, so $X_{\mathrm{n}}$ is regular nilpotent in $Z_G(X_{\mathrm{ss}})$ by regularity of $X$ in $G$ and Lemma~\ref{lemma:jordan-decomposition-centralizer}. By Proposition~\ref{prop:connectedness-of-lie-algebra-centralizers}, we have that
\[
Z_G(X) = Z_{Z_G(X_{\mathrm{ss}})}(X_{\mathrm{n}}) = Z(Z_G(X_{\mathrm{ss}})) \times V,
\]
where $V$ is a smooth connected unipotent closed subgroup of $Z_G(X_{\mathrm{ss}})$. By Proposition~\ref{prop:semisimple-centralizer-lie-algebra}, the quotient $Z(Z_G(X_{\mathrm{ss}}))/Z(G)$ is a torus, so it follows that the group
\[
Z_G(X)/Z(G) = (Z(Z_G(X_{\mathrm{ss}}))/Z(G)) \times V
\]
is indeed smooth and connected.
\end{proof}

\subsection{Centralizers of fiberwise regular sections over a general base scheme}\label{subsection:regular-centralizers}

We aim to prove the following theorem. As mentioned in the introduction, the flatness assertion in the second point is a slightly weaker form of \cite[Thm.\ 4.2.8]{Bouthier-Cesnavicius}: namely, the latter assumes that for all $s \in S$, the characteristic of $k(s)$ is non-torsion for $G_s$, while our result assumes that $\chara k(s)$ is not bad for $G_s$. These conditions are almost the same, but our result does not include the case that $G_s$ admits a factor of type $\mathrm{C}_n$ ($n \geq 2$) when $\chara k(s) = 2$, nor the case that $G_s$ admits a factor of type $\mathrm{G}_2$ when $\chara k(s) = 3$.

\begin{theorem} \label{theorem:flat-centralizer}
Let $S$ be a scheme and let $G \to S$ be a reductive group scheme. Suppose that $|\pi_1(\sD(G))|$ is invertible on $S$.
\begin{enumerate}
    \item\label{item:flat-regular-centralizer-group} If $g \in G(S)$ is a fiberwise strongly regular section of $G$, then $Z_G(g)$ is flat and $Z_G(g)/Z(G)$ is smooth.
    \item\label{item:flat-regular-centralizer-lie-algebra} Suppose that for every $s \in S$, the characteristic of $k(s)$ is good for $G_s$. If $X \in \mathfrak{g}(S)$ is a fiberwise regular section of $\mathfrak{g}$, then $Z_G(X)$ is flat and $Z_G(X)/Z(G)$ is smooth.
\end{enumerate}
In particular, if $g$ (resp.\ $X$) is fiberwise regular semisimple, then $Z_G(g)$ (resp.\ $Z_G(X)$) is a torus.
\end{theorem}

\begin{proof}
By Lemma~\ref{lemma:representability-of-stabilizer}, $Z_G(g)$ and $Z_G(X)$ are both represented by finitely presented closed subgroups of $G$. In light of this, the final claim follows from the others by \cite[Thm.\ B.4.1]{Conrad}, since in that case the fibral centralizers $Z_{G_s}(g_s)$ and $Z_{G_s}(X_s)$ are tori (the latter by Proposition~\ref{prop:semisimple-centralizer-lie-algebra} and dimension considerations). In general, \cite[Thm.\ 3.3.4]{Conrad} shows that $Z(G)$ is a closed subgroup scheme of $G$ of multiplicative type over $S$, so by \cite[Exp.\ VIII, Thm.\ 5.1]{SGA3II} $Z_G(g)/Z(G)$ and $Z_G(X)/Z(G)$ exist as schemes and $Z_G(g) \to Z_G(g)/Z(G)$ and $Z_G(X) \to Z_G(X)/Z(G)$ are $Z(G)$-torsors. Thus it suffices to show that $Z_G(g)/Z(G)$ and $Z_G(X)/Z(G)$ are smooth. We deal first with (\ref{item:flat-regular-centralizer-group}). \smallskip

If $\widetilde{G}$ is the universal cover of $\sD(G)$ and $\widetilde{g} \in \widetilde{G}(S)$ is a lift of $g$ (which exists fppf-locally), by Lemma~\ref{lemma:etale-central-isogeny-centralizer}(\ref{item:nu-factors-through-identity-component}) we may pass from $G$ to $\widetilde{G}$ and from $g$ to $\widetilde{g}$ to assume that $G$ is simply connected. In particular, the notions of regularity and strong regularity coincide in this case. By the Existence and Isomorphism Theorems \cite[Thms.\ 6.1.16 and 6.1.17]{Conrad}, $G$ comes from the base change of a split reductive $\bZ$-group scheme $\bG$, i.e., $G \cong \bG \times_{\Spec \bZ} S$. It suffices to consider the universal case of $g \in \bG(\bG_{\rm{reg}})$ (where $\bG_{\rm{reg}}$ denotes the open regular locus in $\bG$), so we may pass from $S$ to $\bG_{\rm{reg}}$ to assume that $S$ is reduced and noetherian. By the valuative criterion of flatness \cite[IV\textsubscript{3}, Thm.\ 11.8.1]{EGA}, we may pass to the case that $A$ is a DVR. \smallskip

Since $A$ is a DVR, properness of the scheme of parabolics shows that after extending $A$ we may assume that there is a Borel $A$-subgroup $B$ of $G$ such that $g$ lies in $B(A)$. We will show first that $Z_B(g)$ is $A$-flat. Let $c_g: B \to B$ be the conjugation morphism given by $c_g(b) = bgb^{-1}$. If $U$ is the unipotent radical of $B$, then $c_g$ factors through $gU$ by the same reasoning as in the first paragraph of Section~\ref{subsection:regular-elements}. It suffices therefore to show that the morphism $c_g: B \to gU$ is flat. Since $B$ is $A$-flat, the fibral flatness criterion and flat descent reduce one to showing this when the base is an algebraically closed field $k$. In this case, because $B$ and $gU$ are smooth, Miracle Flatness \cite[Thm.\ 23.1]{Matsumura} reduces one to showing that all nonempty fibers of $c_g$ are of the same dimension. But now $c_g$ is an orbit map, so one sees functorially that every nonempty fiber over a $k$-point is isomorphic to $Z_B(g)$. So indeed $Z_B(g)$ is flat. \smallskip

Finally, we need to show that $Z_B(g) = Z_G(g)$. Since $Z_B(g)$ is flat, the fibral isomorphism theorem shows that it suffices to prove that the morphism $Z_B(g) \to Z_G(g)$ is an isomorphism on fibers, where it follows from Corollary~\ref{corollary:general-centralizer}. Thus indeed $Z_G(g)$ is flat. In the presence of flatness, smoothness may be checked fibrally, where it follows also from Corollary~\ref{corollary:general-centralizer}. \smallskip

For (\ref{item:flat-regular-centralizer-lie-algebra}), we note that one may reduce in the same way to the case that $S = \Spec A$ is the spectrum of a DVR whose residue characteristic is good for $G$. At this point, one may proceed in the same way as before using Corollary~\ref{corollary:smooth-connected-lie-algebra-centralizer} in place of Corollary~\ref{corollary:general-centralizer}.
\end{proof}

\begin{cor}\label{corollary:commutative-centralizer}
Under the hypotheses of Theorem~\ref{theorem:flat-centralizer}(\ref{item:flat-regular-centralizer-group}) (resp.\ (\ref{item:flat-regular-centralizer-lie-algebra})), $Z_G(g)$ (resp.\ $Z_G(X)$) is commutative.
\end{cor}

\begin{proof}
We only treat commutativity of $Z_G(g)$, the other case being similar (and already proven in slightly greater generality in \cite[Thm.\ 4.2.8]{Bouthier-Cesnavicius}). As in the proof of Theorem~\ref{theorem:flat-centralizer}, we may and do assume that $G$ is simply connected. By spreading out and working etale-locally, we may assume $S$ is noetherian and $G$ is $S$-split, and so $G \cong \bG \times_{\Spec \bZ} S$ for some split reductive group scheme $\bG$ over $\bZ$. We may pass as before to the case that $g \in \bG(\bG_{\mathrm{reg}})$ is the universal point. If $k$ is the residue field of the generic point of $\bG_{\mathrm{reg}}$, then the associated point $g_k$ is regular semisimple in $\bG(k)$, and thus $Z_{G_k}(g_k)$ is a torus (hence commutative). Because $Z_G(g)$ is flat, it follows that it is commutative.
\end{proof}

\begin{remark}
When $S$ is the spectrum of a field and $g$ is regular unipotent, Corollary~\ref{corollary:commutative-centralizer} was first proved (on the level of geometric points) in good characteristic by Springer \cite{Springer-note}, and later in bad characteristic by Lou \cite{Lou} using extensive case-by-case checking and computer calculations. As the reader can check, our proof does not rely, even implicitly, on any case-checking, and thus it gives a new case-free proof of this result. Even in the field case, the main input in our proof is Theorem~\ref{theorem:flat-centralizer}.
\end{remark}

\section{The unipotent and nilpotent schemes}\label{section:unip-nilp}

In this section our main aim is to introduce \textit{canonical} definitions of the unipotent and nilpotent schemes over a general base scheme, and to show that they have good properties. The precise statements are somewhat complicated, and we refer the reader to Theorems~\ref{theorem:unipotent-scheme} and \ref{theorem:nilpotent-scheme}.

\subsection{Steinberg morphisms}

Throughout this section, $G$ is a connected reductive group over a field $k$. In \cite[\S 6]{SteinbergReg}, Steinberg analyzes the natural morphism $\chi: G \to G/\!/G$, now called the \textit{Steinberg morphism}. This is a very interesting morphism, and the main point of \cite{SteinbergReg} is to understand its properties. We summarize some of these in the following lemma.

\begin{lemma}\label{lemma:properties-of-chi}
Let $G$ be a connected semisimple group over a field $k$ and let $\chi: G \to G/\!/G$ be the Steinberg morphism as above.
\begin{enumerate}
    \item\label{item:chi-unipotence} \cite[Cor.\ 6.7]{SteinbergReg} An element $g \in G(k)$ is unipotent if and only if $\chi(g) = \chi(1)$.
    \item\label{item:irreducible-fibers} \cite[Thm.\ 6.11]{SteinbergReg} Every geometric fiber $\chi^{-1}(x)$ is irreducible of codimension $\rk G$ in $G$. It contains a unique conjugacy class of regular elements, and this conjugacy class is open in $\chi^{-1}(x)_{\rm{red}}$ with complementary codimension at least $2$.
    \item\label{item:regularity-and-smoothness} \cite[3.10, Thm.\ vi)]{Slodowy} If $p \nmid |\pi_1(G)|$ then an element $g \in G(k)$ is regular if and only if $\chi$ is smooth at $g$.
\end{enumerate}
\end{lemma}

By \cite[Cor.\ 6.4]{SteinbergReg}, if $T$ is a split maximal $k$-torus of $G$ and $W$ is the Weyl group of $(G, T)$, then the natural restriction map $k[G]^G \to k[T]^W$ is an isomorphism. A special case of the following lemma is mentioned but unproved in \cite[4.5, Rmk.]{Slodowy}.

\begin{lemma}\label{lemma:adjoint-quotient-smooth-near-1}
Let $G$ be a connected semisimple group over a field $k$ of characteristic $p \geq 0$ such that $p \nmid |\pi_1(\sD(G))|$ and suppose that $T$ is a split maximal $k$-torus of $G$ with Weyl group $W$. Let $\Omega$ be the open subscheme of $T$ consisting of those $t \in T$ such that $Z_{G_{k(t)}}(t)$ is connected, where $k(t)$ is the residue field of $t$. Then the smooth locus of $T/\!/W$ contains the image of $\Omega$.
\end{lemma}

\begin{proof}
Since the formations of $k[T]^W$ and $k[\widetilde{T}]^W$ commute with all extensions of $k$, we may and do assume that $k$ is algebraically closed. Let $\pi_G: \widetilde{G} \to G$ be the universal cover of $G$, and let $Z = \ker \pi_G$; by assumption, the group $Z$ is finite constant. If $\widetilde{T} = \pi_G^{-1}(T)$, then we have $T/\!/W = (\widetilde{T}/\!/W)/\!/Z$, where $Z$ acts by translation. By \cite[Chap.\ VI, \S 3.4, Thm.\ 1]{Bourbaki}, the scheme $\widetilde{T}/\!/W$ is isomorphic to an affine space, and in particular it is smooth. To show smoothness of $T/\!/W$ at the image of $\Omega$, it therefore suffices to show that $Z$ acts freely on $\Omega$. Since $Z$ is a constant $k$-group and $k = \overline{k}$, it suffices even to show that $Z(k)$ acts freely on $\Omega$. \smallskip

Suppose that the image $[\widetilde{t}]$ of $\widetilde{t} \in \widetilde{T}$ in $\widetilde{T}/\!/W$ is fixed under translation by some non-identity element $c \in Z(k)$. This means that $c\widetilde{t}$ and $\widetilde{t}$ are $W$-conjugate; i.e., there is some $w \in W$ such that $c\widetilde{t} = w\widetilde{t}w^{-1}$. Motivated by this, for given $c \in Z(k)$ and $w \in W$ we will set
\[
\widetilde{T}_{c, w} = \{\widetilde{t} \in \widetilde{T}: c\widetilde{t} = w\widetilde{t}w^{-1}\}.
\]
Thus if $[\widetilde{t}]$ denotes the image of $\widetilde{t} \in \widetilde{T}$ in $(\widetilde{T}/\!/W)$, then we have
\[
\{\widetilde{t} \in \widetilde{T}: \Stab_{Z(k)}([\widetilde{t}]) \neq 1\} = \bigcup_{\substack{c \in Z(k) - \{1\} \\ w \in W}} \widetilde{T}_{c, w}
\]
and the free locus of the action of $Z(k)$ on $\widetilde{T}/\!/W$ is therefore the (open) complement of the image of the right side of the above equality. Note that if $\widetilde{t} \in \widetilde{T}_{c, w}$, then there is a representative of $w$ in $Z_{G_{k(t)}}(\pi_G(\widetilde{t}))$, although there is none in $Z_{\widetilde{G}_{k(t)}}(\widetilde{t})$. Thus from \cite[Lem.\ 2.14]{SteinbergTorsion}, we see that the image of $\Omega$ is \textit{equal} to the free locus of $Z(k)$, as desired.
\end{proof}

There is also an analogue of the Steinberg morphism for $\fg = \Lie G$: one has a natural morphism $\chi: \fg \to \fg/\!/G := \Spec (\Sym_k \fg^*)^G$.

\begin{lemma}\label{lemma:properties-of-chi-nilpotent} \cite[7.13, Prop.]{Jantzen-nilpotent}
Let $G$ be a connected reductive group over a field $k$ with Lie algebra $\fg$ and let $\chi: \fg \to \fg/\!/G$ be the Steinberg morphism as above.
\begin{enumerate}
    \item\label{item:Lie-chi-nilpotence} An element $X \in \fg$ is nilpotent if and only if $\chi(X) = \chi(0)$.
    \item\label{item:Lie-chi-irreducible-fibers} Every geometric fiber $\chi^{-1}(x)$ is irreducible of codimension $r$ in $G$.
\end{enumerate}
\end{lemma}

If $T$ is a split maximal torus of $G$ with Weyl group $W$ and Lie algebra $\ft$, then by \cite[7.12]{Jantzen-nilpotent} the natural restriction map $(\Sym \fg^*)^G \to (\Sym \ft^*)^W$ is injective, and it is surjective if either $\chara k \neq 2$ or $\chara k = 2$ and $\alpha \not\in 2X(T)$ for all $\alpha \in \Phi(G, T)$; this excludes precisely the groups $\Sp_{2n}$ ($n \geq 1$) in characteristic $2$ by \cite[1.3.1]{Chaput-Romagny}.

\subsection{The unipotent scheme}\label{subsection:unip-sch}

Let $G$ be a connected reductive group over a field $k$. If $k$ is separably closed, then we define the \textit{unipotent variety} $\sU^{\var} = \sU_G^{\var}$ to be the Zariski closure of the set of unipotent elements in $G(k)$. We note that if $k$ is imperfect, then although the Jordan decomposition need not be defined in $G(k)$ the notion of unipotence still makes sense and is preserved by any field extension. If $k$ is not assumed separably closed, then the Galois action on $G(k_s)$ evidently preserves $\sU^{\var}(k_s)$, so one can still define the unipotent variety $\sU^{\var}$ of $G$. In \cite[Prop.\ 4.4]{Springer} it is proved using the Springer resolution that $\sU^{\var}$ is an irreducible closed subscheme of $G$ of dimension $\dim G - \rk G$. We remark also that $\sU^{\var}_G = \sU^{\var}_{\sD(G)}$, and we will use this fact without comment.

\begin{lemma}\label{lemma:properties-of-unipotent-variety}
Let $G$ be a connected reductive group over a field $k$ of characteristic $p \geq 0$.
\begin{enumerate}
    \item\label{item:unipotent-variety-commutes-with-field-extension} The formation of $\sU^{\var}$ commutes with any field extension of $k$, and $\sU^{\var}$ is geometrically reduced and generically smooth.
    \item\label{item:unipotent-variety-separable-central-isogeny} If $f: G \to G'$ is a central isogeny of connected reductive $k$-groups, then the induced map $\sU^{\var}_G \to \sU^{\var}_{G'}$ is bijective on geometric points. If $f$ is separable, then this map is an isomorphism.
    \item\label{item:unipotent-variety-simply-connected} If $p \nmid |\pi_1(\sD(G))|$, then $\sU^{\var} = \chi^{-1}(\chi(1))$ as schemes, where $\chi: \sD(G) \to \sD(G)/\!/\sD(G)$ is the Steinberg morphism.
    \item\label{item:unipotent-variety-is-normal} If $p \nmid |\pi_1(\sD(G))|$, then the (open) regular locus $\sU^{\var}_{\rm{reg}}$ is smooth and $\sU^{\var}$ is normal and Cohen-Macaulay.
\end{enumerate}
\end{lemma}

\begin{proof}
First, note that formation of $\sU^{\var}_G$ commutes with all separable algebraic extensions of $k$ by construction. Thus by passing to a finite separable extension of $k$ we may assume that $G$ is split. Passing to $\sD(G)$, we may also assume that $G$ is semisimple. Let $k_0 \subset k$ be the prime field, so that $k_0$ is perfect. Let $G_0$ be the split connected semisimple $k_0$-group such that $(G_0)_k = G$. Note that $\sU^{\var}_{G_0}$, being reduced over a perfect field, is geometrically reduced and generically smooth. Thus the natural map $(\sU^{\var}_{G_0})_k \to \sU^{\var}_G$ is a closed embedding of integral schemes of the same dimension $\dim G - \rk G$, and hence it is an isomorphism. In particular, $\sU^{\var}_G$ is generically smooth. If $k'$ is a field extension of $k$, we have
\[
(\sU^{\var}_G)_{k'} = (\sU^{\var}_{G_0})_{k'} = \sU^{\var}_{G_{k'}}
\]
since the extension $k'/k_0$ is separable. This shows that formation of $\sU^{\var}_G$ commutes with the field extension $k'/k$, as desired. \smallskip

In light of the previous paragraph, we may and do assume from now on that $k$ is algebraically closed. Now let $f: G \to G'$ be a central isogeny. The induced morphism $\sU^{\var}_G \to \sU^{\var}_{G'}$ is finite, and it is easily checked to be bijective on $\overline{k}$-points, so we will assume from now on that $f$ is smooth. This assumption implies that $f^{-1}(\sU^{\var}_{G'})$ is a reduced closed subscheme of $G$. On the level of $k$-points we have $f^{-1}(\sU^{\var}_{G'}) = (\ker f) \cdot \sU^{\var}_G$, so reducedness of $f^{-1}(\sU^{\var}_{G'})$ shows that this is true schematically. By centrality of $\ker f$ and uniqueness of the Jordan decomposition, we see that
\[
f^{-1}(\sU^{\var}_{G'}) = \bigsqcup_{c \in (\ker f)(k)} c \cdot \sU^{\var}_G.
\]
In particular, the natural map $\sU^{\var}_G \to f^{-1}(\sU^{\var}_{G'})$ is an isomorphism onto a connected component of $f^{-1}(\sU^{\var}_{G'})$, so $\sU^{\var}_G \to \sU^{\var}_{G'}$ is etale and $\sU^{\var}_G(k) \to \sU^{\var}_{G'}(k)$ is bijective. It follows that the morphism is radicial, and hence \cite[IV\textsubscript{4}, Thm.\ 17.9.1]{EGA} shows that $\sU^{\var}_G \to \sU^{\var}_{G'}$ is an isomorphism. \smallskip

Now note that $\chi^{-1}(\chi(1))(k) = \sU^{\var}(k)$ by Lemma~\ref{lemma:properties-of-chi}(\ref{item:chi-unipotence}). By Lemma~\ref{lemma:adjoint-quotient-smooth-near-1}, $G /\!/ G = T /\!/ W$ is smooth in a neighborhood of $\chi(1)$, so $\chi(1)$ is locally cut out by $r = \rk G$ functions. Since $\dim \sU^{\var} = \dim G - r$, it follows that $\chi^{-1}(\chi(1))$ is Cohen-Macaulay. By Lemma~\ref{lemma:properties-of-chi}(\ref{item:regularity-and-smoothness}) and (\ref{item:irreducible-fibers}), we see that the smooth locus of $\chi^{-1}(\chi(1))$ is equal to $\sU^{\var}_{\rm{reg}}$ and hence has complementary codimension $\geq 2$. By Serre's criterion for normality, it follows that $\chi^{-1}(\chi(1))$ is (geometrically) normal and in particular reduced, proving (\ref{item:unipotent-variety-simply-connected}) and (\ref{item:unipotent-variety-is-normal}).
\end{proof}

\begin{remark}
We show in \cite[Theorem 1.5]{Cotner-non-etale} that Lemma~\ref{lemma:properties-of-unipotent-variety}(\ref{item:unipotent-variety-is-normal}) holds even when $p$ divides $|\pi_1(\sD(G))|$, although Lemma~\ref{lemma:properties-of-unipotent-variety}(\ref{item:unipotent-variety-simply-connected}) \textit{always fails} in this case, in the rather strong sense that $\chi^{-1}(\chi(1))$ is generically non-reduced. However, it is always true that $\sU^{\var} = \chi^{-1}(\chi(1))_{\red}$, as one sees by Lemma~\ref{lemma:properties-of-unipotent-variety}(\ref{item:unipotent-variety-simply-connected}) and the fact that $\sU^{\var}$ and $\chi^{-1}(\chi(1))_{\red}$ are both the schematic images in $G$ under the quotient map of the analogous objects in the universal cover of $\sD(G)$.
\end{remark}

We move on now from the unipotent variety to the unipotent scheme.

\begin{theorem}[Unipotent schemes]\label{theorem:unipotent-scheme}
There exists a unique assignment to each scheme $S$ and each reductive $S$-group scheme $G$ a closed subscheme $\sU_G$ of $\sD(G)$ satisfying the following conditions.
\begin{enumerate}
    \item If $\overline{s}$ is a geometric point of $S$, then $\sU_G(k(\overline{s}))$ is the set of unipotent elements of $G(k(\overline{s}))$.
    \item If $S'$ is an $S$-scheme, then there is an equality $\sU_{G_{S'}} = (\sU_G)_{S'}$ of closed subschemes of $G_{S'}$.
    \item If $S$ is an integral scheme of generic characteristic $0$, then $\sU_G$ is reduced.
\end{enumerate}
Moreover, $\sU_G$ enjoys the following extra properties.
\begin{itemize}
    \item $\sU_G \subset \chi^{-1}(\chi(1))$, where $\chi: \sD(G) \to \sD(G)/\!/\sD(G)$ is the Steinberg morphism. If $|\pi_1(\sD(G))|$ is invertible on $S$, then equality holds.
    \item $\sU_G$ is $S$-flat and stable under the action of $\Aut_{G/S}$.
    \item If $f: G \to G'$ is a central isogeny of reductive $S$-group schemes, then $f$ restricts to an $S$-morphism $\sU_G \to \sU_{G'}$.
    \item If $\overline{s}$ is a geometric point of $S$ such that $\chara k(\overline{s})$ does not divide $|\pi_1(\sD(G))|$, then $(\sU_G)_{k(\overline{s})}$ is normal.
\end{itemize}
In particular, if $S$ is normal and $\chara k(s)$ does not divide $|\pi_1(\sD(G_s))|$ for every $s \in S$, then $\sU_G$ is normal.
\end{theorem}

\begin{proof}
The remainder of this section will be spent proving Theorem~\ref{theorem:unipotent-scheme}. First we prove the uniqueness of this assignment: suppose $\sU$ and $\sU'$ are two assignments satisfying the above conditions. Let $S$ be a scheme, let $G$ be a reductive $S$-group scheme, and let $S' \to S$ be an fppf cover such that $G_{S'}$ is $S'$-split. To check that $\sU_G = \sU'_G$, it suffices to check that $\sU_{G_{S'}} = \sU'_{G_{S'}}$ (using (2)). Note that $G_{S'} = \cG_{S'}$ for some split reductive group scheme $\cG$ over $\bZ$, so by (2) we may pass from $G$ to $\cG$ to assume that $S = \Spec \bZ$. In this case, uniqueness is simple: (1) determines $\sU_G$ as a \textit{subset} of the ringed space $G$, and (3) then determines $\sU_G$ as a \textit{closed subscheme} of $G$. So it suffices to establish existence and properties. \smallskip

Suppose first that we have defined $\sU_G$ for every \textit{split} reductive group scheme $G$, and suppose that this assignment satisfies the desired properties. If $S$ is an arbitrary scheme and $G$ is a reductive $S$-group scheme, let $S'$ be an fppf $S$-scheme such that $G_{S'}$ is $S'$-split. Then $\sU_{G_{S'}}$ is defined, and by hypothesis it is preserved by all automorphisms of $G_{S'}$, even after base change in $S$. Thus by descent \cite[Exp.\ VIII, Cor.\ 1.9]{SGA1}, there is a unique closed subscheme $\sU_G$ of $G$ such that $(\sU_G)_{S'} = \sU_{G_{S'}}$. By construction, formation of $\sU_G$ commutes with base change. Since all of the remaining properties asserted in Theorem~\ref{theorem:unipotent-scheme} may be checked after base change on $S$, we have reduced to proving the existence of this assignment (and its properties) for split reductive group schemes $G$. \smallskip

Next suppose that we have defined $\sU$ for every split reductive group scheme over $\bZ$, and suppose that this assignment satisfies the properties in (1), (3), and the second bullet point of Theorem~\ref{theorem:unipotent-scheme}. Let $S$ be a scheme, and let $G$ be a split reductive $S$-group scheme, so $G \cong \cG_S$ for some split reductive group scheme $\cG$ over $\bZ$. We let $\sU_G = (\sU_{\cG})_S$. To show that this is well-defined, one must show that $(\sU_{\cG})_S$ is preserved by all automorphisms of $\cG_S$. Any automorphism of $\cG_S$ certainly preserves $(\sU_{\cG})_S$ as a \textit{set}, and if $(\sU_{\cG})_S$ were reduced it would follow that such an automorphism preserves $(\sU_{\cG})_S$ as a \textit{scheme} as well. We will thus reduce to the case that $(\sU_{\cG})_S$ is reduced. \smallskip

Working (fpqc-)locally and spreading out, we may and do assume that $S = \Spec A$ for a complete noetherian local ring $A$. By the Cohen structure theorem, we have $A = A_0/I$ for some complete regular local ring $A_0$ of generic characteristic $0$. Since $\Aut_{\cG/\bZ}$ is smooth, any automorphism of $\cG_A$ lifts to an automorphism of $\cG_{A_0}$, so we may pass from $A$ to $A_0$ to assume that $S$ is integral of generic characteristic $0$. In particular, $(\sU_{\cG})_S$ is reduced by property (3), and it follows from the previous paragraph that $(\sU_{\cG})_S$ is preserved by all automorphisms of $\cG_S$. \smallskip

To summarize, we have now shown that to define $\sU_G$ for general $G$, it is enough to define it for all split reductive $\bZ$-group schemes $G$ and to demonstrate properties (1), (3), and the second bullet point of Theorem~\ref{theorem:unipotent-scheme} for such $G$. For existence in this case, let $\sU$ be the schematic closure of the unipotent variety $\sU_{G_{\bQ}}$ in $G$. Since $\sU_{G_{\bQ}}$ lies in $\sD(G_{\bQ})$, it follows that $\sU$ is a closed subscheme of $\sD(G)$. Since $\sU_{G_{\bQ}}$ is stable under the action of $\Aut_{G_{\bQ}/\bQ}$ and $\Aut_{G/\bZ}$ is flat, it follows that $\sU$ is stable under the action of $\Aut_{G/\bZ}$. Since $\bZ$ is Dedekind, $\sU$ is $\bZ$-flat, and so its fibers are all of the same dimension. Moreover, for each geometric point $\overline{s}$ of $\Spec \bZ$, $\sU(k(\overline{s}))$ consists of unipotent elements of $G(k(\overline{s}))$: to check this, one may use a $\bZ$-embedding of $G$ into $\GL_n$ to reduce to the case $G = \GL_n$, in which case it follows from considerations with eigenvalues. Because all of the fibral unipotent varieties of $G$ are irreducible of the same dimension, flatness of $\sU$ implies that $\sU(k(\overline{s}))$ is \textit{equal} to the set of unipotent elements of $G(k(\overline{s}))$. Note that we have established (1) and the second bullet point of Theorem~\ref{theorem:unipotent-scheme}. The third bullet point is also not difficult, because for split $G$ every central isogeny is already defined over $\bZ$ and the unipotent scheme is reduced over $\bZ$. \smallskip

Let $S$ be an integral scheme of generic characteristic $0$, and let $G$ be a reductive $\bZ$-group scheme. Since $\sU_{G_{\bQ}}$ is geometrically reduced by Lemma~\ref{lemma:properties-of-unipotent-variety}, the generic fiber of $\sU_{G_S}$ is reduced. Since moreover $\sU_{G_S}$ is $S$-flat, it follows that $\sU_{G_S}$ is reduced, proving (3). Thus we have now \textit{defined} $\sU$ for arbitrary reductive group schemes over arbitrary schemes. Note also that the final statement of the theorem follows from the fourth bullet point and \cite[23.9, Cor.]{Matsumura}. \smallskip

It remains to establish the first and fourth bullet points of Theorem~\ref{theorem:unipotent-scheme}. For the remainder of this section, $S$ is a scheme and $G$ is a reductive $S$-group scheme. We define $\chi: \sD(G) \to \sD(G)/\!/\sD(G)$ as in the beginning of this section, and we let $\sU'_G = \chi^{-1}(\chi(1))$. We will show that $\sU_G = \sU'_G$ if $|\pi_1(\sD(G))|$ is invertible on $S$, and in any case $\sU_G \subset \sU'_G$.

\begin{theorem}\label{theorem:adjoint-quotients-agree-group-case}\cite[Thm.\ 4.1]{Lee-adjoint}
Let $S$ be a scheme, and let $G$ be a reductive $S$-group scheme. If $S'$ is an $S$-scheme, then the natural $S'$-morphism $G_{S'}/\!/G_{S'} \to (G/\!/G)_{S'}$ is an isomorphism. Moreover, if $G$ admits a maximal $S$-torus $T$ with Weyl group $W$, then the natural map $T/\!/W \to G/\!/G$ is an isomorphism.
\end{theorem}

\begin{proof}
We note that \cite[Thm.\ 4.1]{Lee-adjoint} only considers the case that $S$ and $S'$ are both affine, but this is not essential because the formation of GIT quotients commutes with flat base change (in particular, passage to Zariski open covers).
\end{proof}

Thus we find that the formation of $\sU'_G$ commutes with base change in $S$. To show $\sU_G \subset \sU'_G$, it therefore suffices to pass to an fppf cover of $S$ to assume that $G$ is split, and by base change again we may assume $S = \Spec \bZ$. But now Lemma~\ref{lemma:properties-of-unipotent-variety} shows that if $k$ is a field then $\sU'_G(k)$ is the set of unipotent elements in $\sD(G)(k)$, which is also the set of unipotent elements in $G(k)$. Thus $\sU_G$ and $\sU'_G$ have the same set of points, and since $\sU_G$ is reduced it follows that $\sU_G \subset \sU'_G$.

\begin{lemma}\label{lemma:chi-flat}
Let $S$ be a scheme, and let $G$ be a semisimple $S$-group scheme. If $|\pi_1(G)|$ is invertible on $S$ then $\chi$ is flat in a Zariski open neighborhood of $\chi(1)$, and $\chi^{-1}(\chi(1))$ has geometrically normal fibers.
\end{lemma}

\begin{proof}
The statement concerning fibers follows from Lemma~\ref{lemma:properties-of-unipotent-variety}. To show flatness in a neighborhood of $\chi(1)$, we note that openness of the flat locus \cite[IV\textsubscript{3}, Thm.\ 11.3.1]{EGA} and the fibral flatness criterion \cite[IV\textsubscript{3}, Thm.\ 11.3.10]{EGA} reduce us to the case $S = \Spec k$ for an algebraically closed field $k$. In this case, Theorem~\ref{theorem:adjoint-quotients-agree-group-case} and Lemma~\ref{lemma:properties-of-chi} show that every geometric fiber of $\chi$ is irreducible of the same dimension, so by Miracle Flatness \cite[Thm.\ 23.1]{Matsumura} it suffices to show that $G/\!/G$ is smooth in a neighborhood of $\chi(1)$. By Theorem~\ref{theorem:adjoint-quotients-agree-group-case}, if $T$ is a (split) maximal $k$-torus of $G$ with Weyl group $W$, then there is an isomorphism $G/\!/G \cong T/\!/W$. Finally, by Lemma~\ref{lemma:adjoint-quotient-smooth-near-1}, we see that $G/\!/G$ is indeed smooth in a neighborhood of $\chi(1)$.
\end{proof}

Finally we are ready to conclude the proof of Theorem~\ref{theorem:unipotent-scheme}. By Lemma~\ref{lemma:chi-flat}, it remains to show that $\sU_G = \sU'_G$ if $|\pi_1(\sD(G))|$ is invertible on $S$. Using Theorem~\ref{theorem:adjoint-quotients-agree-group-case}, we may and do assume $S = \Spec \bZ$. Now Lemma~\ref{lemma:chi-flat} shows that in this case $\chi^{-1}(\chi(1))$ is flat with geometrically normal fibers, whence it is normal by \cite[23.9, Cor.]{Matsumura} and in particular reduced. Since $\sU_G \subset \sU'_G$ and both closed subschemes of $G$ have the same points, we obtain equality.
\end{proof}

\begin{remark}
Unfortunately, we do not know whether the unipotent scheme is functorial in $G$; that is, if $f: G \to G'$ is a homomorphism of reductive $S$-group schemes, does $f$ restrict to a morphism $\sU_G \to \sU_{G'}$? This is certainly true on the level of points, so if $\sU_G$ is \textit{reduced} then this is true. Moreover, if $|\pi_1(\sD(G'))|$ is invertible on $S$ then $\sU_{G'} = \chi_{G'}^{-1}(\chi_{G'}(1))$ and $\sU_G \subset \chi_G^{-1}(\chi_G(1))$. Clearly $f$ restricts to a morphism $\chi_G^{-1}(\chi_G(1)) \to \chi_{G'}^{-1}(\chi_{G'}(1))$, so again $f$ restricts to a morphism $\sU_G \to \sU_{G'}$ in this case. Thus the only remaining case of interest is the case that neither $|\pi_1(\sD(G))|$ nor $|\pi_1(\sD(G'))|$ are invertible on the base. The same remark applies to the nilpotent scheme defined in the next section.
\end{remark}

\begin{cor}\label{corollary:etale-central-isogeny-iso-unipotent-schemes}
Let $S$ be a scheme and let $f: G \to G'$ is an etale isogeny of reductive $S$-group schemes. If $|\pi_1(\sD(G))|$ is invertible on $S$, then the induced morphism $\sU_G \to \sU_{G'}$ is an isomorphism.
\end{cor}

\begin{proof}
Because of the flatness and fibral assertions in Theorem~\ref{theorem:unipotent-scheme}, we may apply the fibral isomorphism criterion to reduce to the already-proven Lemma~\ref{lemma:properties-of-unipotent-variety}(\ref{item:unipotent-variety-separable-central-isogeny}).
\end{proof}

\subsection{The nilpotent scheme}\label{subsection:nilp-sch}

We may define the \textit{nilpotent variety} $\sN^{\var} = \sN_G^{\var}$ of $G$ in a manner completely similar to the definition of the unipotent variety: if $k$ is separably closed, then $\sN^{\var}$ is the Zariski closure of the set of nilpotent elements in $\fg$. By Galois descent, we may define $\sN^{\var}$ over an arbitrary field $k$. In \cite[Prop.\ 2.1]{Springer-isomorphism} it is proved that $\sN^{\var}$ is an irreducible closed subscheme of $G$ of dimension $\dim G - \rk G$. We remark that $\sN_G^{\var} = \sN_{\sD(G)}^{\var}$, and we will use this fact without comment.

\begin{lemma}\label{lemma:properties-of-nilpotent-variety}
Let $G$ be a connected reductive group over a field $k$ of characteristic $p \geq 0$.
\begin{enumerate}
    \item The formation of $\sN^{\var}$ commutes with any field extension of $k$, and $\sN^{\var}$ is geometrically reduced and generically smooth.
    \item\label{item:nilpotent-variety-separable-central-isogeny} If $f: G \to G'$ is a central isogeny of connected reductive groups, then the induced map $\sN_G^{\var} \to \sN_{G'}^{\var}$ is bijective on geometric points. If $f$ is moreover separable, then this map is an isomorphism.
    \item\label{item:nilpotent-variety-chi-description} If $p$ is not a torsion prime for $G$, $p \nmid |\pi_1(\sD(G))|$, and $p \neq 2$ if $G$ has a simple factor of type $\rm{C}_n$ ($n \geq 1$), then $\sN^{\var} = \chi^{-1}(\chi(0))$.
    \item\label{item:nilpotent-variety-is-normal} If $p$ is good for $G$ and $p \nmid |\pi_1(\sD(G))|$, then $\sN^{\var}$ is normal.
\end{enumerate}
\end{lemma}

\begin{proof}
The proofs of the first two points are entirely similar to the proofs of the analogous points in Lemma~\ref{lemma:properties-of-unipotent-variety}, and we omit them. For the rest, we may and do assume that $G$ is semisimple and simply connected. To prove (\ref{item:nilpotent-variety-chi-description}), let $T$ be a split maximal torus of $G$ with Weyl group $W$ and Lie algebra $\ft$. By \cite[7.12]{Jantzen-nilpotent}, the restriction map $\Sym(\fg^*)^G \to \Sym(\ft^*)^W$ is an isomorphism under our hypotheses. By \cite[Thm.\ 3, Corollaire]{Demazure}, $\Sym(\ft^*)^W$ is a polynomial ring, so using Lemma~\ref{lemma:properties-of-chi-nilpotent}(\ref{item:Lie-chi-irreducible-fibers}) we see that $\chi^{-1}(\chi(0))$ is Cohen-Macaulay. We claim that the smooth locus of $\chi^{-1}(\chi(0))$ is open and nonempty. Away from type A, note that $G$ satisfies hypotheses (H1)-(H3) from \cite[2.9]{Jantzen-nilpotent}, so \cite[7.14]{Jantzen-nilpotent} explains that $\chi$ is smooth at an element $X \in \fg$ if and only if $X$ is regular, so the result follows (since $\fg$ contains a regular nilpotent element by \cite[Lem.\ 3.1.1]{Riche-universal-centralizer}). In type A, i.e., $G \cong \SL_n$ for some $n$, the claim follows from the very concrete calculations of \cite[7.2]{Jantzen-nilpotent}. Since $\chi^{-1}(\chi(0))$ is Cohen-Macaulay and generically smooth, Serre's criterion for reducedness shows that $\chi^{-1}(\chi(0))$ is reduced. By Lemma~\ref{lemma:properties-of-chi-nilpotent}(\ref{item:Lie-chi-nilpotence}) we have $\chi^{-1}(\chi(0))(k) = \sN^{\var}(k)$, so indeed $\chi^{-1}(\chi(0)) = \sN^{\var}$. The final point (\ref{item:nilpotent-variety-is-normal}) follows from \cite[8.5, Cor.]{Jantzen-nilpotent}.
\end{proof}

Again, we move on from the nilpotent variety to the nilpotent scheme.

\begin{theorem}[Nilpotent version]\label{theorem:nilpotent-scheme}
There exists a unique assignment to each scheme $S$ and each reductive $S$-group scheme $G$ with $\fg \coloneqq \Lie G$ and $\fg' \coloneqq \Lie \sD(G)$ a closed subscheme $\sN_G$ of $\fg'$ satisfying the following conditions.
\begin{enumerate}
    \item If $\overline{s}$ is a geometric point of $S$, then $\sN_G(k(\overline{s}))$ is the set of nilpotent elements of $\fg(k(\overline{s}))$.
    \item If $S'$ is an $S$-scheme, then there is an equality $\sN_{G_{S'}} = (\sN_G)_{S'}$ of closed subschemes of $\fg_{S'}$.
    \item If $S$ is an integral scheme of generic characteristic $0$, then $\sN_G$ is reduced.
\end{enumerate}
Moreover, $\sN_G$ satisfies the following extra conditions.
\begin{itemize}
    \item $\sN_G \subset \chi^{-1}(\chi(0))$, where $\chi: \fg' \to \fg'/\!/G$ is the Steinberg morphism. If $|\pi_1(\sD(G))|$ is invertible on $S$ and $\chara k(s)$ is good for $G_s$ for every $s \in S$, then equality holds.
    \item $\sN_G$ is $S$-flat and stable under the action of $\Aut_{G/S}$.
    \item If $f: G \to G'$ is a central isogeny of reductive $S$-group schemes, then $f$ restricts to an $S$-morphism $\sN_G \to \sN_{G'}$.
    \item If $\overline{s}$ is a geometric point of $S$ such that $\chara k(\overline{s})$ is good for $G_{\overline{s}}$ and does not divide $|\pi_1(\sD(G))|$, then $(\sN_G)_{k(\overline{s})}$ is normal.
\end{itemize}
In particular, if $S$ is normal and $\chara k(s)$ is good for $G_{\overline{s}}$ and does not divide $|\pi_1(\sD(G_s))|$ for every $s \in S$, then $\sN_G$ is normal.
\end{theorem}

\begin{proof}
Uniqueness of the assignment is established just as in the proof of Theorem~\ref{theorem:unipotent-scheme}. Completely similar arguments to those in the proof of Theorem~\ref{theorem:unipotent-scheme} also reduce one to constructing $\sN_G$ for split reductive $\bZ$-group schemes $G$ and to demonstrate the properties (except (2)) for such $\sN_G$. As before, we define $\sN_G$ for such $G$ to be the schematic closure of $\sN_{G_\bQ}$ in $\fg$, and we deduce (1), (3), and the second and third bullet points as before. Again, the final statement of the theorem follows from the fourth bullet point and \cite[23.9, Cor.]{Matsumura}. \smallskip

It remains to establish the first and fourth bullet points. For the remainder of this section, $S$ is a scheme and $G$ is a reductive $S$-group scheme. We define $\chi: \fg' \to \fg'/\!/\sD(G)$ to be the Steinberg morphism, and we let $\sN_G' = \chi^{-1}(\chi(0))$. We will show that $\sU_G = \sU_G'$ if $|\pi_1(\sD(G))|$ is invertible on $S$ and $\chara k(s)$ is good for $G_s$ for all $s \in S$, and in any case $\sN_G \subset \sN_G'$. \smallskip

To prove the remaining claims, we may and do assume that $G$ is semisimple, and since $|\pi_1(G)|$ is invertible on $S$ we may and do pass to the universal cover to assume that $G$ is simply connected. By passing to simply factors, we may and do assume that $G$ is simple. Because of the irritating caveat in Lemma~\ref{lemma:properties-of-nilpotent-variety}(\ref{item:nilpotent-variety-chi-description}), we will first deal separately with the case that $G$ is of type A, i.e., $G \cong \SL_n$ for some $n$. Note that $G$ and $\fg$ are both naturally closed subschemes of $\gl_n$, and the $G$-equivariant morphism $\rho: \gl_n \to \gl_n$, $g \mapsto g - 1$ sends $\sU$ to a closed subscheme $\sN'$ of $\fg$ which is equal to the nilpotent variety on fibers. Since $\sU$ is $\bZ$-flat, normal, and stable under $G$-conjugation, the same is true of $\sN'$, and it is clear that $\sN' = \sN$. In all, we have proved the result when $G \cong \SL_n$. Thus from now on we may and do assume that every residue characteristic of $R$ is either $0$ or a \textit{very good} prime for $G$. In this case, we will use the following theorem from \cite{Bouthier-Cesnavicius}.

\begin{theorem}\label{theorem:adjoint-quotients-agree-lie-algebra-case}\cite[Thm.\ 4.1.10, Prop.\ 4.1.14]{Bouthier-Cesnavicius}
Let $S$ be a scheme, and let $G$ be a root-smooth reductive $S$-group scheme. If $T$ is a maximal $S$-torus of $G$ with Lie algebra $\ft$ and Weyl group $W \coloneqq N_G(T)/T$, then the natural $S$-morphism $\ft/W \to \fg/\!/G$ is an isomorphism. If $\chara k(s)$ is a non-torsion prime for $G_s$ for every $s \in S$, then the formation of $\fg/\!/G$ commutes with base change in $S$, and etale-locally on $S$ it is an affine space of dimension $\rk G$.
\end{theorem}

By \cite[4.1.1]{Bouthier-Cesnavicius}, $G$ is root-smooth whenever $\chara k(s)$ is very good for $G_s$ for all $s \in S$. \smallskip

The idea now is very similar to the unipotent case. By Theorem~\ref{theorem:adjoint-quotients-agree-lie-algebra-case}, the formation of $\sN_G'$ commutes with base change in $S$, so to show $\sN_G \subset \sN_G'$ it suffices to pass to an fppf cover of $S$ to assume that $G$ is split, and by base change again we may assume $S = \Spec \bZ$. But now Lemma~\ref{lemma:properties-of-nilpotent-variety}(\ref{item:nilpotent-variety-separable-central-isogeny}) shows that if $k$ is a field then $\sN_G'(k)$ is the set of nilpotent elements in $\fg'(k)$, which is also the set of nilpotent elements in $\fg(k)$. Thus $\sN_G$ and $\sN_G'$ have the same set of points, and since $\sN_G$ is reduced it follows that $\sN_G \subset \sN_G'$. \smallskip

It remains to show that $\sN_G = \sN_G'$ if $\chara k(s)$ is good for $G_s$ for every $s \in S$. By Theorem~\ref{theorem:adjoint-quotients-agree-lie-algebra-case}, we may and do assume $S$ is an open subscheme of $\Spec \bZ$. Using Lemma~\ref{lemma:properties-of-chi-nilpotent}(\ref{item:Lie-chi-irreducible-fibers}) and the smoothness of $\fg/\!/G$ established in Theorem~\ref{theorem:adjoint-quotients-agree-lie-algebra-case}, it follows from Miracle Flatness \cite[Thm.\ 23.1]{Matsumura} that $\chi$ is flat. Using Lemma~\ref{lemma:properties-of-chi-nilpotent}(\ref{item:nilpotent-variety-is-normal}), we find that $\sN_G'$ is flat with geometrically normal fibers, whence it is normal by \cite[23.9, Cor.]{Matsumura}. In particular, $\sN_G'$ is reduced. Since $\sN_G \subset \sN_G'$ and both reduced closed subschemes of $G$ have the same points, we obtain equality.
\end{proof}

\section{The Springer isomorphism}\label{section:springer-iso}

The main point of this section is to prove the following theorem.

\begin{theorem}\label{theorem:relative-springer-isomorphism}
Let $S$ be a scheme and let $G$ be a reductive $S$-group scheme such that
\begin{enumerate}
    \item $|\pi_1(\sD(G))|$ is invertible on $S$,
    \item for each $s \in S$, $\chara k(s)$ is good for $G_s$.
\end{enumerate}
Suppose further that either
\begin{enumerate}[label=(\alph*)]
    \item\label{item:springer-affine} $S$ is affine, or
    \item\label{item:springer-small-cases} if $G \cong \mathrm{R}_{S'/S}(G')$ for a finite etale cover $S' \to S$ and a reductive $S'$-group scheme $G'$ with absolutely simple fibers as in \cite[Exp.\ XXIV, 5.3 and Prop.\ 5.10(i)]{SGA3III}, then for every $s' \in S'$,
    \begin{itemize}
        \item either $\chara k(s') \neq 2$ or the quasi-split inner form of $G_{s'}$ is split, and
        \item either $\chara k(s') \neq 3$ or $G_{s'}$ is not of type $\mathrm{D}_4$.
    \end{itemize}
\end{enumerate}
Then there is a $G$-equivariant isomorphism $\rho: \sU_G \to \sN_G$, where $\sU_G$ is the unipotent scheme of $G$ and $\sN_G$ is the nilpotent scheme of $G$. If $u \in \sU_G(S)$ and $X \in \sN_G(S)$ are fiberwise regular with $X \in (\Lie Z_G(u))(S)$, then there is a unique such $\rho$ with $\rho(u) = X$.
\end{theorem}

Any isomorphism in Theorem~\ref{theorem:relative-springer-isomorphism} is called a \textit{Springer isomorphism}. The hypotheses of Theorem~\ref{theorem:relative-springer-isomorphism} are essentially optimal, as is in \cite[Proposition 5.9, Theorem 5.14]{Cotner-non-etale}. We remark that fiberwise regular sections of $\sU_G$ and $\sN_G$ exist etale-locally on $S$, but they might not exist over $S$.\smallskip

Much of the difficulty in the proof of Theorem~\ref{theorem:relative-springer-isomorphism} comes from showing that $\sU_G$ and $\sN_G$ have good properties, which was dealt with in Section~\ref{section:unip-nilp}. In Section~\ref{subsection:split}, we will show that Theorem~\ref{theorem:relative-springer-isomorphism} holds whenever sections $u$ and $X$ exist (in particular whenever $G$ is split) by constructing a $G$-equivariant isomorphism $\sU_{\rm{reg}} \to \sN_{\rm{reg}}$. Following a digression on pinning-preserving automorphisms in Section~\ref{subsection:digression}, a twisting argument in Section~\ref{subsection:general} will be used to pass from the split case to the general case provided \ref{item:springer-small-cases} holds; the cases of type $\mathrm{A}_n$ ($n \geq 2$) and $\mathrm{D}_4$ will be dealt with separately in Appendix~\ref{appendix}. Finally, Section~\ref{subsection:kawanaka} will show that Springer isomorphisms behave reasonably when restricted to the unipotent radical of a Borel, generalizing \cite[Prop.\ 4.6]{Jay-GGGR}.

\subsection{A special case}\label{subsection:split}

We deal first with the case that there exist sections fiberwise regular $u \in \sU_G(S)$ and $X \in \sU_G(S)$ such that $X \in (\Lie Z_G(u))(S)$.

\begin{lemma}\label{lemma:centralizers-equal}
Let $S$ be a scheme, and let $G$ be a reductive $S$-group scheme as in Theorem~\ref{theorem:relative-springer-isomorphism}. If $u \in \sU_G(S)$ and $X \in \sN_G(S)$ are fiberwise regular and $X \in \Lie Z_G(u)$, then $Z_G(u) = Z_G(X)$.
\end{lemma}

\begin{proof}
By Theorem~\ref{theorem:flat-centralizer} and Corollary~\ref{corollary:commutative-centralizer}, $Z_G(u)$ and $Z_G(X)$ are commutative flat closed $S$-subgroup schemes of $G$. Since $Z_G(u)$ is commutative, $Z_G(u) \subset Z_G(X)$, and in fact equality holds because $Z_G(X)$ is commutative and in particular centralizes $u$.
\end{proof}

If $S$ is a scheme and $G$ is reductive $S$-group scheme, we define $\sU_{\rm{reg}} = \sU_{G, \rm{reg}}$ to be the subset of $\sU = \sU_G$ defined by
\[
\sU_{\rm{reg}} = \{u \in \sU: u \text{ is regular in } G(k(u))\}.
\]
It is open in $\sU$ by upper semicontinuity of fiber dimension, so we may consider $\sU_{\rm{reg}}$ as an open subscheme of $\sU$. In fact, $\sU_{\rm{reg}}$ represents the subfunctor of $\sU_G$ consisting of fiberwise regular sections. Similarly, we define $\sN_{\rm{reg}}$ as a subset of $\sN = \sN_G$. Again, $\sN_{\rm{reg}}$ is open in $\sN$.

\begin{lemma}\label{lemma:regular-orbit-spaces}
With notation as in Lemma~\ref{lemma:centralizers-equal}, the fppf sheaf quotient $G/Z_G(u) = G/Z_G(X)$ is represented by a finitely presented separated $S$-scheme, and the natural orbit $S$-morphisms $G/Z_G(u) \to \sU$ and $G/Z_G(X) \to \sN$ are open embeddings which identify $G/Z_G(u)$ and $G/Z_G(X)$ with $\sU_{\rm{reg}}$ and $\sN_{\rm{reg}}$, respectively.
\end{lemma}

\begin{proof}
First, note that $G/Z_G(u)$ is automatically a finite type separated algebraic space by a general theorem of Artin \cite[Cor.\ 6.4]{Artin-stacks}. Moreover, the orbit $R$-morphism $G/Z_G(u) \to \sU$ is a monomorphism, so a result of Knutson \cite[II, Thm.\ 6.15]{Knutson-algebraic-spaces} shows that $G/Z_G(u)$ is representable by a scheme. The same argument applies to $G/Z_G(X)$.\smallskip

We will only prove that $G/Z_G(u) \to \sU$ is an open embedding, the other claim being similar. Recall that an etale monomorphism of finite schemes is automatically an open embedding by \cite[Exp.\ I, Thm.\ 5.1]{SGA1}. Note that $G/Z_G(u)$ and $\sU$ are both flat, so etaleness is a fibral condition by \cite[IV\textsubscript{4}, Prop.\ 17.8.2]{EGA} and we may and do assume $S = \Spec k$ for a field $k$. In this case, the open embedding claim is a standard fact about orbits for actions of algebraic groups. \smallskip

To conclude, observe that the image of $G/Z_G(u)$ is equal to $\sU_{\rm{reg}}$; this can be checked on the level of geometric points, where it is equivalent to the statement that any two regular unipotent elements are conjugate \cite[Thm.\ 3.3]{SteinbergReg}.
\end{proof}

\begin{proof}[Proof of Theorem~\ref{theorem:relative-springer-isomorphism} if $u$ and $X$ exist]
In view of Lemmas~\ref{lemma:centralizers-equal} and \ref{lemma:regular-orbit-spaces}, there exists a unique $G$-equivariant $S$-isomorphism $\rho_0: \sU_{\rm{reg}} \to \sN_{\rm{reg}}$ such that $\rho_0(u) = X$. Since $\sU_{\rm{reg}}$ and $\sN_{\rm{reg}}$ are universally schematically dense in $\sU$ and $\sN$, respectively, there is at most one extension of $\rho_0$ to a morphism $\rho: \sU \to \sN$, and if it exists then it is automatically $G$-equivariant. \smallskip

In view of the uniqueness, to show the existence of an extension $\rho$ of $\rho_0$ it suffices by spreading out and fpqc descent to assume that $S = \Spec A$ for a complete noetherian local ring $A$ with algebraically closed residue field $k$. We will further reduce to the case that $A$ is normal. By the Cohen structure theorem we have $A = A_0/I$ for some complete regular local ring $A_0$. Since $k$ is algebraically closed, $G$ is split, and thus $G \cong \cG_A$ for some split reductive $\bZ$-group schemes $\cG$. Since $B_0$ be a Borel $A_0$-subgroup of $G_0 \coloneqq \cG_{A_0}$ such that $(B_0)(A)$ contains $u$. If $U_0$ is the unipotent radical of $B_0$, then we can lift $u$ to a section $u_0 \in U_0(A_0)$. Note $u_0$ is automatically fiberwise regular, and $Z_{U_0}(u_0)$ is $A_0$-smooth by Theorem~\ref{theorem:flat-centralizer}. Note $X \in \Lie Z_{U_0}(u_0)(A)$, so by smoothness we may lift $X$ to $X_0 \in \Lie Z_{U_0}(u_0)(A_0)$. By Lemma~\ref{lemma:centralizers-equal}, we have $Z_{G_0}(u_0) = Z_{G_0}(X_0)$, and it suffices to show that there exists a $G_0$-equivariant isomorphism $\sU_{G_0} \to \sN_{G_0}$ sending $u_0$ to $X_0$. Thus we may pass from $A$ to $A_0$ to assume that $S$ is \textit{regular}, and in particular normal. \smallskip

Now by Theorem~\ref{theorem:unipotent-scheme}, since $S$ is normal also $\sU$ and $\sN$ are normal. Since $\sU_{\rm{reg}}$ is of complementary codimension $\geq 2$ in the normal scheme $\sU$ and $\sN$ is affine, $\rho_0$ therefore extends uniquely to a $G$-equivariant $S$-morphism $\rho: \sU \to \sN$ by Hartogs' lemma. Moreover, since $\sN$ is normal, $\rho$ must be an isomorphism. This completes the proof of Theorem~\ref{theorem:relative-springer-isomorphism}.
\end{proof}

\subsection{Digression: pinning-preserving automorphisms}\label{subsection:digression}

In view of the results of the previous section, to show that there is a Springer isomorphism for a given reductive $S$-group scheme $G$, it would be enough to show that there exist fiberwise regular sections $u \in \sU_G(S)$ and $X \in \sN_G(S)$ such that $X \in (\Lie Z_G(u))(S)$. These sections do not exist in general, but they do exist in the split case and we will use a twisting argument to deduce the general case.\smallskip

To make this twisting argument, we must first make a detour through fixed points of pinning-preserving automorphisms. It seems likely that all of these results can be extracted directly from the proof of \cite[Prop.\ 4.1]{Haines}, but we will give some details for the convenience of the reader.

\begin{prop}\label{prop:preservation-of-properties}
Let $k$ be a field of characteristic $p\geq 0$, and let $G$ be a connected semisimple $k$-group which is absolutely simple and simply connected. Let $A$ be a finite group of $k$-automorphisms of $G$ which preserve a common pinning. Suppose that either $p \neq 2$ or $G$ is not of type $\mathrm{A}_{2m}$ ($m \geq 1$). Then $G^A$ is a connected simple $k$-group and $p$ does not divide $|\pi_1(\sD(G^A))|$.\smallskip

Moreover, if $p$ is good for $G$ and either
\begin{itemize}
    \item $p \neq 2$, or
    \item $G$ is of type $\mathrm{D}_4$ and either $p \neq 3$ or $|A| \leq 2$,
\end{itemize}
then $p$ is good for $G^A$.
\end{prop}

\begin{proof}
We may and do assume that $k$ is algebraically closed. In this case, results in \cite{ALRR} (proved via considerations with the open cell, similar to those used in the proof of \cite[Prop.\ 4.1]{Haines}) show that under our hypotheses $(G^A)_{\rm{red}}^0$ is connected reductive with maximal torus $(T^A)^0_{\rm{red}}$, and moreover $G^A$ is smooth (resp.\ connected) if and only if $T^A$ is smooth (resp.\ connected), where $T$ is a maximal torus of $G$. Let $\Delta = \{\alpha_1, \dots, \alpha_n\}$ be a system of simple roots for the pair $(G, T)$. Since $G$ is simply connected, the natural map $\bG_m^n \to T$, $(x_i)_{1\leq i\leq n} \mapsto \prod_{i=1}^n \alpha_i^\vee(x_i)$ is an isomorphism. Moreover, $A$ acts on $T$ via permutation of the roots, so in fact $T^A$ is a torus in this case, of dimension equal to $|\Delta/A|$. Thus $G^A$ is indeed a connected reductive group.\smallskip

The proof of \cite[Prop.\ 4.1]{Haines} (see ``temporary assumption (ii)'', which is proven in the paragraph following \cite[(4.10)]{Haines}) shows that the root system $\Phi^{(A)}$ of $(G^A, T^A)$ is either equal to the set of non-multipliable roots or the set of non-divisible roots in the image of the root system $\Phi$ of $(G, T)$ in $X(T^A)$. In particular, since $A$ permutes $\Delta$ it follows that the image of $\Delta$ in $X^*(T^A)$ spans the latter. Thus $\bZ\Phi^{(A)}$ is a free $\bZ$-module of rank $|\Delta/A|$ by the previous paragraph, and thus $G^A$ is semisimple.\smallskip

Now \cite[Prop.\ 3.5]{Haines2} shows that if $\alpha \in \Phi$ and $G$ is not of type $\mathrm{A}_{2m}$, then the dual coroot of $\overline{\alpha}$ is equal to $\sum_{\beta \in A\cdot\alpha} \beta^\vee$. Thus the first paragraph of this proof shows that the natural map $\gamma: \prod_{\overline{\alpha} \in \Delta^{(A)}} \bG_m \to T^A$ given by $\gamma(x) = \prod_{\overline{\alpha} \in \Delta^{(A)}} \overline{\alpha}^\vee(x)$ is an isomorphism, and hence $G^A$ is simply connected by definition. If $G$ is of type $\mathrm{A}_{2m}$, then the same reference shows that if $\alpha \in \Phi$ then the dual coroot of $\overline{\alpha}$ is equal either to $\sum_{\beta \in A\cdot\alpha} \beta^\vee$ or to $2\sum_{\beta \in A\cdot\alpha} \beta^\vee$. In fact, there is precisely one $\alpha \in \Delta$ with the latter property. Thus the order of $\ker \gamma$ is equal to $2$ and in particularly any prime $p \neq 2$ does not divide $|\pi_1(\sD(G))|$.\smallskip

It remains to show that if $p$ is good for $G$, then $p$ is good for $G^A$. For this, note that $A$ is solvable, so by forming successive centralizers we may and do assume that $A$ is generated by a single automorphism $\lambda$. We will resort to checking cases in the following lemma.
\end{proof}

\begin{lemma}\label{lemma:simple-preservation-of-properties}
Let $k$ be a field, and let $G$ be a connected simple simply connected $k$-group. Suppose that either $p \neq 2$ or $G$ is not of type $\mathrm{A}_{2m}$. Let $\lambda$ be an automorphism of $G$ preserving a pinning. Then $Z_G(\lambda)$ is connected and simple, and in fact the following claims hold.
\begin{enumerate}
    \item If $G$ is of type $\mathrm{A}_{2m+1}$ ($m \geq 1$) and $\lambda$ is of order $2$, then $Z_G(\lambda)$ is simply connected of type $\mathrm{C}_{m+1}$.
    \item If $G$ is of type $\mathrm{A}_{2m}$ ($m \geq 1$), $p \neq 2$, and $\lambda$ is of order $2$, then $Z_G(\lambda)$ is adjoint of type $\mathrm{B}_m$.
    \item If $G$ is of type $\mathrm{D}_n$ ($n \geq 4$) and $\lambda$ is of order $2$, then $Z_G(\lambda)$ is simply connected of type $\mathrm{B}_{n-1}$.
    \item If $G$ is of type $\mathrm{D}_4$ and $\lambda$ is of order $3$, then $Z_G(\lambda)$ is simply connected of type $\mathrm{G}_2$.
    \item If $G$ is of type $\mathrm{E}_6$ and $\lambda$ is of order $2$, then $Z_G(\lambda)$ is simply connected of type $\mathrm{F}_4$.
\end{enumerate}
\end{lemma}

\begin{proof}
We may and do assume that $k$ is algebraically closed, and we choose a maximal $k$-torus $T$ of $G$. We have already seen that $G^A$ is a connected simple simply connected group with maximal torus $T^A$. By the proof of Proposition~\ref{prop:preservation-of-properties}, it suffices to show that the image of the root system $\Phi$ of $(G, T)$ to $X(T^A)$ is of the claimed type in each case. We briefly indicate the calculations in the first case, leaving the rest to the reader.\smallskip

If $G$ is of type $\mathrm{A}_{2m+1}$, then we can order $\Delta = \{\alpha_1, \dots, \alpha_{2m+1}\}$ so that $\lambda(\alpha_i) = \alpha_{2m+2-i}$, and every root in $\Phi$ is of the form $\alpha_i + \dots + \alpha_j$ for $1 \leq i \leq j \leq 2m+1$. Thus every root of $\Phi^{(A)}$ can be written uniquely either in one of the following forms.
\begin{itemize}
    \item $\overline{\alpha_i} + \dots + \overline{\alpha_j}$ for $1 \leq i \leq j \leq m+1$
    \item $(\overline{\alpha_i} + \dots + \overline{\alpha_{m+1}}) + (\overline{\alpha_j} + \dots + \overline{\alpha_m})$ for $1 \leq i \leq m+1$ and $1 \leq j \leq m$.
\end{itemize}
This is precisely the description of the root system of type $\mathrm{C}_{m+1}$; see \cite[Chap.\ VI, \S 4]{Bourbaki}. The other types are dealt with using entirely similar considerations, using the calculations in \cite[Chap.\ VI, \S 4]{Bourbaki}.
\end{proof}

\begin{lemma}\label{lemma:reg-unip-fixed}
Let $k$ be a field, and let $G$ be a connected reductive $k$-group. Suppose either that $\chara k \neq 2$ or $G$ has no simple factors of type $\mathrm{A}_{2m}$ ($m \geq 1$). Let $A$ be a finite group of automorphisms of $G$ preserving a common pinning. If $u \in (G^A)^0(k)$ is regular unipotent as an element of $(G^A)^0_{\rm{red}}$, then $u$ is regular in $G(k)$. Similarly, if $X \in \Lie G^A$ is regular nilpotent, then $X$ is regular in $\Lie G$.
\end{lemma}

\begin{proof}
We may and do assume that $k$ is algebraically closed. We consider only the unipotent case, the nilpotent case being entirely similar. The nilpotent version of this result can be extracted from \cite[Prop.\ 4.1]{Haines}, but we will give the argument here for convenience of the reader. By conjugacy of regular unipotent elements, it is enough to show that there is a single regular unipotent element $u \in G^A(k)$ which is regular in $G(k)$. \smallskip

Let $(B, T, \{X_{\alpha}\}_{\alpha \in \Delta})$ be a pinning preserved by $A$, and let $U$ be the unipotent radical of $B$. Let $\Phi_0$ be the $A$-orbit of a simple root $\alpha \in \Delta$, so $\Phi_0 \subset \Delta$. Moreover, since $G$ has no simple factors of type $\mathrm{A}_{2m}$, it follows that any two distinct roots of $\Phi_0$ are orthogonal. For each $\alpha \in \Delta$, let $x_\alpha: \bG_a \to U_\alpha$ be the $T$-equivariant isomorphism with $\mathrm{d}x_\alpha(1) = X_\alpha$. If $\tau \in A$, then $\tau x_\alpha(u) = x_{\tau\alpha}(u)$, so it follows that $\prod_{\alpha \in \Phi_0} x_\alpha(1)$ lies in $U^A(k)$ (the order being unambiguous by the above-noted orthogonality). \smallskip

Finally, let $\Delta = \Phi_1 \sqcup \cdots \sqcup \Phi_n$ be the decomposition of $\Delta$ into $A$-orbits, and let $u_i = \prod_{\alpha \in \Phi_i} x_\alpha(1)$. Then $u \coloneqq \prod_{i=1}^n u_i$ lies in $U^A(k)$, and it is regular unipotent in $G(k)$ by \cite[Lem.\ 3.2]{SteinbergReg}. To show that $u$ is regular unipotent in $G^A(k)$, it suffices by \cite[Thm.\ 3.3]{SteinbergReg} to show that $u$ lies in a unique Borel $k$-subgroup of $G^A$. But if $B_0$ is a Borel $k$-subgroup of $G^A$ containing $u$, then there is a cocharacter $\phi_0: \bG_m \to G^A$ such that $B_0 = P_{G^A}(\phi_0)$ (in the notation of the dynamic method). Thus $P_0 = P_G(\phi_0)$ is a parabolic $k$-subgroup of $G$ containing $u$ in its unipotent radical (because the unipotent radical of $P_G(\phi_0)$ is $U_G(\phi_0)$ by \cite[Prop.\ 2.1.8(2)]{CGP} and $U_{G^A}(\phi_0) \subset U_G(\phi_0)$ by definition), and it follows from \cite[Lem.\ 3.1]{BMR} that $P_0$ is a Borel $k$-subgroup of $G$. Since $P_0$ contains $u$, it must equal $B$ by \cite[Lem.\ 3.2]{SteinbergReg}. Thus $B_0 = B \cap G^A$, proving uniqueness.
\end{proof}

\subsection{The general case}\label{subsection:general}

Having established some preliminary results on fixed points, we are almost in a position to prove the existence of Springer isomorphisms for split reductive group schemes over open subschemes of $\Spec \bZ$. In fact, we will prove something slightly stronger in Lemma~\ref{lemma:springer-iso-invariant}. In view of the arguments in Section~\ref{subsection:split}, the following lemma is of basic use.

\begin{lemma}\label{lemma:existence-of-sections}
Let $G$ be a split reductive $\bZ$-group scheme and let $R$ be a localization of $\bZ$ such that $|\pi_1(\sD(G))|$ and all bad primes of $G$ are invertible in $R$. There exist sections $u \in \sU(R)$ and $X \in \sN(R)$ such that
\begin{enumerate}
    \item $u$ is fiberwise regular unipotent,
    \item $X$ is fiberwise regular nilpotent,
    \item $X \in (\Lie Z_G(u))(R)$.
\end{enumerate}
Moreover, let $(B, T, \{X_\alpha\}_{\alpha \in \Delta})$ be a pinning of $G$ and suppose
\begin{itemize}
    \item $2$ is invertible in $R$ or $G$ has no simple factor of type $\mathrm{A}_n$ ($n \geq 2$), and
    \item $3$ is invertible in $R$ or $G$ has no simple factor of type $\mathrm{D}_4$.
\end{itemize}
Then we may choose $u$ and $X$ to be invariant under the action of $\Aut(G, B, T, \{X_\alpha\}_{\alpha \in \Delta})$.
\end{lemma}

\begin{proof}
All but the last statement of this lemma is established during the proof of \cite[Prop.\ 3.5]{Springer-isomorphism} in the case that $G$ is semisimple and simply connected. Since the unipotent and nilpotent schemes are unchanged by passing to derived groups or etale isogenies, this implies all but the last statement in general. \smallskip

For the remaining statements, we may again assume that $G$ is semisimple and simply connected, and by passing to simple factors we may and do assume that $G$ is simple. Now let $(B, T, \{X_\alpha\}_{\alpha \in \Delta})$ be a pinning of $G$, and let $A = \Aut(G, B, T, \{X_\alpha\}_{\alpha \in \Delta})$. Recall that $A$ is a finite constant $R$-group scheme. In \cite{ALRR} it is proved that (under our hypotheses) the scheme of fixed points $G^A$ is a reductive $R$-group scheme with split maximal $R$-torus $T^A$. Moreover, by Proposition~\ref{prop:preservation-of-properties}, for every $s \in S$ the fiber $(G^A)_s$ is a connected simple $k(s)$-group for which $\chara k(s)$ is good, and $|\pi_1(G^A)|$ is invertible in $R$. By the first paragraph, we may choose fiberwise regular $u \in \sU_{G^A}(R)$ and $X \in \sN_{G^A}(R)$ such that $X \in (\Lie Z_{G^A}(u))(R)$. \smallskip

To conclude, it is now enough to show that $u$ is fiberwise regular in $G$ and $X$ is fiberwise regular in $\fg$. For this, it suffices to pass to fibers over $R$, in which case the result follows from Lemma~\ref{lemma:reg-unip-fixed}.
\end{proof}

\begin{lemma}\label{lemma:springer-iso-invariant}
Let $G$ be a split reductive $\bZ$-group scheme and let $R$ be a localization of $\bZ$ such that $|\pi_1(\sD(G))|$ and all bad primes of $G$ are invertible in $R$. There is a $G$-equivariant $R$-isomorphism $\rho: \sU_{G_R} \to \sN_{G_R}$.\smallskip

Suppose moreover
\begin{itemize}
    \item $2$ is invertible in $R$ or $G$ has no simple factor of type $\mathrm{A}_n$ ($n \geq 2$), and
    \item $3$ is invertible in $R$ or $G$ has no simple factor of type $\mathrm{D}_4$.
\end{itemize}
Then we may take $\rho$ to be $\Aut_{G_R/R}$-equivariant.
\end{lemma}

\begin{proof}
As usual, we may and do assume that $G$ is semisimple and simply connected. In this case, $G$ is product of factors with absolutely simple fibers, and by passing to one of them we may and do assume that $G$ has absolutely simple fibers. Let $u \in \sU_G(R)$ and $X \in \sN_G(R)$ be fiberwise regular sections such that $Z_{G_R}(u) = Z_{G_R}(X)$ (guaranteed to exist by Lemmas~\ref{lemma:existence-of-sections} and \ref{lemma:centralizers-equal}). By Section~\ref{subsection:split}, there exists a unique Springer isomorphism $\rho: \sU_{G_R} \to \sN_{G_R}$ such that $\rho(u) = X$, giving the first claim.\smallskip

Now suppose that either $2$ is invertible in $R$ or $G$ has no simple factors of type $\mathrm{A}_n$ ($n \geq 2$) and either $3$ is invertible in $R$ or $G$ has no simple factors of type $\mathrm{D}_4$. Fix a pinning of $G$, and choose $u$ and $X$ as above to be invariant under automorphisms preserving the pinning. Note that these exist by Lemmas~\ref{lemma:existence-of-sections}. If $\alpha$ is an automorphism of $G$ preserving the chosen pinning, then $\alpha \circ \rho \circ \alpha^{-1}: \sU_{G_R} \to \sN_{G_R}$ is another Springer isomorphism sending $u$ to $X$. \smallskip

By \cite[Thm.\ 7.1.9]{Conrad}, if $S$ is any $R$-scheme and $\alpha$ is any $S$-automorphism of $G_S$, then we can write $\alpha = \alpha_0 \circ \alpha_1$, where $\alpha_1$ is inner and $\alpha_0$ preserves the chosen pinning. By $G$-equivariance we have $\alpha_1 \circ \rho \circ \alpha_1^{-1} = \rho$. Moreover, Zariski-locally on $S$, $\alpha_0$ is the base change of an $R$-automorphism of $G_R$ preserving the chosen pinning, so we also have $\alpha_0 \circ \rho \circ \alpha_0^{-1} = \rho$. Thus indeed $\rho$ is $\Aut_{G_R/R}$-equivariant.
\end{proof}

Lemma~\ref{lemma:springer-iso-invariant} is the key to proving Theorem~\ref{theorem:relative-springer-isomorphism} in almost every case, which we will do below. Completing the proof in types $\mathrm{A}_n$ ($n \geq 2$) and $\mathrm{D}_4$ will require different considerations, which we leave to Appendix~\ref{appendix}.

\begin{proof}[Proof of Theorem~\ref{theorem:relative-springer-isomorphism} in case~\ref{item:springer-small-cases}]
Let $S$ be a scheme, and let $G$ be a reductive $S$-group scheme as in the theorem statement. Passing to the universal cover of $\sD(G)$, we may and do assume that $G$ is semisimple and simply connected. By \cite[Exp.\ XXIV, 5.3 and Prop.\ 5.10(i)]{SGA3II}, there exists a finite etale cover $S'$ of $S$ and a simply connected semisimple $S'$-group scheme $G'$ with absolutely simple fibers such that $G$ is isomorphic to the Weil restriction $\mathrm{R}_{S'/S}(G')$. Passing from $S$ to $S'$ and from $G$ to $G'$, we may and do assume that $G$ has absolutely simple fibers. Since the root datum of $G$ is locally constant on $S$, we may further pass to direct summands of $S$ to assume that $G$ has constant root datum. As stated, we will suppose that for each $s \in S$, either $\chara k(s) \neq 2$ or $G_{\overline{s}}$ is not of type $\mathrm{A}_n$ for any $n \geq 2$, and that either $\chara k(s) \neq 3$ or $G_{\overline{s}}$ is not of type $\mathrm{D}_4$.\smallskip

There exists an etale cover $S' \to S$ such that $G_{S'}$ is split, so there is some split reductive $\bZ$-group scheme $\cG$ and an $S'$-isomorphism $f: G_{S'} \to \cG_{S'}$. By Lemma~\ref{lemma:springer-iso-invariant} there is an $\Aut_{\cG_{S}/S}$-equivariant $S$-isomorphism $\rho_0: \sU_{\cG_S} \to \sN_{\cG_S}$. This gives a Springer isomorphism $\rho': \sU_{G_{S'}} \to \sN_{G_{S'}}$ defined by $\rho' = f \circ \rho_{0, S'} \circ f^{-1}$ (where we abuse notation and use $f$ to also denote the isomorphism $\Lie G_{S'} \cong \Lie \cG_{S'}$). \smallskip

Now we show that $\rho'$ descends to a $G$-equivariant $S$-isomorphism $\rho: \sU_G \to \sN_G$. If $\mathrm{pr}_1, \mathrm{pr}_2: S' \times_S S' \to S'$ are the two projections, then by fpqc descent \cite[Exp.\ VIII, Thm.\ 5.2]{SGA1} it is necessary and sufficient that $\mathrm{pr}_1^*\rho' = \mathrm{pr}_2^*\rho'$. Since $\rho_0$ is defined over $S$, we have $\mathrm{pr}_1^*\rho_{0, S'} = \mathrm{pr}_2^*\rho_{0, S'}$, and thus
\[
\mathrm{pr}_1^*\rho' = \mathrm{pr}_2^*f \circ (\alpha \circ \rho_{0, S' \times_S S'} \circ \alpha^{-1}) \circ \mathrm{pr}_2^*f^{-1},
\]
where $\alpha = \mathrm{pr}_2^*f^{-1} \circ \mathrm{pr}_1^*f$ is an $S' \times_S S'$-automorphism of $G_{S' \times_S S'}$. Since $\rho_0$ is $\Aut_{\cG_S/S}$-equivariant, it follows that $\alpha \circ \rho_{0, S' \times_S S'} \circ \alpha^{-1} = \rho_{0, S' \times_S S'}$. Thus the displayed equation above shows
\[
\mathrm{pr}_1^*\rho' = \mathrm{pr}_2^*\rho',
\]
as desired.
\end{proof}

\begin{remark}\label{remark:aut-equivariant}
Let $k$ be an algebraically closed field of characteristic $p \geq 0$, and let $G$ be a simply connected simple $k$-group with a fixed pinning. Let $H \coloneqq G^A$, where $A$ is the group of automorphisms of $G$ preserving the given pinning. Fix a regular unipotent element $u \in H(k)$. If $\rho: \sU_G \to \sN_G$ is an $\Aut_{G/k}$-equivariant $k$-isomorphism, the fact that $\rho$ is determined by $\rho(u)$ shows that $\rho(u)$ also lies in $\Lie H$. However, if $p > 0$ is good for $G$ but not for $H$, then this is not possible: we will show in the next paragraph that $\Lie Z_H(u)$ does not contain a regular nilpotent element. It follows that no such $\rho$ can exist in this case. Thus for outer forms of $\SL_{n+1}$ in characteristic $2$ or certain outer forms of $\Spin(8)$ in characteristic $3$, a different argument is needed.\smallskip

To see the claim above, we work slightly more generally: let $H$ be a simply connected simple $k$-group, and suppose that $p > 0$ is a bad prime which is not torsion. Note that the $H$ in the previous paragraph is of this form by Lemma~\ref{lemma:simple-preservation-of-properties}. Let $u_0 \in H(k)$ be regular unipotent, and suppose for the sake of contradiction that there exists a regular nilpotent $X_0 \in \Lie H$. Let $R$ be a complete DVR with residue field $k$ and generic characteristic $0$, and let $\sH$ be a simply connected simple $R$-group scheme with special fiber $H$. Lift $u$ to a fiberwise regular unipotent section $u_0 \in \sH(R)$. By Theorem~\ref{theorem:flat-centralizer}, we may lift $X$ to $X_0 \in \Lie (Z_{\sH}(u_0))(R)$ which is fiberwise regular nilpotent. Now Theorem~\ref{theorem:flat-centralizer} and \cite[Thm.\ 4.2.8]{Bouthier-Cesnavicius} show that $Z_{\sH}(u_0)$ and $Z_{\sH}(X_0)$ are flat closed $R$-subgroup schemes of $\sH$, and they have equal generic fibers by Lemma~\ref{lemma:centralizers-equal}. Thus flatness shows $Z_{\sH}(u_0) = Z_{\sH}(X_0)$ and in particular $Z_H(u) = Z_H(X)$. But $Z_H(u)/Z(H)$ is smooth by Lemma~\ref{lemma:smooth-unipotent-centralizer}, while $Z_H(X)/Z(H)$ is not smooth by the calculations in \cite[Thm.\ 2.6]{Springer}.
\end{remark}

\subsection{Kawanaka isomorphisms}\label{subsection:kawanaka}

The notion of Kawanaka isomorphism was introduced in \cite[Def.\ 4.1]{Jay-GGGR}, and we recall the definition here. If $k$ is a field, $G$ is a connected reductive $k$-group, and $\lambda: \bG_m \to G_{\overline{k}}$ is a geometric cocharacter, then there is a canonically associated parabolic subgroup $P(\lambda)$ of $G_{\overline{k}}$ with unipotent radical $U(\lambda)$. We write $\fu(\lambda) \coloneqq \Lie U(\lambda)$, and note that $\fu(\lambda)$ admits a decreasing filtration $(\fu(\lambda, i))_{i > 0}$, where $\fu(\lambda, i)$ is the direct sum of the eigenspaces of weight $\geq i$ for the action of $\lambda$ on $\fu(\lambda)$. Correspondingly, one obtains a smooth connected $\lambda$-equivariant subgroup $U(\lambda, i)$ of $U(\lambda)$ with Lie algebra $\fu(\lambda, i)$. \smallskip

Now one says that an isomorphism $\psi: U(\lambda) \to \fu(\lambda)$ is a \textit{Kawanaka isomorphism} provided that the following three conditions hold.
\begin{itemize}
    \item $\psi(U(\lambda, 2)) \subset \fu(\lambda, 2)$,
    \item $\psi(uv) - \psi(u) - \psi(v) \in \fu(\lambda, i + 1)$ for all $u, v \in U(\lambda, i)$ and $i \in \{1, 2\}$,
    \item $\psi([u, v]) - c_i[\psi(u), \psi(v)] \in \fu(\lambda, 2i + 1)$ for all $u, v \in U(\lambda, i)$ and $i \in \{1, 2\}$, where $c_i \in k^\times$ is a constant not depending on $u$ or $v$.
\end{itemize}

In \cite[Prop.\ 4.6]{Jay-GGGR}, it is shown that if $G$ is absolutely simple, $\chara k$ is good for $G$, and $\chara k \nmid |\pi_1(\sD(G))|$, then there exists a Springer isomorphism $\phi: \sU_G \to \sN_G$ which restricts to a Kawanaka isomorphism $U(\lambda) \to \fu(\lambda)$ for every geometric cocharacter $\lambda$. In Proposition~\ref{prop:kawanaka}, we show that in fact stronger conditions than the above hold for \textit{every} Springer isomorphism over an arbitrary base. First, we need the following mild strengthening of \cite[Thm.\ E(ii)]{McNinch-Testerman-springer}.

\begin{lemma}\label{lemma:mcninch-testerman}
Let $G$ be a connected reductive group over a scheme $S$ such that $\sD(G_s)$ is absolutely simple and $\chara k(s)$ is good for $G_s$ for every $s \in S$ and such that $|\pi_1(\sD(G))|$ is invertible on $S$. If $\rho: \sU_G \to \sN_G$ is a Springer isomorphism, then for every $S$-scheme $S'$ and any cocharacter $\lambda: \bG_{m, S'} \to G_{S'}$, $\rho$ restricts to an isomorphism $U_{G_{S'}}(\lambda) \to \fu_{G_{S'}}(\lambda)$.\smallskip

Moreover, there is some (unique) $a \in \bG_m(S)$ such that for any $S$-scheme $S'$ and any Borel $S$-subgroup $B$ of $G_{S'}$ with unipotent radical $U$, $\rho$ restricts to an isomorphism $U \to \Lie U$, and the differential $\mathrm{d}\rho: \Lie U \to \Lie U$ is equal to $a\cdot\id$.
\end{lemma}

\begin{proof}
By base change, it is enough to consider the case $S' = S$. Because we are claiming these properties for \textit{every} Springer isomorphism $\rho$ (and because we are claiming uniqueness of $a$), we may work locally, spread out, and pass to an fpqc cover of $S$ to assume $S = \Spec A$ for a complete noetherian local ring $A$ with algebraically closed residue field. By the Cohen structure theorem, we may write $A = A_0/I$ where $A_0$ is a complete regular local ring. Note that $G$ is split, so it is obtained by base change from a split reductive group scheme over $A_0$. Using openness of the regular locus and smoothness of regular centralizers modulo center (Theorem~\ref{theorem:flat-centralizer}), it can be shown that $\rho$ may be obtained by base change from a Springer isomorphism over $A_0$, so we may and do pass from $A$ to $A_0$ to assume that $A$ is a complete noetherian local integral domain with algebraically closed residue field.\smallskip

The fact that $\rho$ restricts to an isomorphism $U(\lambda) \to \fu(\lambda)$ follows directly from the argument of \cite[Rmk.\ 10]{Mcninch-optimal} (using the dynamic method over general base schemes). From now on, fix a Borel $A$-subgroup $B$ of $G$ with unipotent radical $U$. By \cite[Thm.\ E(ii)]{McNinch-Testerman-springer}\footnote{The proof of \cite[Thm.\ E(ii)]{McNinch-Testerman} relies on \cite[(5.2.5)]{McNinch-Testerman}, which uses \cite[Prop.\ 5.2]{McNinch-relative-centralizer} to reduce to a theorem of Kostant. As pointed out in \cite[Rmk.\ 4.20]{Booher}, the proof of \cite[Prop.\ 5.2]{McNinch-relative-centralizer} is not correct. However, the part of \cite[Prop.\ 5.2]{McNinch-relative-centralizer} relevant to the argument of \cite{McNinch-Testerman} is contained in Theorem~\ref{theorem:flat-centralizer}.}, the generic map $\mathrm{d}\rho_\eta: \Lie U_\eta \to \Lie U_\eta$ is a scalar multiple of the identity. Consequently, the same holds of $\mathrm{d}\rho: \Lie U \to \Lie U$, as one sees by considering a matrix for $\mathrm{d}\rho$. It remains to show that $a$ is independent of the choice of $B$. For this, we may and do pass to the geometric generic fiber of $S$ to assume $S = \Spec k$ for an algebraically closed field $k$. \smallskip

By passing from $G$ to the universal cover of $\sD(G)$ (which does not affect the unipotent radical of any Borel), we may and do assume that $G$ is semisimple and simply connected. If $G$ is of type $\mathrm{A}_1$, i.e., $G = \SL_2$, then one can show that $\rho: \sU_G \to \sN_G$ is of the form
\[
\rho(u) = a'(u - 1)
\]
for some $a' \in \bG_m(S)$. In this case, clearly $a = a'$. Thus we may and do assume that $G$ is not of type $\mathrm{A}_1$, i.e., $G$ is of rank $\geq 2$.\smallskip

If $B'$ is any other Borel $A$-subgroup of $G$ with unipotent radical $U'$, then the theory surrounding the Bruhat decomposition shows that $B \cap B'$ contains a maximal $k$-torus $T$ of $G$. Since $G$ is of rank $\geq 2$, one can successively conjugate $B$ by simple reflections in $N_G(T)$ to find a sequence of Borel $A$-subgroups $B = B_0, B_1, \dots, B_n = B'$ containing $T$ with unipotent radicals $U = U_0, U_1, \dots, U_n = U'$ such that $U_i \cap U_{i+1}$ is a nontrivial smooth connected unipotent group for every $0 \leq i \leq n-1$. By the above, each restriction $\mathrm{d}\rho: \Lie U_i \to \Lie U_i$ is a scalar multiple of the identity, and the fact that $U_i \cap U_{i+1}$ is nontrivial for each $i$ shows that all of these scalar multiples are the same. Since $B'$ was arbitrary, $a$ is after all independent of $B$.
\end{proof}

\begin{prop}\label{prop:kawanaka}
Let $S$ be a scheme, and let $G$ be a reductive group scheme over $S$ such that $\sD(G_s)$ is absolutely simple and $\chara k(s)$ is good for $G_s$ for every $s \in S$ and such that $|\pi_1(\sD(G))|$ is invertible on $S$. If $\rho: \sU_G \to \sN_G$ is a Springer isomorphism, there exists some (unique) $c \in \bG_m(S)$ such that for every $S$-scheme $S'$, cocharacter $\lambda: \bG_m \to G_{S'}$, pair of integers $m, n \geq 1$, and $u \in U(\lambda, m)(S')$, $v \in U(\lambda, n)(S')$, we have
\begin{enumerate}
    \item\label{item:kawanaka-1} $\rho(u) \in \fu(\lambda, m)$,
    \item\label{item:kawanaka-2} $\rho(uv) - \rho(u) - \rho(v) \in \fu(\lambda, 2\min\{m, n\})$,
    \item\label{item:kawanaka-3} $\rho([u, v]) - c[\rho(u), \rho(v)] \in \fu(\lambda, m + n + \min\{m, n\})$.
\end{enumerate}
The restriction of $\rho$ to each $U(\lambda, i)$ induces an isomorphism $U(\lambda, i) \cong \fu(\lambda, i)$.
\end{prop}

\begin{proof}
By base change, it is enough to consider the case $S' = S$. Just as in the proof of Lemma~\ref{lemma:mcninch-testerman}, we may and do assume that $S = \Spec A$, where $A$ is a complete noetherian local integral domain with algebraically closed residue field. Let the element $a \in \bG_m(S)$ be as in Lemma~\ref{lemma:mcninch-testerman}. Let $T$ be a (split) maximal $S$-torus of $G$ through which $\lambda$ factors. Let $\Phi$ be the root system of $(G, T)$, let $B$ be a Borel $S$-subgroup of $G$ containing $T$ with system of positive roots $\Phi^+$, and assume (as we may) that $U(\lambda)$ is contained in the unipotent radical $U$ of $B$. Note we have
\[
U(\lambda, i) = \prod_{\langle\alpha, \lambda\rangle \geq i} U_\alpha
\]
for the $T$-root groups $U_\alpha$ (product taken in a fixed order), and 
\[
\fu(\lambda, i) = \bigoplus_{\langle\alpha, \lambda\rangle \geq i} \fg_\alpha
\]
for the $T$-weight spaces $\fg_\alpha$. Fix $T$-equivariant $S$-isomorphisms $x_\alpha: \bG_a \to U_\alpha$ as usual, where $T$ acts by $\alpha$ on $\bG_a$. \smallskip

By Lemma~\ref{lemma:mcninch-testerman}, $\rho$ restricts to isomorphisms $U \to \Lie U$ and $U(\lambda) \to \fu(\lambda)$. Moreover, $\mathrm{d}\rho: \Lie U \to \Lie U$ is a scalar multiple of the identity, i.e., there is some (unique) $a \in \bG_m(S)$ such that $\mathrm{d}\rho|_{\Lie U} = aI$. Thus we may write
\[
\rho\left(\prod_{\langle\alpha, \lambda\rangle > 0} x_\alpha(u_\alpha)\right) = \sum_{\langle\alpha, \lambda\rangle > 0} (\mathrm{d}x_{\alpha})(au_\alpha + f_\alpha(\{u_\beta\})),
\]
where $f_\alpha(\{u_\beta\})$ is a polynomial in the $u_\beta$ whose $u_\alpha$-coefficient is $0$. \smallskip

We first claim that if $\prod_{\langle\beta, \lambda\rangle > 0} u_\beta^{n_\beta}$ is a monomial with nontrivial coefficient in $f_\alpha$, then we have $\sum_{\langle\beta, \lambda\rangle > 0} n_\beta \beta = \alpha$. To this end, let $t$ be a local section of $T$ and note
\[
\rho\left(\prod_{\langle\alpha, \lambda\rangle > 0} tx_\alpha(u_\alpha)t^{-1}\right) = \sum_{\langle\alpha, \lambda\rangle > 0} (\mathrm{d}x_{\alpha})(a\alpha(t)u_\alpha + f_\alpha(\{\beta(t)u_\beta\}))
\]
while on the other hand
\[
\Ad(t)\rho\left(\prod_{\langle\alpha, \lambda\rangle > 0} x_\alpha(u_\alpha)\right) = \sum_{\langle\alpha, \lambda\rangle > 0} (\mathrm{d}x_{\alpha})(\alpha(t)(au_\alpha + f_\alpha(\{u_\beta\})))
\]
for all local sections $u_\alpha$ of $\bG_a$. Since $\rho$ is $T$-equivariant, it follows that $f_\alpha(\{\beta(t)u_\beta\}) = \alpha(t)f_\alpha(\{u_\beta\})$ for all $u_\beta$, and the claim follows. \smallskip

By the previous paragraph, we have in particular that if $u_\beta$ occurs in the expression for $f_\alpha$ then $\beta < \alpha$. In particular, if $u \in U(\lambda, m)(S')$ for an $S$-scheme $S'$ then writing $u = \prod_{\langle\alpha, \lambda\rangle \geq m} x_\alpha(u_\alpha)$, we find $\rho(u) \in \fu(\lambda, m)$. Thus (\ref{item:kawanaka-1}) holds, and it follows similarly that $\rho$ restricts to an isomorphism $U(\lambda, i) \cong \fu(\lambda, i)$. Write moreover $v = \prod_{\langle\alpha, \lambda\rangle \geq n} x_\alpha(v_\alpha)$, so by the Chevalley commutation relations \cite[Prop.\ 5.1.14]{Conrad} we have
\[
uv \in \left(\prod_{\langle\alpha, \lambda\rangle > 0}x_\alpha(u_\alpha + v_\alpha)\right) \cdot U(\lambda, m + n).
\]
Moreover, using the previous paragraph we find
\[
\rho(uv) \in \sum_{\langle\alpha, \lambda\rangle > 0} (\mathrm{d}x_\alpha)(au_\alpha + av_\alpha) + \fu(\lambda, 2\min\{m, n\}).
\]
Using similar expressions for $\rho(u)$ and $\rho(v)$, we find
\[
\rho(uv) - \rho(u) - \rho(v) \in \fu(\lambda, 2\min\{m, n\}),
\]
proving (\ref{item:kawanaka-2}). \smallskip

Finally we prove (\ref{item:kawanaka-3}). Using the Chevalley commutation relations again, there are constants $C_{\alpha, \beta}$ such that
\[
[u, v] \in \prod_{\langle\alpha, \lambda\rangle > 0, \langle\beta, \lambda\rangle > 0} x_{\alpha+\beta}(C_{\alpha, \beta}u_\alpha v_\beta) \cdot U(\lambda, m + n + \min\{m, n\}).
\]
It follows then that
\[
\rho([u, v]) \in \sum_{\langle\alpha, \lambda\rangle > 0, \langle\beta, \lambda\rangle > 0} (\mathrm{d}x_{\alpha+\beta})(aC_{\alpha, \beta}u_\alpha v_\beta) + \fu(\lambda, m + n + \min\{m, n\}).
\]
Totally similar calculations using Chevalley's rule show
\[
[\rho(u), \rho(v)] \in \sum_{\langle\alpha, \lambda\rangle > 0, \langle\beta, \lambda\rangle > 0} (\mathrm{d}x_{\alpha+\beta})(a^2 C_{\alpha, \beta}u_\alpha v_\beta) + \fu(\lambda, m + n + \min\{m, n\}).
\]
Consequently
\[
\rho([u, v]) - a^{-1}[\rho(u), \rho(v)] \in \fu(\lambda, m + n + \min\{m, n\}),
\]
and so (\ref{item:kawanaka-3}) holds with $c = a^{-1}$.
\end{proof}

\subsection{Quasi-logarithms}\label{ss:quasi-log}

We recall the following definition from \cite[Def.\ 1.8.1]{Kazhdan-Varshavsky}.

\begin{definition}
    If $S$ is a scheme and $G$ is a reductive $S$-group scheme, then a \textit{quasi-logarithm} for $G$ is a $G$-equivariant $S$-morphism $\Phi\colon G \to \fg$ such that $\Phi(1) = 0$ and $\rm{d}\Phi\colon \fg \to \fg$ is the identity.
\end{definition}

\begin{lemma}\label{lemma:quasi-log-induces-springer}
    Let $S$ be a scheme, and let $G$ be a reductive $S$-group scheme such that $|\pi_1(\sD(G))|$ is invertible on $S$. If $\Phi\colon G \to \fg$ is a quasi-logarithm, then $\Phi|_{\sU_G}$ factors through an isomorphism $\sU_G \to \sN_G$.
\end{lemma}

\begin{proof}
    By Theorems~\ref{theorem:unipotent-scheme} and \ref{theorem:nilpotent-scheme}, the $S$-schemes $\sU_G$ and $\sN_G$ are flat with reduced fibers, and $\sU_G = \chi_G^{-1}(\chi_G(1))$ and $\sN_G = \chi_{\fg}^{-1}(\chi_{\fg}(0))$. The fact that $\Phi|_{\sU_G}$ factors through $\sN_G$ therefore follows from the fact the $\Phi$ induces a map $G/\!/G \to \fg/\!/G$ sending $\chi_G(1)$ to $\chi_{\fg}(0)$. To check that the factored map is an isomorphism, we appeal to the fibral isomorphism criterion and \cite[Cor.\ 9.3.4]{Bardsley-Richardson}, which proves the result on the level of varieties when $S$ is the spectrum of a field.
\end{proof}

It is shown in \cite[Lem.\ 1.8.12]{Kazhdan-Varshavsky} that if $S = \Spec \sO$ for the ring of integers $\sO$ in a non-archimedean local field of residue characteristic $p$ such that $p$ is \textit{very good} for $G$ and for each simple factor $G_i$ of $G$,
\begin{enumerate}
    \item either $p \neq 2$ or the quasi-split inner form of $G_i$ is split, and
    \item either $p \neq 3$ or $G_i$ is not a triality outer form of $\rm{D}_4$,
\end{enumerate}
then there exists a quasi-logarithm for $G$.\footnote{The formulation of this result in \cite[Lem.\ 1.8.12]{Kazhdan-Varshavsky} is slightly different, but the claims are equivalent.} Thus Theorem~\ref{theorem:relative-springer-isomorphism} follows from Lemma~\ref{lemma:quasi-log-induces-springer} in this case. In general, Theorem~\ref{theorem:relative-springer-isomorphism} deals with type A cases in small characteristic, and it works over a more general base scheme.

\appendix

\section{End of the proof of the main theorem}\label{appendix}

The point of this appendix is to prove Theorem~\ref{theorem:relative-springer-isomorphism} in case~\ref{item:springer-affine}. In view of what has been proven in Section~\ref{subsection:general}, we are reduced to the following setting: $S$ is an affine scheme and $G$ is a simply connected semisimple $S$-group scheme with absolutely simple fibers either of type $\mathrm{A}_n$ ($n \geq 2$) or $\mathrm{D}_4$, such that $\chara k(s)$ is good for $G_s$ for each $s \in S$, and we want to show that there exists a $G$-equivariant $S$-isomorphism $\sU_G \to \sN_G$.\smallskip

We begin with type $\mathrm{A}_n$ ($n \geq 2$).

\begin{lemma}\label{lemma:springer-type-a}
Let $S$ be a scheme and let $n \geq 1$ be an integer. Any $G$-equivariant $S$-morphism $\rho: \sU_{\SL_{n+1, S}} \to \sN_{\SL_{n+1, S}}$ is of the form $\rho(1 + e) = a_1 e + a_2 e^2 + \cdots + a_n e^n$ for some $a_i \in \Gamma(S, \sO_S)$. The morphism $\rho$ is an isomorphism if and only if $a_1 \in \Gamma(S, \sO_S^\times)$.
\end{lemma}

\begin{proof}
Let $G = \SL_{n+1, S}$, let $\sU$ and $\sN$ denote $\sU_G$ and $\sN_G$, respectively, and let $\sU_{\rm{reg}}$ and $\sN_{\rm{reg}}$ denote the open subschemes of fiberwise regular sections. There is a fiberwise regular unipotent section $u \in \sU_{\rm{reg}}(S)$ with $1$s on the diagonal and first superdiagonal, and $0$s everywhere else. Since $\sU_{\rm{reg}}$ is relatively schematically dense in $\sU$ and the orbit map $G/Z_G(u) \to \sU_{\rm{reg}}$ is an isomorphism by Lemma~\ref{lemma:regular-orbit-spaces}, any $G$-equivariant $S$-morphism $\rho: \sU \to \sN$ is determined by its value on $u$. By $G$-equivariance of $\rho$, we see that $\rho(u)$ must lie in $(\Lie Z_G(u))(S)$.\smallskip

It is a straightforward exercise to show that for any $S$-scheme $S'$, $Z_G(u)$ consists of those matrices in $G(S')$ such that the diagonal and each superdiagonal is constant. Consequently, $(\Lie Z_G(u))(S)$ consists of those upper-triangular matrices in $\mathfrak{sl}_{n+1}(S)$ such that the diagonal and each superdiagonal is constant. Since $\rho(u)$ must be nilpotent, it follows that it must be of the form $a_1(u-1) + a_2(u-1)^2 + \cdots + a_n(u-1)^n$ for some $a_i \in \Gamma(S, \sO_S^\times)$. The $S$-morphism $\rho'$ given by $\rho'(1+e) = a_1e + \dots + a_ne^n$ is clearly $G$-equivariant, so indeed any such $\rho$ must be of the form given in the lemma statement.\smallskip

Finally, $a_1(u-1) + a_2(u-1)^2 + \cdots + a_n(u-1)^n$ is fiberwise regular if and only if $a_1 \in \Gamma(S, \sO_S^\times)$, giving the final claim.
\end{proof}

Suppose now that $G$ is of type $\mathrm{A}_n$ for some $n \geq 2$. Since every reductive group scheme has a quasi-split inner form by \cite[Exp.\ XXIV, Cor.\ 3.12]{SGA3III}, the same twisting argument used in Section~\ref{subsection:general} allows us to assume that $G$ is quasi-split in the sense of \cite[Exp.\ XXIV, 3.9]{SGA3III}. In particular, there is a degree $2$ surjective finite etale morphism $S' \to S$ such that $G_{S'}$ is split. Being degree $2$, $S' \to S$ is automatically Galois, say with nontrivial $S$-automorphism $\sigma$.\smallskip

Let $\vp: \SL_{n+1, S'} \to G_{S'}$ be an $S'$-isomorphism sending the standard upper-triangular pinning of $\SL_{n+1, S'}$ to some pinning of $G_{S'}$ induced by a quasi-pinning of $G$ (see \cite[Exp.\ XXIV, 3.9]{SGA3III}). Consider the $S'$-automorphism
\[
\psi': \SL_{n+1, S'} \xrightarrow[]{1 \times \sigma} \SL_{n+1, S'} \xrightarrow[]{\vp} G_{S'} \xrightarrow[]{1 \times \sigma} G_{S'} \xrightarrow[]{\vp^{-1}} \SL_{n+1, S'}
\]
of $\SL_{n+1, S'}$. Note that $\psi'$ preserves a pinning of $\SL_{n+1, S'}$ defined over $S$, so it descends to an $S$-morphism $\psi: \SL_{n+1, S} \to \SL_{n+1, S}$. Write $S = S_1 \sqcup S_2$ for open subschemes $S_1$ and $S_2$, where $\psi|_{\SL_{n+1, S_1}}$ is the identity and $\psi|_{\SL_{n+1, S_2}}$ is the non-trivial automorphism preserving the upper-triangular pinning. We may pass separately to $S_1$ and $S_2$ to consider either the case $S = S_1$ or $S = S_2$.\smallskip

If $S = S_1$, then $G$ is $S$-split: indeed, since $\psi'$ is the identity we see that $\vp$ is ($1 \times \sigma$)-invariant by definition, so $\vp$ descends to an $S$-isomorphism $\SL_{n+1, S} \to G$. But now the Existence and Isomorphism Theorems couple with Lemma~\ref{lemma:springer-iso-invariant} to show the existence of a $G$-equivariant $S$-isomorphism $\sU_G \to \sN_G$, as desired.\smallskip

Finally, suppose $S = S_2$, so $\psi(g) = w(g^\top)^{-1}w^{-1}$, where $w$ is the anti-diagonal matrix with alternating entries $1, -1, 1, -1, \dots$ beginning in the upper right corner. Let $\rho': \sU_{\SL_{n+1, S'}} \to \sN_{\SL_{n+1, S'}}$ be a Springer isomorphism. Suppose $\rho' \circ \psi' \circ (1 \times \sigma) = \psi' \circ (1 \times \sigma) \circ \rho'$, where we use $\psi'$ also to denote the induced automorphism $X \mapsto -\Ad(w)X^\top$ of $\mathfrak{sl}_{n+1, S'}$. Unraveling the definitions gives
\[
(1 \times \sigma) \circ \vp \circ \rho' \circ \vp^{-1} \circ (1 \times \sigma) = \vp \circ \rho' \circ \vp^{-1},
\]
and it follows from Galois descent that $\vp \circ \rho' \circ \vp^{-1}: \sU_{G_{S'}} \to \sN_{G_{S'}}$ descends to a $G$-equivariant $S$-isomorphism $\rho: \sU_G \to \sN_G$.\smallskip

Every Springer isomorphism $\rho': \sU_{\SL_{n+1, S'}} \to \sN_{\SL_{n+1, S'}}$ is of the form
\begin{align}\label{equation:springer-1}
\rho'(1 + e) = a_1 e + a_2 e^2 + \cdots + a_n e^n
\end{align}
for some $a_i \in \Gamma(S', \sO_{S'})$ with $a_1 \in \Gamma(S', \sO_{S'}^\times)$. We are now reduced to finding such $a_i$ such that the $\rho'$ from (\ref{equation:springer-1}) satisfies $\rho' \circ \psi' \circ (1 \times \sigma) = \psi' \circ (1 \times \sigma) \circ \rho'$. Since $S = S_2$, we know what $\psi'$ is, and this equation signifies
\[
a_1 e + \cdots + a_n e^n = -\overline{a_1}\left(\sum_{i=1}^n (-1)^i e^i\right) - \cdots - \overline{a_n}\left(\sum_{i=1}^n (-1)^i e^i\right)^n,
\]
where $\overline{a} \coloneqq \sigma^*(a)$ for $a \in \Gamma(S', \sO_{S'})$. Expanding this shows that for all $1 \leq i \leq n$, we have
\begin{align}\label{equation:springer-2}
(-1)^i \sum_{j=1}^i \binom{i-1}{j-1} a_j = -\overline{a_i}.
\end{align}
We will choose the $a_i$ inductively to satisfy (\ref{equation:springer-2}). 
\smallskip

First, we may choose $a_1 \in \Gamma(S, \sO_S^\times)$ arbitrarily. In general, let $m \geq 1$ and suppose we have chosen $a_1, \dots, a_m \in \Gamma(S', \sO_{S'}^\times)$ satisfying (\ref{equation:springer-2}) for all $i \leq m$. We must then choose $a_{m+1} \in \Gamma(S', \sO_{S'}^\times)$ satisfying (\ref{equation:springer-2}) for $i = m+1$. Here we will finally use the assumption that $S = \Spec R$ is affine, so $S' = \Spec R'$ is also affine. If $m$ is odd, then this means $a_{m+1} + \overline{a_{m+1}} = f_m$, where $f_m \in R'$ is some element determined by $a_1, \dots, a_m$. The $R$-linear trace map $R' \to R$ is surjective, so this allows us to choose $a_{m+1}$. Similarly, if $m$ is even, then this means $a_{m+1} - \overline{a_{m+1}} = f_m$, where again $f_m \in R'$ is some element determined by $a_1, \dots, a_m$. The same argument which shows that the trace is surjective can be used to show that the $R$-linear map $R' \to R$ sending $a$ to $a - \overline{a}$ is surjective, so again this allows us to choose $a_{m+1}$ in this case, and we find that a Springer isomorphism exists.

\begin{remark}\label{remark:affine-necessary-type-a}
We note that affineness of $S$ is necessary in general in Theorem~\ref{theorem:relative-springer-isomorphism}: let $k$ be an algebraically closed field of characteristic $2$, and let $S' \to S$ be a degree $2$ Galois cover of smooth proper connected $k$-schemes. These exist in abundance; for instance, one may take $S$ to be an ordinary elliptic curve. Corresponding to this cover and any integer $n \geq 2$ is a quasi-split outer form $G$ of $\SL_{n+1}$ over $S$, and the proof above shows that a Springer isomorphism $\rho: \sU_G \to \sN_G$ over $S$ corresponds to a Springer isomorphism $\rho': \sU_{\SL_{n+1, S'}} \to \sN_{\SL_{n+1, S'}}$ given by $\rho'(1 + e) = a_1 e + \cdots a_n e^n$ for some $a_i \in \Gamma(S', \sO_{S'})$ with $a_1 \in \Gamma(S', \sO_{S'}^\times)$ satisfying (\ref{equation:springer-2}). Since $n \geq 2$, we find that $a_2 + \overline{a_2} = -\overline{a_1}$. Since $S'$ is a smooth proper connected $k$-scheme, we have $\Gamma(S', \sO_{S'}) = k$, and since $\chara k = 2$ we have $a + \overline{a} = 0$ for all $a \in k$ (as $\sigma^*$ is the identity on $k$). Since $a_1$ is required to be a unit, we see that elements $a_i$ with the desired properties do not exist.
\end{remark}

We move on now to the case that $G$ is of type $\mathrm{D}_4$. To proceed, we will need an analogue of Lemma~\ref{lemma:springer-type-a}, and to this end we must first describe the root groups of the split group $G_0$ of type $\mathrm{D}_4$ in some reasonable way. First, calculations in \cite[Chap.\ VI, \S 4]{Bourbaki} show that there is a pair $(B_0, T_0)$ of a Borel $S$-subgroup $B_0 \subset G_0$ and a maximal $S$-torus $T_0 \subset B_0$ such that the corresponding system of positive roots consists of roots $\alpha_1, \alpha_3, \alpha_4, \alpha_2, \alpha_1 + \alpha_2, \alpha_2 + \alpha_3, \alpha_2 + \alpha_4, \alpha_1 + \alpha_2 + \alpha_3, \alpha_1 + \alpha_2 + \alpha_4, \alpha_2 + \alpha_3 + \alpha_4, \alpha_1 + \alpha_2 + \alpha_3 + \alpha_4, \alpha_1 + 2\alpha_2 + \alpha_3 + \alpha_4$ (where we have grouped roots in the same $\Aut_{G_0/S}$-orbit together).\smallskip

Straightforward (but tedious) computations with $\SO(8)$ show that we can choose parameterizations $x_\alpha: \bG_a \to U_\alpha$ for each $\alpha \in \Phi^+$ satisfying the following commutation relations.
\begin{align*}
    (x_{\alpha_1}(u), x_{\alpha_2}(v)) &= x_{\alpha_1 + \alpha_2}(uv) \\
    (x_{\alpha_2}(u), x_{\alpha_3}(v)) &= x_{\alpha_2 + \alpha_3}(uv) \\
    (x_{\alpha_1}(u), x_{\alpha_2 + \alpha_3}(v)) &= x_{\alpha_1 + \alpha_2 + \alpha_3}(uv) \\
    (x_{\alpha_1 + \alpha_2}(u), x_{\alpha_3}(v)) &= x_{\alpha_1 + \alpha_2 + \alpha_3}(uv) \\
    (x_{\alpha_2}(u), x_{\alpha_4}(v)) &= x_{\alpha_2 + \alpha_4}(uv) \\
    (x_{\alpha_1}(u), x_{\alpha_2 + \alpha_4}(v)) &= x_{\alpha_1 + \alpha_2 + \alpha_4}(uv) \\
    (x_{\alpha_1}(u), x_{\alpha_2 + \alpha_3 + \alpha_4}(v)) &= x_{\alpha_1 + \alpha_2 + \alpha_3 + \alpha_4}(uv) \\
    (x_{\alpha_1 + \alpha_2}(u), x_{\alpha_4}(v)) &= x_{\alpha_1 + \alpha_2 + \alpha_4}(uv) \\
    (x_{\alpha_2}(u), x_{\alpha_1 + \alpha_2 + \alpha_3 + \alpha_4}(v)) &= x_{\alpha_1 + 2\alpha_2 + \alpha_3 + \alpha_4}(uv) \\
    (x_{\alpha_3}(u), x_{\alpha_1 + \alpha_2 + \alpha_4}(v)) &= x_{\alpha_1 + \alpha_2 + \alpha_3 + \alpha_4}(uv) \\
    (x_{\alpha_3}(u), x_{\alpha_2 + \alpha_4}(v)) &= x_{\alpha_2 + \alpha_3 + \alpha_4}(uv) \\
    (x_{\alpha_4}(u), x_{\alpha_2 + \alpha_3}(v)) &= x_{\alpha_2 + \alpha_3 + \alpha_4}(uv) \\
    (x_{\alpha_4}(u), x_{\alpha_1 + \alpha_2 + \alpha_3}(v)) &= x_{\alpha_1 + \alpha_2 + \alpha_3 + \alpha_4}(uv) \\
    (x_{\alpha_2 + \alpha_3 + \alpha_4}(u), x_{\alpha_1 + \alpha_2}(v) &= x_{\alpha_1 + 2\alpha_2 + \alpha_3 + \alpha_4}(uv) \\
    (x_{\alpha_2 + \alpha_3}(u), x_{\alpha_1 + \alpha_2 + \alpha_4}(v)) &= x_{\alpha_1 + 2\alpha_2 + \alpha_3 + \alpha_4}(uv) \\
    (x_{\alpha_2 + \alpha_4}(u), x_{\alpha_1 + \alpha_2 + \alpha_3}(v)) &= x_{\alpha_1 + 2\alpha_2 + \alpha_3 + \alpha_4}(uv)
\end{align*}
(The main point, of course, is to get the signs right on the right side.) For each $\alpha \in \Phi^+$, there is a corresponding element $X_\alpha \coloneqq (\mathrm{d}x_\alpha)(1)$ of $(\Lie G_0)(S)$. In particular, $(X_{\alpha_1}, X_{\alpha_3}, X_{\alpha_4}, X_{\alpha_2})$ gives rise to a pinning $(B_0, T_0, \{X_\alpha\})$ of $G_0$.\smallskip

Note that the element $u = x_{\alpha_1}(1)x_{\alpha_3}(1)x_{\alpha_4}(1)x_{\alpha_2}(1)$ of $G_0(S)$ is preserved by all automorphisms of $(G_0, B_0, T_0, \{X_{\alpha}\})$. Moreover, since $u$ is fiberwise regular \cite[Lem.\ 3.2]{SteinbergReg}, any $G$-equivariant $S$-morphism $\rho: \sU_{G_0} \to \sN_{G_0}$ is determined by $\rho(u)$. Conversely, $\rho(u)$ can be any element of $\sN_{G_0}(S)$ fixed by $\Ad(u)$. Using the considerations above, the following lemma is straightforward to check, and we will omit its verification.

\begin{lemma}\label{lemma:springer-type-d4}
Let $S$ be a scheme on which $2$ is invertible and let $G_0$ be the simply connected semisimple $S$-group scheme with simple fibers of type $\mathrm{D}_4$. For a Borel pair $(B_0, T_0)$ of $G_0$, fix parameterizations $x_\alpha: \bG_a \to U_\alpha$ of the positive root groups as above, let $X_\alpha = (\mathrm{d}x_\alpha)(1)$, and let $u = x_{\alpha_1}(1)x_{\alpha_3}(1)x_{\alpha_4}(1)x_{\alpha_2}(1)$. Then $(\Lie Z_{G_0}(u))(S)$ is the free $\Gamma(S, \sO_S)$-module with basis given by
\begin{align*}
    E_1 &= X_{\alpha_1} + X_{\alpha_3} + X_{\alpha_4} + X_{\alpha_2} + \frac{1}{2}X_{\alpha_1 + \alpha_2} - \frac{1}{2}X_{\alpha_2 + \alpha_3} - \frac{1}{2}X_{\alpha_2 + \alpha_4} - \frac{1}{2}X_{\alpha_2 + \alpha_3 + \alpha_4} \\
    E_2 &= X_{\alpha_1 + \alpha_2 + \alpha_3} - X_{\alpha_2 + \alpha_3 + \alpha_4} \\
    E_3 &= X_{\alpha_1 + \alpha_2 + \alpha_4} - X_{\alpha_2 + \alpha_3 + \alpha_4} \\
    E_4 &= X_{\alpha_1 + 2\alpha_2 + \alpha_3 + \alpha_4}
\end{align*}
Thus any $G_0$-equivariant $S$-morphism $\rho: \sU_{G_0} \to \sN_{G_0}$ satisfies $\rho(u) = a_1E_1 + a_2E_2 + a_3E_3 + a_4E_4$ for some $a_i \in \Gamma(S, \sO_S)$, and such $\rho$ is an isomorphism if and only if $a_1 \in \Gamma(S, \sO_S^\times)$.
\end{lemma}

Now let $\lambda: G_0 \to G_0$ be the order $2$ $S$-isomorphism preserving the chosen pinning, fixing $\alpha_1$, and sending $\alpha_3$ to $\alpha_4$. Further, let $\mu: G_0 \to G_0$ be the order $3$ $S$-isomorphism preserving the chosen pinning and sending $\alpha_1$ to $\alpha_3$ (and $\alpha_3$ to $\alpha_4$). Using the commutation relations listed above, a straightforward calculation shows
\begin{align}
    &\lambda(E_1) = E_1, \, \lambda(E_2) = E_3, \, \lambda(E_3) = E_2, \, \lambda(E_4) = E_4 \label{equation:lie-algebra-d4-1}\\
    &\mu(E_1) = E_1 - \frac{1}{2}E_3, \, \mu(E_2) = -E_3, \, \mu(E_3) = E_2 - E_3, \, \mu(E_4) = E_4. \label{equation:lie-algebra-d4-2}
\end{align}

Finally we are ready to prove the existence of a Springer isomorphism when $S$ is affine with $2$ invertible and $G$ is simple of type $\mathrm{D}_4$. As in the previous case, a twisting argument allows us to reduce to the case that $G$ is quasi-split. In this case, there is a finite Galois morphism $S' \to S$ (corresponding to an open and closed subscheme of the Dynkin scheme as in \cite[Exp.\ XXIV, 3.3]{SGA3III}) such that $G_{S'}$ is split. Let $G_0$ be the split form of $G$ over $S$, and fix the pinning $(B_0, T_0, \{X_\alpha\})$ as described above. Choose an $S'$-isomorphism $\vp: G_{S'} \to G_{0, S'}$ sending $(B_0, T_0, \{X_\alpha\})$ to a quasi-pinning of $G_{0, S'}$ defined over $S'$. For any $\sigma \in \Aut(S'/S)$, we get an $S'$-automorphism
\[
\psi'_{\sigma}: G_{0, S'} \xrightarrow[]{1 \times \sigma^{-1}} G_{0, S'} \xrightarrow[]{\vp} G_{S'} \xrightarrow[]{1 \times \sigma} G_{S'} \xrightarrow[]{\vp^{-1}} G_{0, S'}
\]
of $G_{0, S'}$. Since $\psi'_{\sigma}$ preserves a pinning of $G_0$ coming from $S$, it follows that $\psi'_{\sigma}$ descends to an $S$-automorphism $\psi_{\sigma}: G_0 \to G_0$. In particular, it follows from a simple calculation that $\psi_\sigma\psi_\tau = \psi_{\sigma\tau}$ for all $\sigma, \tau \in \Aut(S'/S)$. By Galois descent (as in the type $\mathrm{A}_n$ case), Springer isomorphisms $\rho: \sU_G \to \sN_G$ correspond to Springer isomorphisms $\rho': \sU_{G_{0, S'}} \to \sN_{G_{0, S'}}$ such that $\rho' \circ \psi'_{\sigma} \circ (1 \times \sigma) = \psi'_{\sigma} \circ (1 \times \sigma) \circ \rho'$ for all $\sigma \in \Aut(S'/S)$. By splitting up $S$ into subschemes which are open and closed as in the type $\mathrm{A}_n$ case, we may and do assume that $S' \to S$ is of constant degree, which we may assume is either $1$, $2$, $3$, or $6$ (corresponding to the size of the orbit of $\alpha_1$); in the degree $6$ case, $\Aut(S'/S) \cong S_3$. Moreover, we may and do assume that for each $\sigma \in \Aut(S'/S)$, the $S$-morphism $f: \Aut(S'/S) \to \Aut(G_0, B_0, T_0, \{X_\alpha\})$ between constant $S$-group schemes is constant.\smallskip

If $\rho'(u) = a_1E_1 + a_2E_2 + a_3E_3 + a_4E_4$ as in Lemma~\ref{lemma:springer-type-d4}, then because $u$ is stable under $\Aut(S'/S)$ and $\Aut(G_0, B_0, T_0, \{X_\alpha\})$, we require
\begin{align}\label{equation:springer-d4}
a_1 E_1 + a_2 E_2 + a_3 E_3 + a_4 E_4 = \sigma^*(a_1)f(E_1) + \sigma^*(a_2)f(E_2) + \sigma^*(a_3)f(E_3) + \sigma^*(a_4)f(E_4)
\end{align}
for all $\sigma \in \Aut(S'/S)$. We now split into cases for $f$. Note that if $f$ is trivial then $G$ is $S$-split, so the result follows from Lemma~\ref{lemma:springer-iso-invariant}.\smallskip

Now suppose that $\Aut(S'/S) = \bZ/2$ and $f$ is injective. Let $\tau \in \Aut(S'/S)$ be the nontrivial automorphism; by symmetry, we may and do assume that $f(\tau) = \lambda$ as above. Then in (\ref{equation:springer-d4}) we may apply (\ref{equation:lie-algebra-d4-1}) to obtain the relations $a_1 = \tau^*(a_1)$, $a_2 = \tau^*(a_3)$, $a_3 = \tau^*(a_2)$, and $a_4 = \tau^*(a_4)$. Thus we require $a_1, a_4 \in \Gamma(S, \sO_S)$ with $a_1 \in \Gamma(S, \sO_S^\times)$, and we may choose $a_2 \in \Gamma(S', \sO_{S'})$ arbitrarily and then let $a_3 = \tau^*(a_2)$.\smallskip

At this point we will finally use the assumption that $S$ is affine, say $S = \Spec R$ and $S' = \Spec R'$. Suppose that $\Aut(S'/S) = \bZ/3$ and $f$ is injective. Let $\sigma \in \Aut(S'/S)$ be a nontrivial automorphism such that $f(\sigma) = \mu$ as above. Then in (\ref{equation:springer-d4}) we may apply (\ref{equation:lie-algebra-d4-2}) to obtain the relations $a_1 = \sigma^*(a_1)$, $a_2 = \sigma^*(a_3)$, $a_3 = -\frac{1}{2}\sigma^*(a_1) - \sigma^*(a_2) - \sigma^*(a_3)$, and $a_4 = \sigma^*(a_4)$. Thus we require $a_1, a_4 \in R$ with $a_1 \in R^\times$, and since $a_2 = \sigma^*(a_3)$ we need only choose $a_3 \in R'$ such that $a_3 + \sigma^*(a_3) + (\sigma^2)^*(a_3) = -\frac{1}{2}a_1$. Since the $R$-linear trace map $R' \to R$ is surjective, it follows that we may choose $a_3$ satisfying these relations.\smallskip

Finally, suppose that $\Aut(S'/S) = S_3$ and $f$ is injective. Let $\sigma, \tau \in \Aut(S'/S)$ be automorphisms such that $f(\sigma) = \mu$ and $f(\tau) = \lambda$ as in (\ref{equation:lie-algebra-d4-1}) and (\ref{equation:lie-algebra-d4-2}), so in (\ref{equation:springer-d4}) we find
\begin{align*}
&a_1 = \tau^*(a_1) = \sigma^*(a_1) \\
&a_2 = \tau^*(a_3) = \sigma^*(a_3) \\
&a_3 = \tau^*(a_2) = -\frac{1}{2}\sigma^*(a_1) - \sigma^*(a_2) - \sigma^*(a_3) \\
&a_4 = \tau^*(a_4) = \sigma^*(a_4).
\end{align*}
Thus we need to choose $a_1, a_4 \in R$ with $a_1 \in R^\times$, and moreover $a_3$ must lie in $(R')^{\tau^*\sigma^*}$ satisfying $a_3 + \sigma^*(a_3) + (\sigma^2)^*(a_3) = -\frac{1}{2}a_1$. Again, we can find such $a_3$ by surjectivity of the trace map $(R')^{\tau^*\sigma^*} \to R$, and then we can let $a_2 = \sigma^*(a_3) = \tau^*(a_3)$.

\begin{remark}\label{remark:affine-necessary-type-d4}
As in Remark~\ref{remark:affine-necessary-type-a}, affineness is necessary in type $\mathrm{D}_4$. Indeed, if $S' \to S$ is either a $\bZ/3$-torsor or an $S_3$-torsor, then there is a corresponding quasi-split outer form $G$ over $S$ of $\Spin(8)$. If the trace map $\Gamma(S', \sO_{S'}) \to \Gamma(S, \sO_S)$ is trivial, then the proof above shows that there cannot be a Springer isomorphism $\sU_G \to \sN_G$. If $S'$ and $S$ are proper integral schemes over a field of characteristic $3$, then this occurs; on the other hand, if $3$ is invertible in $S$ then this cannot occur.
\end{remark}

\bibliography{bibliography}

@book {Bourbaki,
    AUTHOR = {Bourbaki, N.},
     TITLE = {\'{E}l\'{e}ments de math\'{e}matique. {F}asc. {XXXIV}. {G}roupes et
              alg\`ebres de {L}ie. {C}hapitre {IV}: {G}roupes de {C}oxeter et
              syst\`emes de {T}its. {C}hapitre {V}: {G}roupes engendr\'{e}s par
              des r\'{e}flexions. {C}hapitre {VI}: syst\`emes de racines},
    SERIES = {Actualit\'{e}s Scientifiques et Industrielles [Current Scientific
              and Industrial Topics], No. 1337},
 PUBLISHER = {Hermann, Paris},
      YEAR = {1968},
     PAGES = {288 pp. (loose errata)},
   MRCLASS = {22.50 (17.00)},
  MRNUMBER = {0240238},
MRREVIEWER = {G. B. Seligman},
}

@incollection {Conrad,
    AUTHOR = {Conrad, Brian},
     TITLE = {Reductive group schemes},
 BOOKTITLE = {Autour des sch\'{e}mas en groupes. {V}ol. {I}},
    SERIES = {Panor. Synth\`eses},
    VOLUME = {42/43},
     PAGES = {93--444},
 PUBLISHER = {Soc. Math. France, Paris},
      YEAR = {2014},
   MRCLASS = {14L15},
  MRNUMBER = {3362641},
}

@book {SGA3I,
    LABEL = {SGA3\textsubscript{I new}},
     TITLE = {Sch\'{e}mas en groupes ({SGA} 3). {T}ome {I}.
              {P}ropri\'{e}t\'{e}s g\'{e}n\'{e}rales des sch\'{e}mas en
              groupes},
    SERIES = {Documents Math\'{e}matiques (Paris) [Mathematical Documents
              (Paris)]},
    VOLUME = {7},
    EDITOR = {Gille, Philippe and Polo, Patrick},
   EDITION = {annotated},
      NOTE = {S\'{e}minaire de G\'{e}om\'{e}trie Alg\'{e}brique du Bois
              Marie 1962--64. [Algebraic Geometry Seminar of Bois Marie
              1962--64],
              A seminar directed by M. Demazure and A. Grothendieck with the
              collaboration of M. Artin, J.-E. Bertin, P. Gabriel, M.
              Raynaud and J-P. Serre},
 PUBLISHER = {Soci\'{e}t\'{e} Math\'{e}matique de France, Paris},
      YEAR = {2011},
     PAGES = {xxviii+610},
      ISBN = {978-2-85629-323-2},
   MRCLASS = {14L15},
  MRNUMBER = {2867621},
}

@book {SGA3II,
    LABEL = {SGA3\textsubscript{II}},
     TITLE = {Sch\'{e}mas en groupes. {II}: {G}roupes de type multiplicatif, et structure des sch\'{e}mas en groupes g\'{e}n\'{e}raux},
    SERIES = {Lecture Notes in Mathematics},
    VOLUME = {Vol. 152},
      NOTE = {S\'{e}minaire de G\'{e}om\'{e}trie Alg\'{e}brique du Bois
              Marie 1962/64 (SGA 3),
              Dirig\'{e} par M. Demazure et A. Grothendieck},
 PUBLISHER = {Springer-Verlag, Berlin-New York},
      YEAR = {1970},
     PAGES = {ix+654},
   MRCLASS = {14.50},
  MRNUMBER = {274459},
}

@book {SGA3III,
    LABEL = {SGA3\textsubscript{III new}},
     TITLE = {Sch\'{e}mas en groupes ({SGA} 3). {T}ome {III}. {S}tructure des
              sch\'{e}mas en groupes r\'{e}ductifs},
    SERIES = {Documents Math\'{e}matiques (Paris) [Mathematical Documents
              (Paris)]},
    VOLUME = {8},
    EDITOR = {Gille, Philippe and Polo, Patrick},
      NOTE = {S\'{e}minaire de G\'{e}om\'{e}trie Alg\'{e}brique du Bois Marie 1962--64.
              [Algebraic Geometry Seminar of Bois Marie 1962--64],
              A seminar directed by M. Demazure and A. Grothendieck with the
              collaboration of M. Artin, J.-E. Bertin, P. Gabriel, M.
              Raynaud and J-P. Serre,
              Revised and annotated edition of the 1970 French original},
 PUBLISHER = {Soci\'{e}t\'{e} Math\'{e}matique de France, Paris},
      YEAR = {2011},
     PAGES = {lvi+337},
      ISBN = {978-2-85629-324-9},
  MRCLASS = {14L15},
  MRNUMBER = {2867622},
}

@ARTICLE{EGA,
    LABEL = "EGA",
    AUTHOR = "Dieudonn{\'e}, Jean and Grothendieck, Alexander",
    TITLE = "\'{E}l\'ements de g\'eom\'etrie alg\'ebrique",
    JOURNAL = "Inst. Hautes \'Etudes Sci. Publ. Math.",
    VOLUME = "4, 8, 11, 17, 20, 24, 28, 32",
    YEAR = "1961--1967"
}

@article {Lou,
    AUTHOR = {Lou, Betty},
     TITLE = {The centralizer of a regular unipotent element in a
              semi-simple algebraic group},
   JOURNAL = {Bull. Amer. Math. Soc.},
  FJOURNAL = {Bulletin of the American Mathematical Society},
    VOLUME = {74},
      YEAR = {1968},
     PAGES = {1144--1146},
      ISSN = {0002-9904},
   MRCLASS = {14.50},
  MRNUMBER = {231826},
MRREVIEWER = {T. A. Springer},
       DOI = {10.1090/S0002-9904-1968-12085-3},
       URL = {https://doi.org/10.1090/S0002-9904-1968-12085-3},
}

@book {Matsumura,
    AUTHOR = {Matsumura, Hideyuki},
     TITLE = {Commutative ring theory},
    SERIES = {Cambridge Studies in Advanced Mathematics},
    VOLUME = {8},
   EDITION = {Second},
      NOTE = {Translated from the Japanese by M. Reid},
 PUBLISHER = {Cambridge University Press, Cambridge},
      YEAR = {1989},
     PAGES = {xiv+320},
      ISBN = {0-521-36764-6},
   MRCLASS = {13-01},
  MRNUMBER = {1011461},
}

@article {Springer,
    AUTHOR = {Springer, T. A.},
     TITLE = {Some arithmetical results on semi-simple {L}ie algebras},
   JOURNAL = {Inst. Hautes \'{E}tudes Sci. Publ. Math.},
  FJOURNAL = {Institut des Hautes \'{E}tudes Scientifiques. Publications
              Math\'{e}matiques},
    NUMBER = {30},
      YEAR = {1966},
     PAGES = {115--141},
      ISSN = {0073-8301}
}

@article {SteinbergReg,
    AUTHOR = {Steinberg, Robert},
     TITLE = {Regular elements of semisimple algebraic groups},
   JOURNAL = {Inst. Hautes \'{E}tudes Sci. Publ. Math.},
  FJOURNAL = {Institut des Hautes \'{E}tudes Scientifiques. Publications
              Math\'{e}matiques},
    NUMBER = {25},
      YEAR = {1965},
     PAGES = {49--80},
      ISSN = {0073-8301}
}

@book {SteinbergEndomorphisms,
    AUTHOR = {Steinberg, Robert},
     TITLE = {Endomorphisms of linear algebraic groups},
    SERIES = {Memoirs of the American Mathematical Society, No. 80},
 PUBLISHER = {American Mathematical Society, Providence, R.I.},
      YEAR = {1968},
     PAGES = {108},
   MRCLASS = {14.50 (22.00)},
  MRNUMBER = {0230728},
MRREVIEWER = {E. Abe},
}

@article {SteinbergTorsion,
    AUTHOR = {Steinberg, Robert},
     TITLE = {Torsion in reductive groups},
   JOURNAL = {Advances in Math.},
  FJOURNAL = {Advances in Mathematics},
    VOLUME = {15},
      YEAR = {1975},
     PAGES = {63--92},
      ISSN = {0001-8708},
   MRCLASS = {20G15},
  MRNUMBER = {354892},
MRREVIEWER = {S. I. Gel\cprime fand},
       DOI = {10.1016/0001-8708(75)90125-5},
       URL = {https://doi.org/10.1016/0001-8708(75)90125-5},
}

@book {Jantzen,
    AUTHOR = {Jantzen, Jens Carsten},
     TITLE = {Representations of algebraic groups},
    SERIES = {Mathematical Surveys and Monographs},
    VOLUME = {107},
   EDITION = {Second},
 PUBLISHER = {American Mathematical Society, Providence, RI},
      YEAR = {2003},
     PAGES = {xiv+576},
      ISBN = {0-8218-3527-0},
   MRCLASS = {20G05 (17B10)},
  MRNUMBER = {2015057},
}

@article {Booher,
    AUTHOR = {Booher, Jeremy},
     TITLE = {Minimally ramified deformations when {$\ell \neq p$}},
   JOURNAL = {Compos. Math.},
  FJOURNAL = {Compositio Mathematica},
    VOLUME = {155},
      YEAR = {2019},
    NUMBER = {1},
     PAGES = {1--37},
      ISSN = {0010-437X},
   MRCLASS = {11F80},
  MRNUMBER = {3875451},
MRREVIEWER = {\'{A}lvaro Lozano-Robledo},
       DOI = {10.1112/S0010437X18007546},
       URL = {https://doi.org/10.1112/S0010437X18007546},
}

@article {Springer-note,
    AUTHOR = {Springer, T. A.},
     TITLE = {A note on centralizers in semi-simple groups},
   JOURNAL = {Nederl. Akad. Wetensch. Proc. Ser. A 69=Indag. Math.},
    VOLUME = {28},
      YEAR = {1966},
     PAGES = {75--77},
   MRCLASS = {20.75},
  MRNUMBER = {0194423},
MRREVIEWER = {R. Steinberg},
}

@incollection {Springer-Steinberg,
    AUTHOR = {Springer, T. A. and Steinberg, R.},
     TITLE = {Conjugacy classes},
 BOOKTITLE = {Seminar on {A}lgebraic {G}roups and {R}elated {F}inite
              {G}roups ({T}he {I}nstitute for {A}dvanced {S}tudy,
              {P}rinceton, {N}.{J}., 1968/69)},
    SERIES = {Lecture Notes in Mathematics, Vol. 131},
     PAGES = {167--266},
 PUBLISHER = {Springer, Berlin},
      YEAR = {1970},
   MRCLASS = {14.50 (20.00)},
  MRNUMBER = {0268192},
MRREVIEWER = {N. Burgoyne},
}

@book {Borel,
    AUTHOR = {Borel, Armand},
     TITLE = {Linear algebraic groups},
    SERIES = {Graduate Texts in Mathematics},
    VOLUME = {126},
   EDITION = {Second},
 PUBLISHER = {Springer-Verlag, New York},
      YEAR = {1991},
     PAGES = {xii+288},
      ISBN = {0-387-97370-2},
   MRCLASS = {20-01 (20Gxx)},
  MRNUMBER = {1102012},
MRREVIEWER = {F. D. Veldkamp},
       DOI = {10.1007/978-1-4612-0941-6},
       URL = {https://doi.org/10.1007/978-1-4612-0941-6},
}

@article {Riche-universal-centralizer,
    AUTHOR = {Riche, Simon},
     TITLE = {Kostant section, universal centralizer, and a modular derived
              {S}atake equivalence},
   JOURNAL = {Math. Z.},
  FJOURNAL = {Mathematische Zeitschrift},
    VOLUME = {286},
      YEAR = {2017},
    NUMBER = {1-2},
     PAGES = {223--261},
      ISSN = {0025-5874},
   MRCLASS = {17B45 (14G17 14L30 17B08)},
  MRNUMBER = {3648498},
MRREVIEWER = {Paul D. Levy},
       DOI = {10.1007/s00209-016-1761-3},
       URL = {https://doi.org/10.1007/s00209-016-1761-3},
}

@misc{BRR,
  author       = {Bezrukavnikov, Roman and Riche, Simon and Rider, Laura},
  title        = {Modular affine Hecke category and regular unipotent centralizer, I},
  howpublished = {\url{https://arxiv.org/abs/2005.05583}},
  year         = {2020},
}

@incollection {Jantzen-nilpotent,
    AUTHOR = {Jantzen, Jens Carsten},
     TITLE = {Nilpotent orbits in representation theory},
 BOOKTITLE = {Lie theory},
    SERIES = {Progr. Math.},
    VOLUME = {228},
     PAGES = {1--211},
 PUBLISHER = {Birkh\"{a}user Boston, Boston, MA},
      YEAR = {2004},
   MRCLASS = {14L30 (17B20 17B35 20G15)},
  MRNUMBER = {2042689},
MRREVIEWER = {Dmitri I. Panyushev},
}

@incollection {Springer-isomorphism,
    AUTHOR = {Springer, T. A.},
     TITLE = {The unipotent variety of a semi-simple group},
 BOOKTITLE = {Algebraic {G}eometry ({I}nternat. {C}olloq., {T}ata {I}nst.
              {F}und. {R}es., {B}ombay, 1968)},
     PAGES = {373--391},
 PUBLISHER = {Oxford Univ. Press, London},
      YEAR = {1969},
   MRCLASS = {14.50 (14.18)},
  MRNUMBER = {0263830},
MRREVIEWER = {T. Kambayashi},
}

@article {Mcninch-optimal,
    AUTHOR = {McNinch, George J.},
     TITLE = {Optimal {${\rm SL}(2)$}-homomorphisms},
   JOURNAL = {Comment. Math. Helv.},
  FJOURNAL = {Commentarii Mathematici Helvetici. A Journal of the Swiss
              Mathematical Society},
    VOLUME = {80},
      YEAR = {2005},
    NUMBER = {2},
     PAGES = {391--426},
      ISSN = {0010-2571},
   MRCLASS = {20G15 (14L24 14L30 20E45)},
  MRNUMBER = {2142248},
MRREVIEWER = {Guy Rousseau},
       DOI = {10.4171/CMH/19},
       URL = {https://doi.org/10.4171/CMH/19},
}

@book {Slodowy,
    AUTHOR = {Slodowy, Peter},
     TITLE = {Simple singularities and simple algebraic groups},
    SERIES = {Lecture Notes in Mathematics},
    VOLUME = {815},
 PUBLISHER = {Springer, Berlin},
      YEAR = {1980},
     PAGES = {x+175},
      ISBN = {3-540-10026-1},
   MRCLASS = {14J17 (20G15 32B30)},
  MRNUMBER = {584445},
MRREVIEWER = {Jonathan M. Wahl},
}

@article {McNinch-relative-centralizer,
    AUTHOR = {McNinch, George J.},
     TITLE = {The centralizer of a nilpotent section},
   JOURNAL = {Nagoya Math. J.},
  FJOURNAL = {Nagoya Mathematical Journal},
    VOLUME = {190},
      YEAR = {2008},
     PAGES = {129--181},
      ISSN = {0027-7630},
   MRCLASS = {20G15 (17B45)},
  MRNUMBER = {2423832},
MRREVIEWER = {B. Sury},
       DOI = {10.1017/S0027763000009594},
       URL = {https://doi.org/10.1017/S0027763000009594},
}

@book {CGP,
    AUTHOR = {Conrad, Brian and Gabber, Ofer and Prasad, Gopal},
     TITLE = {Pseudo-reductive groups},
    SERIES = {New Mathematical Monographs},
    VOLUME = {26},
   EDITION = {Second},
 PUBLISHER = {Cambridge University Press, Cambridge},
      YEAR = {2015},
     PAGES = {xxiv+665},
      ISBN = {978-1-107-08723-1},
   MRCLASS = {20G15 (14L15)},
  MRNUMBER = {3362817},
       DOI = {10.1017/CBO9781316092439},
       URL = {https://doi.org/10.1017/CBO9781316092439},
}

@article {Demazure,
    AUTHOR = {Demazure, Michel},
     TITLE = {Invariants sym\'{e}triques entiers des groupes de {W}eyl et
              torsion},
   JOURNAL = {Invent. Math.},
  FJOURNAL = {Inventiones Mathematicae},
    VOLUME = {21},
      YEAR = {1973},
     PAGES = {287--301},
      ISSN = {0020-9910},
   MRCLASS = {14M15 (20G05 20H15 32M10)},
  MRNUMBER = {342522},
MRREVIEWER = {S. I. Gel\cprime fand},
       DOI = {10.1007/BF01418790},
       URL = {https://doi.org/10.1007/BF01418790},
}

@article {Artin-stacks,
    AUTHOR = {Artin, M.},
     TITLE = {Versal deformations and algebraic stacks},
   JOURNAL = {Invent. Math.},
  FJOURNAL = {Inventiones Mathematicae},
    VOLUME = {27},
      YEAR = {1974},
     PAGES = {165--189},
      ISSN = {0020-9910},
   MRCLASS = {14D15 (14B05)},
  MRNUMBER = {399094},
MRREVIEWER = {Michael Schlessinger},
       DOI = {10.1007/BF01390174},
       URL = {https://doi.org/10.1007/BF01390174},
}

@book {Knutson-algebraic-spaces,
    AUTHOR = {Knutson, Donald},
     TITLE = {Algebraic spaces},
    SERIES = {Lecture Notes in Mathematics, Vol. 203},
 PUBLISHER = {Springer-Verlag, Berlin-New York},
      YEAR = {1971},
     PAGES = {vi+261},
   MRCLASS = {14A20 (13J15 14F20)},
  MRNUMBER = {0302647},
MRREVIEWER = {H. Kurke},
}

@book {Humphreys,
    AUTHOR = {Humphreys, James E.},
     TITLE = {Conjugacy classes in semisimple algebraic groups},
    SERIES = {Mathematical Surveys and Monographs},
    VOLUME = {43},
 PUBLISHER = {American Mathematical Society, Providence, RI},
      YEAR = {1995},
     PAGES = {xviii+196},
      ISBN = {0-8218-0333-6},
   MRCLASS = {20G15 (22E10)},
  MRNUMBER = {1343976},
       DOI = {10.1090/surv/043},
       URL = {https://doi.org/10.1090/surv/043},
}

@article {Chaput-Romagny,
    AUTHOR = {Chaput, Pierre-Emmanuel and Romagny, Matthieu},
     TITLE = {On the adjoint quotient of {C}hevalley groups over arbitrary
              base schemes},
   JOURNAL = {J. Inst. Math. Jussieu},
  FJOURNAL = {Journal of the Institute of Mathematics of Jussieu. JIMJ.
              Journal de l'Institut de Math\'{e}matiques de Jussieu},
    VOLUME = {9},
      YEAR = {2010},
    NUMBER = {4},
     PAGES = {673--704},
      ISSN = {1474-7480},
   MRCLASS = {20G15 (13A50)},
  MRNUMBER = {2684257},
MRREVIEWER = {Vladimir L. Popov},
       DOI = {10.1017/S1474748010000125},
       URL = {https://doi.org/10.1017/S1474748010000125},
}

@book {SGA1,
     LABEL = {SGA1},
     TITLE = {Rev\^{e}tements \'{e}tales et groupe fondamental ({SGA} 1)},
    SERIES = {Documents Math\'{e}matiques (Paris) [Mathematical Documents
              (Paris)]},
    VOLUME = {3},
      NOTE = {S\'{e}minaire de g\'{e}om\'{e}trie alg\'{e}brique du Bois Marie 1960--61.
              [Algebraic Geometry Seminar of Bois Marie 1960-61],
              Directed by A. Grothendieck,
              With two papers by M. Raynaud,
              Updated and annotated reprint of the 1971 original [Lecture
              Notes in Math., 224, Springer, Berlin;  MR0354651 (50
              \#7129)]},
 PUBLISHER = {Soci\'{e}t\'{e} Math\'{e}matique de France, Paris},
      YEAR = {2003},
     PAGES = {xviii+327},
      ISBN = {2-85629-141-4},
   MRCLASS = {14E20 (14-06 14F35)},
  MRNUMBER = {2017446},
}

@misc{Bouthier-Cesnavicius,
  author       = {Bouthier, Alexis and {\v{C}}esnavi\v{c}ius, K\polhk{e}stutis},
  title        = {Torsors on loop groups and the {H}itchin fibration},
  howpublished = {\url{https://arxiv.org/abs/1908.07480}},
  year         = {2020},
}

@incollection {Lee-adjoint,
    AUTHOR = {Lee, Ting-Yu},
     TITLE = {Adjoint quotients of reductive groups},
 BOOKTITLE = {Autour des sch\'{e}mas en groupes. {V}ol. {III}},
    SERIES = {Panor. Synth\`eses},
    VOLUME = {47},
     PAGES = {131--145},
 PUBLISHER = {Soc. Math. France, Paris},
      YEAR = {2015},
   MRCLASS = {14L24 (14L15 14L30 20G05 20G35)},
  MRNUMBER = {3525843},
MRREVIEWER = {Gergely B\'{e}rczi},
}

@article {McNinch-Testerman,
    AUTHOR = {McNinch, George J. and Testerman, Donna M.},
     TITLE = {Central subalgebras of the centralizer of a nilpotent element},
   JOURNAL = {Proc. Amer. Math. Soc.},
  FJOURNAL = {Proceedings of the American Mathematical Society},
    VOLUME = {144},
      YEAR = {2016},
    NUMBER = {6},
     PAGES = {2383--2397},
      ISSN = {0002-9939},
   MRCLASS = {20G15 (14L15 17B08 17B45)},
  MRNUMBER = {3477055},
MRREVIEWER = {Jorge A. Vargas},
       DOI = {10.1090/proc/12942},
       URL = {https://doi.org/10.1090/proc/12942},
}

@article {Sobaje,
    AUTHOR = {Sobaje, Paul},
     TITLE = {Springer isomorphisms in characteristic {$p$}},
   JOURNAL = {Transform. Groups},
  FJOURNAL = {Transformation Groups},
    VOLUME = {20},
      YEAR = {2015},
    NUMBER = {4},
     PAGES = {1141--1153},
      ISSN = {1083-4362},
   MRCLASS = {14L35 (17B45 20G15)},
  MRNUMBER = {3416442},
MRREVIEWER = {Anthony Henderson},
       DOI = {10.1007/s00031-015-9320-2},
       URL = {https://doi.org/10.1007/s00031-015-9320-2},
}

@article {Jay-GGGR,
    AUTHOR = {Taylor, Jay},
     TITLE = {Generalized {G}elfand-{G}raev representations in small
              characteristics},
   JOURNAL = {Nagoya Math. J.},
  FJOURNAL = {Nagoya Mathematical Journal},
    VOLUME = {224},
      YEAR = {2016},
    NUMBER = {1},
     PAGES = {93--167},
      ISSN = {0027-7630},
   MRCLASS = {20C33 (20G40)},
  MRNUMBER = {3572751},
MRREVIEWER = {Bhama Srinivasan},
       DOI = {10.1017/nmj.2016.33},
       URL = {https://doi.org/10.1017/nmj.2016.33},
}

@article {McNinch-Testerman-springer,
    AUTHOR = {McNinch, George J. and Testerman, Donna M.},
     TITLE = {Nilpotent centralizers and {S}pringer isomorphisms},
   JOURNAL = {J. Pure Appl. Algebra},
  FJOURNAL = {Journal of Pure and Applied Algebra},
    VOLUME = {213},
      YEAR = {2009},
    NUMBER = {7},
     PAGES = {1346--1363},
      ISSN = {0022-4049},
   MRCLASS = {20G15},
  MRNUMBER = {2497582},
MRREVIEWER = {Evgenii L. Bashkirov},
       DOI = {10.1016/j.jpaa.2008.12.007},
       URL = {https://doi.org/10.1016/j.jpaa.2008.12.007},
}

@article {Haines,
    AUTHOR = {Haines, Thomas J.},
     TITLE = {On {S}atake parameters for representations with parahoric
              fixed vectors},
   JOURNAL = {Int. Math. Res. Not. IMRN},
  FJOURNAL = {International Mathematics Research Notices. IMRN},
      YEAR = {2015},
    NUMBER = {20},
     PAGES = {10367--10398},
      ISSN = {1073-7928},
   MRCLASS = {20G25 (11E95 20G20 22E50)},
  MRNUMBER = {3455870},
MRREVIEWER = {Maarten Sander Solleveld},
       DOI = {10.1093/imrn/rnu254},
       URL = {https://doi.org/10.1093/imrn/rnu254},
}

@article {BMR,
    AUTHOR = {Bate, Michael and Martin, Benjamin and R\"{o}hrle, Gerhard},
     TITLE = {Overgroups of regular unipotent elements in reductive groups},
   JOURNAL = {Forum Math. Sigma},
  FJOURNAL = {Forum of Mathematics. Sigma},
    VOLUME = {10},
      YEAR = {2022},
     PAGES = {Paper No. e13, 13},
   MRCLASS = {20G15 (14L30)},
  MRNUMBER = {4386351},
       DOI = {10.1017/fms.2021.82},
       URL = {https://doi.org/10.1017/fms.2021.82},
}

@misc{ALRR,
      title={Fixed points under pinning-preserving automorphisms of reductive group schemes}, 
      author={Pramod N. Achar and João Lourenço and Timo Richarz and Simon Riche},
      year={2022},
      eprint={2212.10182},
      archivePrefix={arXiv},
      primaryClass={math.AG},
      url={https://arxiv.org/abs/2212.10182}, 
}

@article {Borovoi,
    AUTHOR = {Borovoi, Mikhail},
     TITLE = {Abelian {G}alois cohomology of reductive groups},
   JOURNAL = {Mem. Amer. Math. Soc.},
  FJOURNAL = {Memoirs of the American Mathematical Society},
    VOLUME = {132},
      YEAR = {1998},
    NUMBER = {626},
     PAGES = {viii+50},
      ISSN = {0065-9266},
   MRCLASS = {20G10 (11E72 14E20 18G50)},
  MRNUMBER = {1401491},
MRREVIEWER = {James E. Humphreys},
       DOI = {10.1090/memo/0626},
       URL = {https://doi.org/10.1090/memo/0626},
}

@article {Bardsley-Richardson,
    AUTHOR = {Bardsley, Peter and Richardson, R. W.},
     TITLE = {\'{E}tale slices for algebraic transformation groups in
              characteristic {$p$}},
   JOURNAL = {Proc. London Math. Soc. (3)},
  FJOURNAL = {Proceedings of the London Mathematical Society. Third Series},
    VOLUME = {51},
      YEAR = {1985},
    NUMBER = {2},
     PAGES = {295--317},
      ISSN = {0024-6115},
   MRCLASS = {14L30},
  MRNUMBER = {794118},
MRREVIEWER = {Andy R. Magid},
       DOI = {10.1112/plms/s3-51.2.295},
       URL = {https://doi.org/10.1112/plms/s3-51.2.295},
}

@article {Haines2,
    AUTHOR = {Haines, Thomas J.},
     TITLE = {Dualities for root systems with automorphisms and applications
              to non-split groups},
   JOURNAL = {Represent. Theory},
  FJOURNAL = {Representation Theory. An Electronic Journal of the American
              Mathematical Society},
    VOLUME = {22},
      YEAR = {2018},
     PAGES = {1--26},
   MRCLASS = {20C08 (11F70 11G18 17B20 17B22 20G25 22E50)},
  MRNUMBER = {3772644},
MRREVIEWER = {Antoni Wawrzy\'{n}czyk},
       DOI = {10.1090/ert/512},
       URL = {https://doi.org/10.1090/ert/512},
}

@incollection {Kazhdan-Varshavsky,
    AUTHOR = {Kazhdan, David and Varshavsky, Yakov},
     TITLE = {Endoscopic decomposition of certain depth zero
              representations},
 BOOKTITLE = {Studies in {L}ie theory},
    SERIES = {Progr. Math.},
    VOLUME = {243},
     PAGES = {223--301},
 PUBLISHER = {Birkh\"{a}user Boston, Boston, MA},
      YEAR = {2006},
   MRCLASS = {22E50 (22E35)},
  MRNUMBER = {2214251},
       DOI = {10.1007/0-8176-4478-4\_10},
       URL = {https://doi.org/10.1007/0-8176-4478-4_10},
}

@misc{Cotner-non-etale,
      title={Central isogenies and conjugacy classes in reductive groups}, 
      author={Sean Cotner},
      year={2026},
      eprint={2607.00088v1},
      archivePrefix={arXiv},
      primaryClass={math.RT},
      url={https://arxiv.org/abs/2607.00088}, 
}

\end{document}